\documentclass[11pt]{amsart}
\pdfoutput=1
\usepackage{amsmath, amssymb, amsthm, amsbsy, bbm, amsfonts, enumitem, color, verbatim, array, floatflt, yfonts}
\usepackage[margin=1.2in]{geometry}
\usepackage[all]{xy}
\usepackage[english]{babel}

\usepackage{graphicx, color, overpic}
\usepackage[usenames,dvipsnames,svgnames,table]{xcolor}
\usepackage[pdfpagelabels]{hyperref}

\usepackage{thmtools}
\usepackage{thm-restate} 

\theoremstyle{plain}
\newtheorem{thm}{Theorem}[section]
\newtheorem*{thm*}{Theorem}
\newtheorem{prop}[thm]{Proposition}
\newtheorem*{prop*}{Proposition}
\newtheorem{lemma}[thm]{Lemma}
\newtheorem*{lemma*}{Lemma}
\newtheorem{corollary}[thm]{Corollary}

\theoremstyle{definition}
\newtheorem{definition}[thm]{Definition}
\newtheorem{example}[thm]{Example}

\theoremstyle{remark}
\newtheorem{remark}[thm]{Remark}

\newcommand{\cA}{\mathcal{A}}

\newcommand{\R}{\mathbb{R}}
\newcommand{\Z}{\mathbb{Z}}

\newcommand{\Q}{\mathbb{Q}}
\newcommand{\F}{\mathbb{F}}

\newcommand{\cO}{\mathcal O}

\newcommand{\cT}{\mathcal T}

\newcommand{\proj}{\operatorname{proj}}

\newcommand{\id}{\mathbbm{1}} 
\newcommand{\Id}{\operatorname{I}} 
\newcommand{\sW}{W} 
\newcommand{\aW}{\overline{W}} 
\newcommand{\eW}{\widetilde W} 
\newcommand{\Cf}{\mathcal{C}} 
\newcommand{\fa}{\mathbf{a}} 
\newcommand{\Cfs}{\widetilde{\mathcal{C}}} 
\newcommand{\Cw}{\mathcal{C}_w} 
\newcommand{\Cu}{\mathcal{C}_u} 

\newcommand{\rsa}{\rightsquigarrow} 

\newcommand{\Conv}{\operatorname{Conv}} 

\newcommand{\Range}{\operatorname{Im}}
\newcommand{\Ker}{\operatorname{Ker}}

\newcommand{\bas}{\operatorname{bas}}

\newcommand{\s}{\mathbf{s}}
\newcommand{\x}{\mathbf{x}}
\newcommand{\y}{\mathbf{y}}
\newcommand{\z}{\mathbf{z}}
\newcommand{\w}{\mathbf{w}}
\newcommand{\bb}{\mathbf{b}}

\numberwithin{equation}{subsection}

\definecolor{amethyst}{rgb}{0.6, 0.4, 0.8}
\definecolor{kellygreen}{rgb}{0.3, 0.73, 0.09}
\definecolor{americanrose}{rgb}{1.0, 0.01, 0.24}

\begin{document}

\hypersetup{pdfauthor={Milicevic, Schwer, Thomas},pdftitle={Affine Deligne--Lusztig varieties and folded galleries governed by chimneys}}

\title[ADLVs and folded galleries governed by chimneys]{Affine Deligne--Lusztig varieties and folded galleries governed by chimneys}

\author{Elizabeth Mili\'{c}evi\'{c}}
\address{Elizabeth Mili\'{c}evi\'{c}, Department of Mathematics \& Statistics, Haverford College, 370 Lancaster Avenue, Haverford, PA, USA}
\email{emilicevic@haverford.edu}

\author{Petra Schwer}
\address{Petra Schwer, Department of Mathematics, Universitätsplatz 2, Otto-von-Guericke University of Magdeburg, Germany}
\email{petra.schwer@ovgu.de}

\author{Anne Thomas}
\address{Anne Thomas, School of Mathematics \& Statistics, Carslaw Building F07,  University of Sydney NSW 2006, Australia}
\email{anne.thomas@sydney.edu.au}

\thanks{EM was supported by NSF Grant DMS 1600982 and Simons Collaboration Grant 318716. PS was supported by the DFG Project SCHW 1550/4-1. This research was also supported in part by ARC Grant DP180102437.}

\begin{abstract}
We characterize the nonemptiness and dimension problems for an affine Deligne--Lusztig variety $X_x(b)$ in the affine flag variety in terms of galleries that are positively folded with respect to a chimney. If the parabolic subgroup associated to the Newton point of $b$ has rank 1, we then prove nonemptiness for a certain class of Iwahori--Weyl group elements $x$ by explicitly constructing such galleries. 
\end{abstract}

\maketitle


\section{Introduction}\label{sec:Intro}

Deligne--Lusztig varieties have played a central role in the geometric representation theory of reductive groups over finite fields since they were introduced in \cite{DL}. Their affine analogs arose several decades later, and these affine Deligne--Lusztig varieties were formally developed for applications to $p$-divisible groups and isocrystals over perfect fields, special fibers of both Rapoport--Zink spaces and moduli of local shtukas, and bad reduction of Shimura varieties; see \cite{RapSatake, RapShimura}.  Foundational work on affine Deligne--Lusztig varieties has required first tackling an array of coarse algebro-geometric problems; e.g.\ establishing nonemptiness patterns, dimension formulas, connected components, and closure relations.

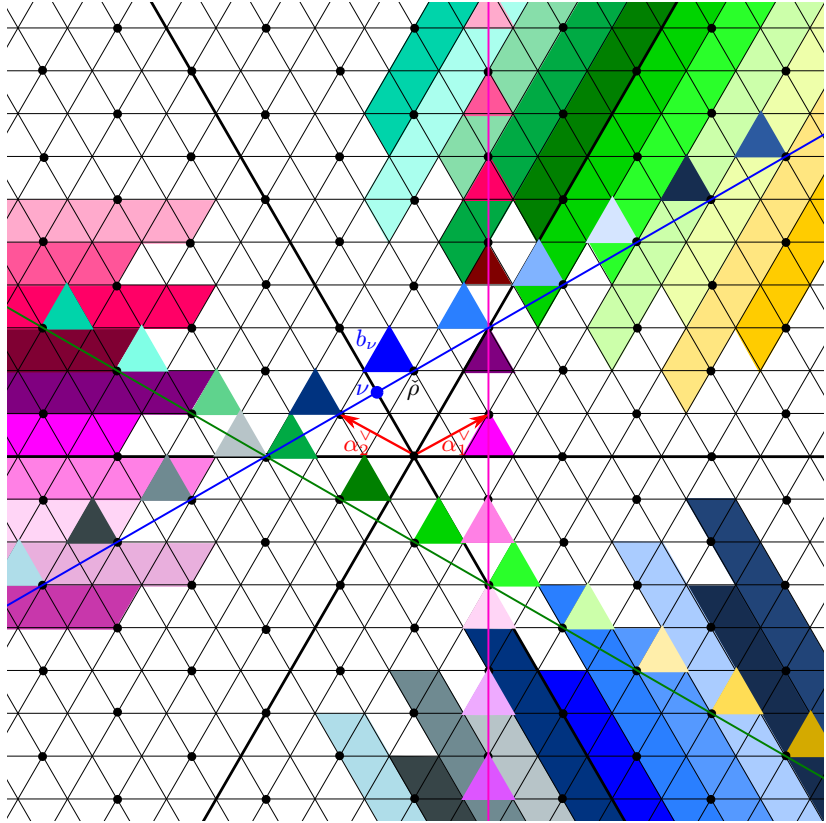
\begin{figure}[ht]
\centering
\begin{overpic}[width=0.7\textwidth]{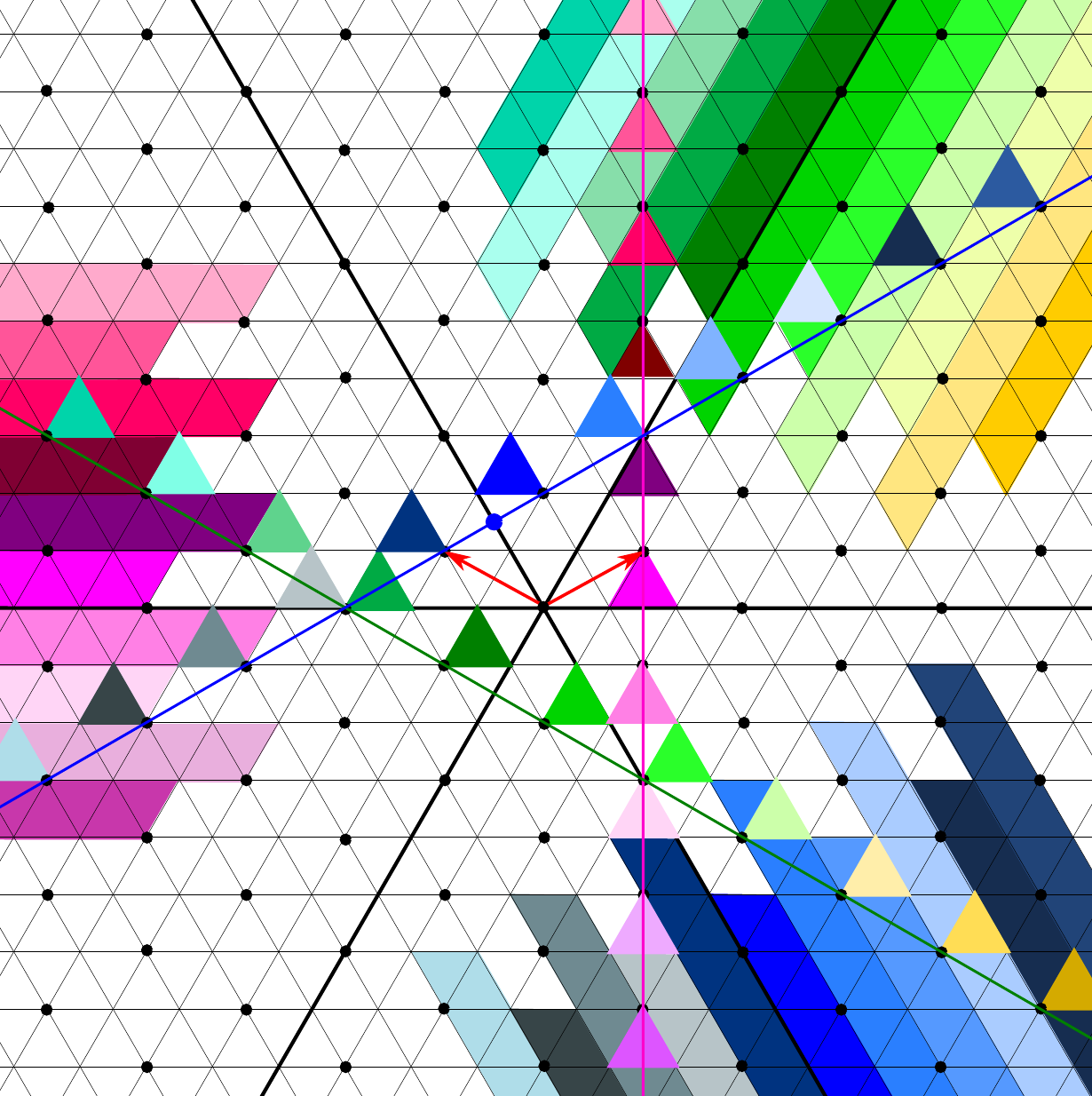}
\put(41,45.25){\footnotesize{\color{red}$\alpha_2^\vee$}}
\put(53,45.25){\footnotesize{\color{red}$\alpha_1^\vee$}}
\put(48.75,52){\footnotesize{\color{black}$\check{\rho}$}}
\put(42.5,58){\footnotesize{\color{blue}$b_\nu$}}
\put(42.5,52){\footnotesize{\color{blue}$\nu$}}
\end{overpic}
\caption{The conjugacy class of  the blue alcove $b_\nu = t^{\check{\rho}} s_1$, together with sectors representing the corresponding chimneys depicted in the same color.}
\label{fig:chimneyintro}
\end{figure}

Building on the pivotal conjectures in \cite{GHKRadlvs}, this paper applies the authors' recent joint work with Naqvi on chimney retractions \cite{MNST} to formulate the nonemptiness and dimension problems for affine Deligne--Lusztig varieties in the affine flag variety in terms of positively folded galleries in the associated Bruhat--Tits building. In particular, our analysis extends freely to the $p$-adic and mixed characteristic cases.  The main theorem then proves nonemptiness for certain affine Deligne--Lusztig varieties by explicitly constructing families of galleries which are positively folded with respect to the chimneys  illustrated in Figure \ref{fig:chimneyintro}.

\subsection{History of the problem}

In this section, we provide a brief overview of the nonemptiness problem for affine Deligne--Lusztig varieties. 
For more details, we refer the reader to the introduction of the authors' previous joint work \cite{MST1}, the arXiv version of He's 2018 ICM address \cite{HeICM}, and the references in these two works.

Classical Deligne--Lusztig varieties are indexed by an element of the finite Weyl group, and Lang's theorem guarantees that they are always nonempty.  In the setting of the affine flag variety, however, given a pair of indexing parameters, it remains an open problem to fully characterize the conditions under which the corresponding affine Deligne--Lusztig variety is nonempty.  In the affine Grassmannian, this question was solved relatively early and is phrased in terms of Mazur's inequality, which relates the coweight indexing the affine Schubert cell to the Newton point of the fixed group element serving as the second parameter; see \cite{KRFcrystals, Luc, Gashi}.  Although Mazur's inequality is a necessary condition for nonemptiness in the affine flag variety, it is far from sufficient.

Early progress in the affine flag variety focused on the basic case, in which the Newton point is as small as possible. G\"{o}rtz, Haines, Kottwitz, and Reuman first established an emptiness pattern \cite{GHKRadlvs}, and the nonemptiness problem was eventually settled by G\"ortz, He, and Nie \cite{GHN}. Beyond the basic case, major progress on the nonemptiness problem is largely shaped by the suite of formative conjectures in \cite{GHKR,GHKRadlvs}, based in large part on evidence from computer experiments.  These nonemptiness conjectures have still only been fully verified in type $\tilde{A}_2$, by the first author \cite{Be1} and again by Yang using different methods \cite{Yang}.  The first general nonemptiness results outside of the basic case were for affine Deligne--Lusztig varieties indexed by translation elements, in a lecture series by He \cite{HeCDM} and independently by the authors in \cite{MST1}, each under slightly different hypotheses and using very different methods.  The results on maximal Newton points by  Viehmann \cite{VieMaxNP} and the first author \cite{MilQBG}, as well as Viehmann's description of minimal Newton points \cite{VieMinNP}, also imply nonemptiness for certain affine Deligne--Lusztig varieties.

More recently, the first author's joint work with Viehmann on cordial elements establishes a condition under which the Newton stratum is saturated \cite{MilVie}, identifying families of affine Deligne--Lusztig varieties in the affine flag variety for which the nonemptiness pattern behaves as nicely as it does in the affine Grassmannian.  Certain cordial elements from \cite{MilVie} were then used by He as a new base case for extending the classical Deligne--Lusztig reduction methods to obtain the most general nonemptiness result to date \cite{HeCordial}.  The current work provides a distinct new proof when the Newton point is associated to a rank 1 parabolic subgroup, by adapting the relevant orbit intersections occurring in \cite{GHKRadlvs} to the context of retractions in the Bruhat--Tits building based at chimneys from the authors' joint work with Naqvi \cite{MNST}.

\subsection{Main results and key ideas}

In this section, we formally state the two main theorems in this paper, in addition to discussing key ideas from their proofs.

In \cite{MST1}, the authors use the folded galleries of Parkinson, Ram, and Schwer \cite{PRS} and the root operators of Gaussent and Littelmann \cite{GaussentLittelmann} to study the intersections of Iwahori and unipotent orbits occurring in the first work of G\"ortz, Haines, Kottwitz, and Reuman on affine Deligne--Lusztig varieties in the affine flag variety \cite{GHKR}. These particular orbits capture coarse algebro-geometric information about affine Deligne--Lusztig varieties indexed by translation elements, whose Newton points are typically regular dominant coweights.
 
For affine Deligne--Lusztig varieties indexed by Newton points which lie on a wall of the dominant Weyl chamber, G\"ortz, Haines, Kottwitz, and Reuman characterize both the nonemptiness and dimension problems in terms of intersections of Iwahori and $I_P$-orbits in the affine flag variety \cite{GHKRadlvs}; see Theorem~\ref{thm:GHKRThm} in the present paper.  These orbit intersections are interpreted as positively folded alcove walks with respect to orientations induced by the $P$-chimneys in the authors' joint work with Naqvi \cite{MNST}.  We recall the corresponding result of \cite{MNST} here as Theorem~\ref{thm:DoubleCosetIntersection}. Combining these two statements, we prove the following theorem in Section~\ref{sec:alcoveWalks}, using analogous arguments to those in \cite[Sec.~5]{MST1}. Most of the relevant terminology is formally developed in Section \ref{sec:Preliminaries}.

\begin{restatable}{thm}{ADLVChimneyRestate} 
\label{thm:ADLVChimneys} 
Let $[b] \in B(G)_P$ where $P$ is a standard spherical parabolic subgroup of $G(F)$.  Denote the Newton point of $[b]$ by $\nu$ and the standard representative by $b_\nu$. For any $x \in \eW$:
	\begin{enumerate}
		\item $X_x(b) \neq \emptyset$ if and only if for some $y \in \eW$, there exists a gallery $\gamma:\fa \rsa \bb_\nu^y$ of type $\vec{x}$ which is positively folded with respect to the $(P,y)$-chimney.
		\item If $X_x(b) \neq \emptyset$, then
		\[ \dim X_x(b) = \left( \sup_{y \in \eW} \dim(\gamma) \right) - \langle 2\rho, \nu \rangle, \]
		where we take the supremum over all galleries $\gamma$ as in (1).
	\end{enumerate}
\end{restatable}

\noindent  Here $B(G)_P$ denotes the set of $\sigma$-conjugacy classes with associated parabolic $P$ as defined in Section \ref{sec:NewtonMap}, together with the Newton and Kottwitz points. Standard representatives for $\sigma$-conjugacy classes are discussed in Section~\ref{sec:StandardReps-General}, and the definition of the affine Delgine--Lusztig variety $X_x(b)$ is found in Section \ref{sec:ADLVs}. All relevant terminology on positively folded galleries with respect to chimneys occurs in Sections \ref{sec:galleriesChimneys} and \ref{sec:dimGallery}.

Having established Theorem \ref{thm:ADLVChimneys}, one typically needs to explicitly construct certain folded galleries in order to prove nonemptiness statements. The folded galleries constructed in \cite{MST1} correspond to the special case originally considered in \cite{PRS}, in which $P=B$ so that the orientation induced by the $P$-chimney is periodic.  Moreover, the root operators of \cite{GaussentLittelmann} as applied to the alcove walks in \cite{MST1} are guaranteed to preserve positivity of folds only in the special case $P=B$, for which an entire Weyl chamber is a representative sector for the corresponding chimney. Applying Theorem \ref{thm:ADLVChimneys} when $P\neq B$ thus requires entirely different constructions of positively folded galleries than any that have appeared in the literature to date. 

In this work, we focus on the case in which the Newton points have associated parabolic subgroups with a Levi factor of rank 1.  We then treat the family of alcoves in the shrunken dominant Weyl chamber having finite part equal to the longest element of the finite Weyl group. 
Our main application of Theorem \ref{thm:ADLVChimneys} is the following nonemptiness statement. We refer the reader to Section~\ref{sec:Notation} for our notation on Weyl groups, root systems, and alcoves.

\begin{thm}\label{thm:w0ShrunkenDominant} Let $b \in G(F)$, let $\nu_b$ be the Newton point of the $\sigma$-conjugacy class $[b]$, and let $P=P_{\nu_b}$ be the associated spherical standard parabolic subgroup of $G(F)$.  Let $x_0 = t^\lambda w_0 \in \eW$ be such that the Kottwitz point $\kappa_G(x_0) = \kappa_G(b)$.  Assume that:
\begin{enumerate}
\item the alcove $\x_0$ is in the shrunken dominant Weyl chamber; 
\item $\nu_b$ is contained in the polytope $\Conv( \sW (\lambda - 2\check{\rho}) )$; and
\item one of the following holds:
\begin{enumerate}
\item $P$ has rank $0$; that is, $P = B$;
\item $P$ has rank $1$; that is, $P = B \sqcup Bs_i B$ for some spherical simple reflection~$s_i$; 
or
\item $P$ has full rank; that is, $P = G$.
\end{enumerate}
\end{enumerate}
Then $X_{x_0}(b) \neq \emptyset$.  
\end{thm}

We remark that several of the nonemptiness results in Theorem~\ref{thm:w0ShrunkenDominant} already appeared in the authors' earlier work \cite{MST1}; e.g.\ both the case of regular translations (3a) and an alternate proof of the basic case (3c).  He has also recently obtained even more general nonemptiness results \cite{HeCordial}, although the methods in the present paper differ quite  substantially.  Our intention in Theorem~\ref{thm:w0ShrunkenDominant} is to highlight those uniform nonemptiness results for which proofs using exclusively folded gallery methods exist. Of course, applying any automorphism of the Dynkin diagram yields uniform gallery constructions for further nonemptiness results via \cite[Thm.\ 10.3]{MST1}, though we do not include these additional cases in the above statement.

We continue by discussing some of the key elements in the constructions required to prove Theorem \ref{thm:w0ShrunkenDominant} via Theorem \ref{thm:ADLVChimneys}.
First, we require an explicit understanding of the standard representatives for $\sigma$-conjugacy classes defined in \cite{GHKRadlvs}, in order to fully describe their conjugacy classes in the extended affine Weyl group.  As is clear from the statement of Theorem \ref{thm:ADLVChimneys}, these $\eW$-conjugacy classes coincide with the target end alcoves for our folded galleries. We characterize these conjugacy classes geometrically in Section~\ref{sec:conjugacy} in terms of certain affine subspaces in an apartment of the Bruhat--Tits building. For example, in Figure \ref{fig:chimneyintro}, the three lines form the so-called transverse spaces, along which the colored alcoves corresponding to the elements of the $\eW$-conjugacy class lie. 

More delicately, we actually need a solution to the inverse problem; i.e.\ given any target $z$ in the $\eW$-conjugacy class of a fixed standard representative $b_\nu$, we provide an explicit element $y$ which conjugates $b_\nu$ to $z$. Applying Theorem \ref{thm:ADLVChimneys} requires this level of precision, since the conjugating element $y$ controls the location of the chimney; observe how the color-coded chimneys vary with the corresponding target alcove in Figure \ref{fig:chimneyintro}. Although many of our results in Sections \ref{sec:NewtonPointsReps} through \ref{sec:conjugacy} concern general Newton points, the question of whether the $\eW$-conjugacy class of a fixed standard representative coincides with the alcoves along the transverse subspaces (as depicted in Figure \ref{fig:chimneyintro}) is very sensitive to the rank of the corresponding parabolic subgroup; see Example \ref{ex:s1s3}, which illustrates this claim for $G=SL_4$.

The final two sections of the paper consist of delicate gallery constructions. We begin in Section~\ref{sec:minGallery} by carefully choosing a minimal gallery $\gamma: \fa\rightsquigarrow \x_0$, starting with a specific reduced expression for $w_0$ which is tailored to the parabolic subgroup $P$. We then apply a particular PRS-folding sequence to $\gamma$ in Section~\ref{sec:folding}, exercising sufficient control to prove in Section~\ref{sec:conjugationFinal} both that the end alcove of the resulting folded gallery $\gamma_{\check{\rho}}$ is an element of the $\eW$-conjugacy class of the standard representative $b_\nu$, and that $\gamma_{\check{\rho}}$ is positively folded with respect to the chimney dictated by the conjugating element. The main result of Section \ref{sec:firstTarget} is then Theorem~\ref{thm:firstTarget}, which proves nonemptiness for the largest possible expected rank 1 Newton point below the maximum.

In Section~\ref{sec:remainingTargets}, we modify the gallery $\gamma_{\check{\rho}}$ constructed in Section \ref{sec:firstTarget} to prove Theorem \ref{thm:remainingTargets}, which is a nonemptiness statement for all smaller Newton points associated to the same rank 1 parabolic subgroup. As in \cite[Sec.~6.3]{MST1}, the key idea here is to apply the root operators of \cite{GaussentLittelmann}, suitably modified to the context of alcove walks. However, due to the aperiodicity of the chimney orientations, extra care is needed to preserve the positivity of the folds in $\gamma_{\check{\rho}}$ when applying these root operators. We are thus forced to split the remaining constructions in two parts.  We discuss larger Newton points in Sections~\ref{sec:rootOps} and \ref{sec:lower}, and then further modify those galleries obtained via root operators, in order to separately tackle smaller Newton points  in Section~\ref{sec:lowest}. 

\subsection{Discussion of the general case}

In this section, we make some additional remarks concerning obstructions to  straightforward generalizations of Theorem \ref{thm:w0ShrunkenDominant}.

As discussed in \cite[Sec.\ 7]{MST1}, the ability to treat other spherical directions apart from the $w_0$-position in the translation case hinged upon an algebraic resolution to the nonemptiness question in the basic case.  While a more general algebraic solution is recently available in \cite{HeCordial}, it is our intention to present here only those results for which we are able to provide a logically independent proof, using the distinct method of positively folded galleries. Resolutions to Questions 1 and 2 posed in \cite{MST1} would likely be the first steps toward constructions that are less reliant upon the properties of vertex-to-vertex galleries developed in \cite{GaussentLittelmann}, which are best suited for galleries whose type labels an alcove in the $w_0$-position in the dominant Weyl chamber.

Another approach would be to use the recent saturation results of the first author in collaboration with Viehmann.  More specifically, part (a) of \cite[Thm.\ 1.2]{MilVie} guarantees the existence of those positively folded galleries required to apply Theorem \ref{thm:ADLVChimneys} whose type labels an element in the antidominant Weyl chamber.  To prove nonemptiness for all elements satisfying Reuman's criterion in the shrunken dominant Weyl chamber, as predicted by \cite[Conj.\ 9.5.1]{GHKRadlvs}, then requires an adaptation of \cite[Prop.\ 9.4]{MST1} concerning the geometric effects of conjugating a folded gallery by a simple reflection. A naive first try produced some type-dependent subtleties in case the gallery either begins or ends with a fold, and so we have elected not to further pursue this method, given the timing of this work.

As Example \ref{ex:s1s3} illustrates, the primary issue in treating higher rank parabolic subgroups is the fact that the $\eW$-conjugacy classes do not necessarily fill out the transverse subspaces associated to a given Newton point. Roughly speaking, these problems arise when there is a smaller possible parabolic subgroup associated to the Newton point than the choice made in \cite[Sec.\ 7]{GHKRadlvs}.  Early experiments with a modified choice of representative for the $\sigma$-conjugacy class suggest significant barriers to adapting the proof of \cite[Thm.\ 11.3.1]{GHKRadlvs} in a manner which also preserves the nature of the dimension theory developed in \cite[Sec.\ 10]{GHKRadlvs}. Moreover, the gallery constructions needed to prove nonemptiness for higher rank parabolic subgroups are even more delicate than the constructions we provide here, as the number of families of hyperplanes with aperiodic orientations increases linearly with rank.

Using the ideas developed in Section \ref{sec:alcoveWalks}, one also easily obtains a lower bound on the dimensions of the affine Deligne--Lusztig variety corresponding to each gallery we construct.  However, comparisons to the known upper bound of the virtual dimension from \cite{HeAnnals} reveal that equality only holds in type $\tilde{A}_2$, and the dimension formula in this case appeared earlier in \cite{Yang}.  In general, since the number of folds is fixed at $\ell(w_0)-1$ in the parabolic rank 1 case, one needs to produce a gallery which results in more positive crossings once folded.  

Finally, we comment on the requirement that the Newton point be contained in the convex hull of the $\sW$-orbit of the coweight $\lambda$, only after shifting by $-2\check{\rho}$.  Since Mazur's inequality is no longer sufficient for nonemptiness in the affine flag variety, the maximum Newton point generally differs from $\lambda$, and the required correction factors are greatest in the dominant Weyl chamber; see \cite{MilQBG}. Moreover, the first author's thesis \cite{Be1} provides concrete examples in which the only nonempty Newton stratum between $\lambda$ and $\lambda-2\check{\rho}$ is in fact the generic one.  Therefore, condition (2) in Theorem \ref{thm:w0ShrunkenDominant} is the cleanest hypothesis under which one could reasonably expect to prove a uniform nonemptiness result; compare \cite[Cor.\ 6.4]{HeCordial}

\subsection{Organization of the paper}

This section provides a brief overview of the contents of the body of the paper. We refer the reader to the opening remarks in each section for further pointers about the specific contents within each subsection.

In Section~\ref{sec:Preliminaries}, we recollect all terminology and external results which are used later in the paper: notation concerning Weyl groups and their root systems, basics on Newton points and their standard representatives, a review of affine Deligne--Lusztig varieties in the affine flag variety, as well as the main definitions concerning folded galleries and chimney orientations. The goal of Section \ref{sec:NewtonPointsReps} is to provide concrete type-free formulas for the standard representatives of the relevant Newton points, for which we rephrase the Newton point calculation geometrically in terms of certain averaging projection operators.  In Section~\ref{sec:alcoveWalks}, we establish the direct link between folded alcove walks and the general problem of determining nonemptiness and dimensions of affine Deligne--Lusztig varieties in the affine flag variety.   
In Section \ref{sec:conjugacy}, we obtain a complete characterization of the $\eW$-conjugacy class of a standard representative for any Newton point with associated parabolic of rank 1.  Explicit galleries designed to reach the maximum Newton point under consideration, as well as transformations of these galleries via root operators to reach the lower Newton points, are constructed in Sections~\ref{sec:firstTarget} and \ref{sec:remainingTargets}.

\subsection*{Acknowledgements}

We thank Ulrich G\"ortz, Gregory Maloney, and Arun Ram for helpful conversations. Crucial early discussions occurred during the visit by EM and AT to the University of Melbourne in January 2016, hosted by Arun Ram. The authors are grateful to the Mathematical Sciences Research Institute (MSRI), which hosted AT for a long-term stay, as well as short-term visits by each of EM and PS, in Fall 2016. PS wishes to thank the University of Sydney Mathematical Research Institute (SMRI) for hosting her in August 2019.  We also gratefully acknowledge the support of the Max-Planck-Institut f\"ur Mathematik (MPIM), which has hosted two long-term sabbatical visits by EM in 2016 and 2020, in addition to a collaborative visit by PS in March 2020. Finally, we thank Ulrich G\"ortz for useful comments on an earlier version of this manuscript, and the referee for a very careful reading and thoughtful suggestions.


\section{Preliminaries}\label{sec:Preliminaries}

This section reviews all preliminary definitions and results which are required throughout the paper.  Section \ref{sec:Notation} establishes our notation for Weyl groups, root systems, and alcoves, using the conventions from \cite{Bourbaki4-6}.  In Section \ref{sec:NewtonMap}, we review the Newton and Kottwitz maps, which determine the two invariants that parameterize $\sigma$-conjugacy classes. Following \cite{GHKRadlvs}, we discuss the standard representative for a given $\sigma$-conjugacy class in Section \ref{sec:StandardReps-General}. We provide a brief review of affine Deligne--Lusztig varieties in Section \ref{sec:ADLVs}. We conclude the preliminary material in Section \ref{sec:galleriesChimneys} by recalling the necessary background on combinatorial galleries and chimneys from \cite{MST1, MNST}.

\subsection{Notation}\label{sec:Notation}

Let $G$ be a split connected reductive group over $\F_q$, with $q$ a prime power, and let $T$ be a split maximal torus of $G$. Let $F=\overline{\mathbb{F}}_q((t))$ be the discretely valued field with ring of integers $\cO = \overline{\mathbb{F}}_q[[t]]$.  The choice of $T$ corresponds to fixing an apartment $\cA$ in the Bruhat--Tits building for $G(F)$.  We identify the apartment $\cA$ with $\mathfrak{a} := X_*(T)\otimes_{\Z}\R \cong \R^n$ where $n$ is the rank of $G$, and we let $\mathfrak{a}_{\Q} := X_*(T)\otimes_{\Z}\Q$. Fix a Borel subgroup $B$ and let $P=MN$ be a standard parabolic subgroup, with $M$ the unique Levi subgroup containing $T$ and unipotent radical $N$.  In the special case $P=B$, we write $B=TU.$ The finite or spherical Weyl group of $T$ in $G$ equals $\sW = N_G(T)/T$.  The extended affine Weyl group is then $\eW = N_GT(F)/T(\cO) \cong X_*(T) \rtimes \sW.$ Given any $\lambda \in X_*(T)$, we write $t^\lambda$ to denote the image of $t$ under the homomorphism $\lambda: \mathbb{G}_m \rightarrow T$, and $t^\lambda$ acts on $\mathfrak{a}$ by translation by $\lambda$.  Any element $x \in \eW$ can  be uniquely written as $x = t^\lambda w$ for some $\lambda \in X_*(T)$ and spherical direction $w \in \sW$.

Denote by $\Phi$ the set of roots of $T$ in $G$, which we assume is irreducible throughout the paper; equivalently, the associated Dynkin diagram is assumed to be connected. Denote by $\Phi^+$ those roots which are positive with respect to the opposite Borel subgroup $B^-$.  Let $\Delta = \{\alpha_i\}_{i=1}^n$ be a basis of simple roots in $\Phi^+$.  The set of indices on the elements of $\Delta$ shall be denoted by $[n] = \{1,2, \dots, n\}$. Denote by $\rho$ the half-sum of the positive roots in $\Phi^+$, and let $\tilde{\alpha}$ denote the highest root.  The coroot associated to $\alpha \in \Phi$ with respect to the evaluation pairing $\langle \cdot, \cdot \rangle: X^*(T) \times X_*(T) \rightarrow \Z$ is then $\alpha^\vee = 2\alpha/ (\alpha, \alpha )$, where $(\cdot , \cdot )$ denotes the standard Euclidean dot product. Denote by $\Delta^\vee$ the corresponding simple coroots. The fundamental weights $\{\omega_i\}$ and the fundamental coweights $\{\varpi_i\}$ are dual bases to $\Delta^\vee$ and $\Delta$, respectively, with respect to $\langle \cdot, \cdot \rangle$.   The dominant Weyl chamber is defined as $\Cf = \{v \in \mathfrak{a} \mid \langle \alpha, v \rangle \geq 0 \ \text{for all}\ \alpha \in \Phi^+\}$, and $\mathfrak{a}_{\Q}^+$ denotes the subset of $\mathfrak{a}_{\mathbb{Q}}$ which is contained in $\Cf$. The shrunken dominant Weyl chamber $\Cfs$ is defined by the same inequalities, after replacing the $0$ by $1$. 

Let $R^\vee = \bigoplus \Z\alpha_i^\vee \subset X_*(T)$ denote the coroot lattice, and let $Q^\vee = \bigoplus \Z \varpi_i \subset X_*(T)$ denote the coweight lattice.  Denote by $\check{\rho}$ the element of $Q^\vee$ defined as the half-sum of the positive coroots, or equivalently the sum of the fundamental coweights; note that $\check{\rho}$ may not be a scalar multiple of $\rho$. An element of $Q^\vee$ which lies in $\Cf$ is called a dominant coweight, and we remark that $\check{\rho}$ is dominant. The dominance ordering on $R^\vee$ defines $\lambda \geq \mu$ if and only if $\lambda-\mu$ is a non-negative $\Z$-linear combination of positive coroots. The partial ordering $\geq$ is then extended to $Q^\vee$ and $\mathfrak{a}$ by replacing $\Z$ with $\Q$ and $\R$, respectively. Denote by $\Lambda_G = X_*(T)/R^\vee$.

The finite Weyl group $\sW$ is a Coxeter group, generated by the set $S = \{ s_i\}_{i=1}^n$ of (simple) reflections across the hyperplanes orthogonal to the respective simple roots in $\Delta$.  Denote by $\ell(w)$ the length of an element $w \in \sW$ with respect to $S$, and denote by $w_0$ the longest element in $\sW$. We may also view $\sW$ as a finite reflection group acting on $\mathfrak{a} \cong \R^n$. Given any $\mu \in \mathfrak{a}$, we denote by $\mu^+$ the unique element of $\Cf$ which lies in the $\sW$-orbit of $\mu$.

Given any root $\alpha \in \Phi$ and any integer $k \in \Z$, define $H_{\alpha, k} = \{v \in \mathfrak{a} \mid \langle \alpha, v \rangle = k\}.$ In particular, note that $\check{\rho} \in H_{\alpha_i,1}$ for all $\alpha_i \in \Delta$. The complement in $\mathfrak{a}$ of the collection of all the affine hyperplanes $H_{\alpha,k}$ are called alcoves. The base alcove is defined to be the unique alcove in $\Cf$ whose closure contains the origin, denoted by $\fa_0 =\{v \in \Cf \mid \langle \alpha, v \rangle \leq 1 \ \text{for all}\ \alpha \in \Phi^+\}$.   The reflection across the affine hyperplane $H_{\alpha, k}$ defined by $s_{\alpha, k}(v) = v - (\langle \alpha, v \rangle - k)\alpha^\vee$ is an affine transformation of $\mathfrak{a}$, and the collection of all such reflections generates the (non-extended) affine Weyl group $\aW$. When $k=0$, we typically write $H_{\alpha, 0} = H_{\alpha}$ and $s_{\alpha, 0} = s_\alpha$. We remark that choosing $\Phi^+$ to be positive with respect to $B^-$ implies that the positive roots lie on the same side of the hyperplane $H_{\tilde{\alpha}}$ as $\Cf$.  Even when $k \neq 0$, we refer to $H_{\alpha,k}$ as an $\alpha$-hyperplane.

The action of the affine Weyl group $\aW$ on $\mathfrak{a}$ preserves the collection of hyperplanes $H_{\alpha, k}$, and so the element $x \in \aW$ naturally corresponds to the alcove $x\fa_0$, which we typically denote in bold by $\x$. The extended affine Weyl group $\eW$ acts simply transitively on the set of extended alcoves, which may be viewed as $|\Lambda_G|$ distinct copies of $\mathfrak{a}$, where each copy is endowed with the same configuration of alcoves on which the affine Weyl group $\aW$ acts simply transitively. The affine Weyl group can also be viewed as $\aW \cong R^\vee \rtimes \sW$, and so every element $x \in \aW$ can be uniquely expressed as $x = t^\lambda w$ for some $\lambda \in R^\vee$ and $w \in \sW$. Similarly, the extended affine Weyl group is isomorphic to $\eW \cong Q^\vee \rtimes \sW$.

The affine Weyl group $\aW$ is a Coxeter group with generating set $\tilde{S} = S \cup \{ s_{\tilde{\alpha},1}\},$ and the additional affine generator is  denoted by $s_0 :=s_{\tilde{\alpha},1}$.  We again denote the corresponding length function by $\ell: \aW \rightarrow \Z$. Although the extended affine Weyl group $\eW$ is no longer a Coxeter group, the length function $\ell$ extends to $\eW$ by defining those elements whose extended alcove projects onto the base alcove $\fa_0$ in $\mathfrak{a}$ to have length zero. Denote by $\Omega_G$ the subgroup of elements $\omega \in \eW$ such that $\ell(\omega) = 0$, and note that $|\Omega_G| = |\Lambda_G|$. An element $\omega \in \Omega_G$ acts on the base alcove $\fa_0$ by an automorphism, which induces a permutation action on the set $\tilde I = \{0,1,\dots, n\}$ indexing the elements of $\tilde S$. We then also have $\eW \cong \aW \rtimes \Omega_G$, so that for any $x \in \eW$, we may also write $x = y \omega$ for unique $y \in \aW$ and $\omega \in \Omega_G$.

Following \cite{Ram}, we view the elements in $\Omega_G$ as deck transformations, which move between the sheets of extended alcoves in the product $\Omega_G \times \mathfrak{a}$. For any $\omega \in \Omega_G$, the extended alcove $\omega \times \fa_0$ then projects onto the extended base alcove $0 \times \fa_0 \in \Omega_G \times \mathfrak{a}$. Given $\omega \in \Omega_G$, we denote the corresponding extended base alcove by $\fa_\omega = \omega \times \fa_0$. By abuse of notation, we often denote the extended base alcoves by $\fa_0 = 0\times \fa_0$ and $\fa = \fa_\omega$ for nonzero $\omega \in \Omega_G$. Similarly, the extended alcove in $\Omega_G \times \mathfrak{a}$ corresponding to the element $x \in \eW$ is typically denoted in bold by $\x=x\fa_0$. We occasionally refer to an extended alcove simply as an alcove, since the precise meaning should always be clear from context. Further, we present some figures which overlay extended alcoves from different sheets together in a single copy of $\mathfrak{a}$.

Since $\eW \cong Q^\vee \rtimes W \cong \aW \rtimes \Omega_G$ and $\aW \cong R^\vee \rtimes W$, then an element $\omega \in \Omega_G$ also determines a nonzero coset in $Q^\vee /R^\vee \cong \Omega_G$. As such, we shall typically index distinct sheets of extended alcoves by the elements of $Q^\vee/R^\vee$.   In particular, the image in $Q^\vee/R^\vee$ of the translation part of $x \in \eW$  uniquely identifies the sheet containing the extended alcove $\x$. By abuse of notation, we also frequently write $\lambda \in \omega + R^\vee$ for $\omega \in \Omega_G$ to denote the image of a given $\lambda \in Q^\vee$ in the quotient $Q^\vee/R^\vee$.

Given a Levi decomposition for a standard (spherical) parabolic subgroup $P=MN$ of $G$, the set of simple roots for $G$ decomposes as $\Delta = \Delta_M \sqcup \Delta_N,$ where $\Delta_M$ is the set of simple roots for the reductive group $M$, and $\Delta_N$ is the set of simple roots which occur in the Lie algebra of $N$. Note that although $\Phi$ is assumed to be irreducible, the set $\Phi_M$ of roots of $T$ in the Levi subgroup $M$ may not be.  Define $\mathfrak{a}_P = \{ v \in \mathfrak{a} \mid \langle \alpha_i, v \rangle = 0,\ \text{for all}\ \alpha_i \in \Delta_M\};$ equivalently, $\mathfrak{a}_P$ is the intersection of all hyperplanes $H_{\alpha_i}$ corresponding to those simple roots in $\Delta_M$.  The choice of standard parabolic then determines an open chamber in $\mathfrak{a}_P$ defined as
\begin{equation}\label{eq:aP+}
     \mathfrak{a}_P^+ = \{ v \in \mathfrak{a}_P \mid \langle \alpha, v \rangle>0\ \text{for all}\ \alpha \in \Delta_N \}.
\end{equation}
The dominant Weyl chamber is a disjoint union of these open chambers $\Cf = \coprod \mathfrak{a}_P^+,$ where $P$ ranges over all standard (spherical) parabolic subgroups of $G$.  Using the identification $\mathfrak{a}_P \cong \Lambda_M \otimes_{\Z}\R$, denote by $\Lambda_M^+$ the subset of $\Lambda_M$ whose image under this isomorphism lies in $\mathfrak{a}_P^+$.

The Iwahori subgroup $I$ of $G(F)$ is the stabilizer of the base alcove $\fa_0$. Equivalently, $I$ is the preimage of the opposite Borel under the projection $G(\cO) \rightarrow G(\overline{\mathbb{F}}_q)$ defined by $t \mapsto 0$. Note that the base alcove $\fa_0$ coincides with the basepoint of the affine flag variety $G(F)/I$. Given any standard (spherical) parabolic subgroup $P=MN$ of $G$, we define an associated subgroup of $G(F)$ by $I_P = (I\cap M(F))N(F)$. 

The Frobenius automorphism $a \mapsto a^q$ on $\overline{\mathbb{F}}_q$ can be extended to $\sigma: F \rightarrow F$ by defining $\sigma$ to raise each coefficient to the $q^{\text{th}}$ power. Two elements $g,h \in G(F)$ are said to be $\sigma$-conjugate if $g = xh\sigma(x)^{-1}$ for some $x \in G(F)$.  The collection of all $\sigma$-conjugates of an element $g \in G(F)$ will typically be denoted by $[g]$.  Given any subgroup $H$ of $G(F)$ and any $g,h \in G(F)$, we shall denote the usual conjugation action by $H^g = gHg^{-1}$ and $h^g = ghg^{-1}$. Given any $x,y \in \eW$, we also denote by $\x^y$ the extended alcove corresponding to the element $x^y$.

\subsection{The Newton and Kottwitz maps}\label{sec:NewtonMap}

The set $B(G)$ of $\sigma$-conjugacy classes in $G(F)$ is characterized by a pair of invariants: the Newton point and the Kottwitz point. In this section we briefly review the key properties of these two $\sigma$-conjugacy invariants, originally defined in \cite{KotIsoI, KotIsoII}.

Denote by $\kappa_G: G(F) \rightarrow \Lambda_G$ the natural surjection defined in \cite{KotIsoII}, subsequently called the \emph{Kottwitz homomorphism}. By \cite[Lem.~7.2.1]{GHKRadlvs}, the map $\kappa_G$ induces a bijection $\Lambda_G \longleftrightarrow \Omega_G$.  Given any $b \in G(F)$, we may thus view the \emph{Kottwitz point} $\kappa_G(b)$ as specifying an extended base alcove $\fa_\omega$ for some $\omega \in \Omega_G$. Equivalently, $\kappa_G(b)$ determines a coset in $Q^\vee /R^\vee \cong \Omega_G$.

By \cite[Cor.~7.2.2]{GHKRadlvs}, the restriction of the map $G(F) \rightarrow B(G)$ to $N_GT(F)$ factors through $\eW$, and the induced map $\eW \twoheadrightarrow B(G)$ is surjective.  We thus define the Newton map only on the extended affine Weyl group, where it has an especially elementary formula; see, for example, \cite[Sec.~4.2]{GoertzSurvey}. 
For $x = t^\lambda w \in \eW$ where $\lambda \in Q^\vee$ and the order of $w \in \sW$ equals $m$, the \emph{Newton point} for $x$ (viewed as an element of $G$) is given by 
\begin{equation}\label{eq:avgFormula}
    \nu_G(x) = \left( \frac{1}{m} \sum\limits_{i=1}^m w^i(\lambda)\right)^+.
\end{equation}
In particular, a Newton point $\nu$ is an element of $\mathfrak{a}_{\Q}^+$. As with an element of the integral coweight lattice, we say that $\nu \in \mathfrak{a}_{\Q}^+$ 
is \emph{regular} if and only if the stabilizer of $\nu$ in $\sW$ is trivial.

Choosing any representative $\dot{x} \in [b]$ such that $\dot{x} \in \eW$, and any representative $x \in G(F)$ of $\dot{x} \in N_GT(F)/T(\cO)$, the map
\begin{align}\label{eq:NewtonKottMap}
    B(G) & \rightarrow \mathfrak{a}^+_{\Q} \times \Lambda_G \nonumber \\
    [b] & \mapsto (\nu_G(x), \kappa_G(b))
\end{align}
is injective \cite[Sec.~4.13]{KotIsoII}. By abuse of notation, we often write $\nu_G(b)$ to denote the image of $[b]\in B(G)$ under the Newton map, without passing to a representative in $\eW$ or referencing its $\sigma$-conjugacy class.  When the context is clear and no further reference to the ambient group is required, we may also abbreviate $\nu_b :=\nu_G(b)$.

\begin{definition}\label{defn:integralNP}
Given any $[b] \in B(G)$, denote its Newton point by $\nu_b \in \mathfrak{a}_{\Q}^+$ and its Kottwitz point by $\omega_b \in \Lambda_G$. By definition, $\nu_b \in \bigoplus \Q\varpi_i$, where the coefficients on the fundamental coweights are all non-negative. If $\nu_b \in \bigoplus \Z \varpi_i$ and $\nu_b$ has the same image as $\omega_b$ in $Q^\vee/R^\vee$, then we say that $b$ has an \emph{integral} Newton point; equivalently, $\nu_b$ is an \emph{integral} Newton point.  Otherwise, we say that $b$ has a \emph{non-integral} Newton point; equivalently, $\nu_b$ is \emph{non-integral}.
\end{definition}

Although our definition of integrality formally requires the additional input of the Kottwitz point, note that if $Q^\vee = R^\vee$, then $\nu_b$ is integral if and only if $\nu_b$ is an element of the integral weight lattice.  Example \ref{ex:C2integral} below illustrates how the same Newton point (representing two different $\sigma$-conjugacy classes) can be both integral and non-integral, if more generally $Q^\vee \neq R^\vee$.  The primary motivation for splitting the cases as we have in Definition \ref{defn:integralNP} arises later in Proposition \ref{prop:stdReps}, in which our formulas for standard representatives fall precisely into these two distinct cases.

\begin{figure}[ht]
\centering
\begin{overpic}[width=0.6\textwidth]{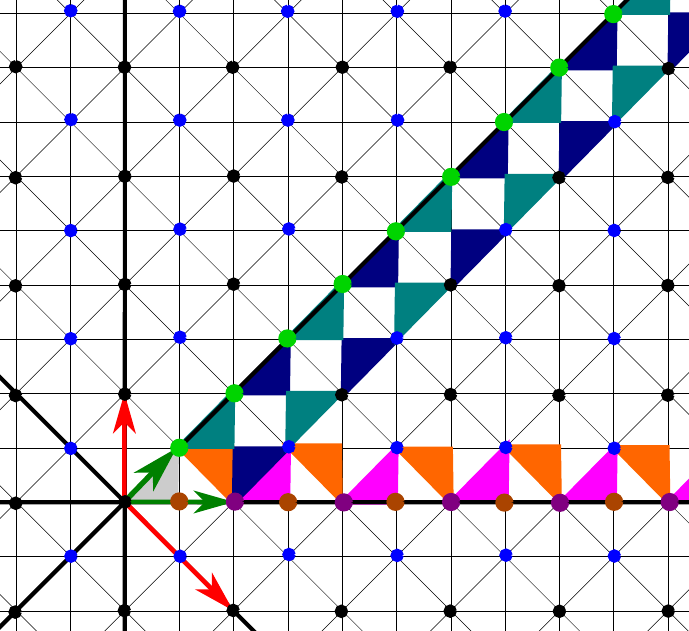}
\put(23,20.5){\footnotesize{\color{black}$\fa_0$}}
\put(13,28){\footnotesize{\color{red}$\alpha_2^\vee$}}
\put(25,4){\footnotesize{\color{red}$\alpha_1^\vee$}}
\put(23,29){\footnotesize{\color{DarkGreen}$\varpi_2$}}
\put(32,15){\footnotesize{\color{DarkGreen}$\varpi_1$}}
\end{overpic}
\caption{Newton points and their standard representatives in type $\tilde{C}_2$.}
\label{fig:C2StdReps}
\end{figure}
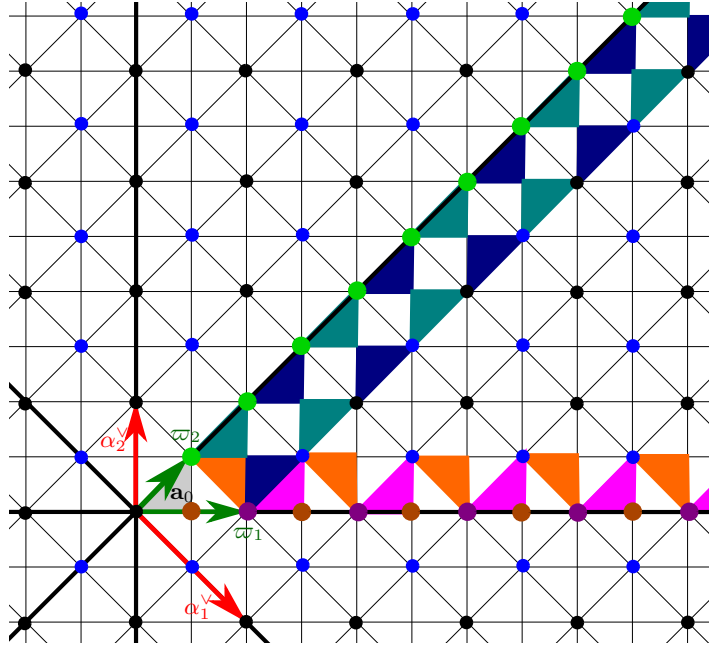

\begin{example}\label{ex:C2integral}
In Figure~\ref{fig:C2StdReps}, the coroot lattice $R^\vee = \Z^2$ for type $\tilde{C}_2$ is shown in black, and the lattice $\varpi_2 + R^\vee$ is shown in blue.
Consider $b = t^{3\varpi_2} \in \eW$ whose (extended) alcove $\bb$ is shaded teal.   The element $b$ has Newton point $\nu_b = \left( \frac{3}{2}, \frac{3}{2} \right) = 3\varpi_2$ by \eqref{eq:avgFormula}.
The translation part $3\varpi_2$ of $b$ is an element of the blue lattice, and so the element $b$ has Kottwitz point $\omega_b = \varpi_2$.  Since $2\varpi_2 \in R^\vee$, the Newton point $\nu_b$ maps to $\varpi_2 \in Q^\vee/R^\vee$, and so $\nu_b=3\varpi_2$ is integral by Definition \ref{defn:integralNP}. 

On the other hand, consider $\mu = (2,1) = 2\alpha_1^\vee + 3\alpha_2^\vee$ and define $b' = t^{\mu}s_1 \in \eW$, which corresponds to a (non-extended) alcove $\bb'$ shaded dark blue, Again by \eqref{eq:avgFormula}, we have $\nu_{b'} = 3\varpi_2$. In particular, $b$ and $b'$ have the same Newton points.  However, since $\mu \in R^\vee$, then $\omega_{b'} = 0$, which is not equal to the image of $\nu_{b'}$ in $Q^\vee/R^\vee$.  Therefore, $\nu_{b'}=3\varpi_2$ is non-integral in this case. Figure \ref{fig:C2StdReps} is discussed in more detail later in Example \ref{ex:C2stdreps}.
\end{example}

We continue by reviewing the geometric method for associating a standard parabolic subgroup to a Newton point; see \cite[Sec.~7]{GHKRadlvs} for equivalent algebraic interpretations. Recall that the dominant Weyl chamber decomposes as $\Cf = \coprod \mathfrak{a}^+_P$, where $P$ ranges over the standard parabolic subgroups. 

\begin{definition}\label{defn:assocParabolic} Given any $[b] \in B(G)$, denote the associated Newton point by $\nu = \nu_G(b)$.  
The \emph{parabolic subgroup associated to $\nu$} is defined to be the unique standard parabolic subgroup $P_\nu$ such that $\nu \in \mathfrak{a}^+_P$. Equivalently, $P_{\nu} = MN$ is defined by the criterion 
    \begin{equation}\label{eq:PnuRoots}
        \alpha_i \in \Delta_M  \iff  \nu \in H_{\alpha_i}.
    \end{equation}
\end{definition}

We now mention several important special cases which arise elsewhere in the literature; compare \cite[Ex.~11.3.6]{GHKRadlvs}, for example.

\begin{example}\label{ex:Pnu}
Let $[b] \in B(G)$, and denote by $\nu = \nu_G(b)$ the associated Newton point.  
 \begin{enumerate}
     \item $P_{\nu} = B$ if and only if $\nu$ is regular. 
 \item $P_\nu = G$ if and only if $\nu$ in the $\omega$-sheet projects onto the origin in the 0-sheet.
 \end{enumerate}
\end{example}

\subsection{Standard representatives for Newton points}\label{sec:StandardReps-General}

Denote by $B(G)_P$ the subset of $B(G)$ consisting of all elements $[b]$ such that $P_{\nu_b}$ equals $P$, and note that $B(G) = \coprod B(G)_P$, where $P$ runs over the set of all standard parabolic subgroups. In this section, we review the fact that the elements of $B(G)_P$ can be represented by certain convenient elements in the corresponding Levi subgroup $M(F)$. 

An element $b \in G(F)$ is called \emph{basic} if its image under the Newton map $\nu_b$ factors through the center $Z(G)$ of $G$.  Fix a standard parabolic subgroup $P$ with Levi decomposition $P = MN$.  For any basic element $[b_m] \in B(M)$, the image under the Newton map with respect to $M$ satisfies $\nu_M(b_m) \in \mathfrak{a}_P$, although it may not additionally satisfy $\nu_M(b_m) \in \mathfrak{a}_P^+$.  Denote by $B(M)_{\bas}^+$ the subset of basic elements $[b_m] \in B(M)$ such that $\nu_M(b_m) \in \mathfrak{a}_P^+$.  Then by \cite[Sec.~5]{KotIsoII}, there is a bijection $B(M)_{\bas}^+ \longleftrightarrow B(G)_P.$ 

By \cite[Lemma 7.2.1]{GHKRadlvs}, the elements in $\Omega_M \subset M(F)$ are $M$-basic, and furthermore, the canonical map $\Omega_M \rightarrow B(M)_{\bas}$ is a bijection. Denote by $\Omega_M^+$ the subset of $\Omega_M$ which is in bijection with $B(M)_{\bas}^+$ under this map, so that   
\begin{equation}\label{eq:StdRepMaps}
    \Omega_M^+ \longleftrightarrow B(M)_{\bas}^+ \longleftrightarrow B(G)_P.
\end{equation}
The following is a slight reformulation of \cite[Def.~7.2.3]{GHKRadlvs}.

\begin{definition}\label{defn:StdRep} Given any $[b] \in B(G)_P \subset B(G)$ associated to the Newton point $\nu = \nu_G(b)$, the element $b_{\nu} \in \Omega_M^+$ obtained from the bijections in \eqref{eq:StdRepMaps} is called the \emph{standard representative of $[b]$}.
\end{definition}

The next lemma simply rephrases the criteria characterizing the standard representative from the bijections in \eqref{eq:StdRepMaps} as a more convenient checklist, although we include a proof for the sake of completeness.

\begin{lemma}\label{lem:stdRepList}
Let $[b] \in B(G)_P$, and denote by $(\nu_b, \omega_b)$ the associated Newton and Kottwitz points. The element $b_\nu \in \eW$ is the standard representative of $[b]$ if and only if the following three conditions hold:
\begin{enumerate}
    \item $\kappa_G(b_\nu)=\omega_b$,
    \item $b_\nu \in \Omega_M$, and
    \item $\nu_M(b_\nu)=\nu_b$.
\end{enumerate}
In particular, these conditions imply that $b_\nu \in [b]$.
\end{lemma}

\begin{proof}
If $b_\nu$ is the standard representative of $[b]$, then in particular $b_\nu \in [b]$.  By \eqref{eq:NewtonKottMap}, we thus have $\nu_G(b_\nu) = \nu_b$ and $\kappa_G(b_\nu) = \omega_b$.  In particular, condition (1) holds.
Condition (2) is necessary by \eqref{eq:StdRepMaps}. In addition, $\nu_M(b_\nu) \in \mathfrak{a}_P^+$ is already $G$-dominant, and so $\nu_b = \nu_G(b_\nu) = \nu_M(b_\nu)$, which means that (3) holds as well.

Conversely, assume that the three above conditions hold.  Condition (2) guarantees that $b_\nu \in \Omega_M$ is an $M$-basic element. Condition (3) then implies that $\nu_M(b_\nu)=\nu_b$ is already $G$-dominant, and so $\nu_M(b_\nu) = \nu_G(b_\nu) = \nu_b$.  Together with condition (1), we thus have by \eqref{eq:NewtonKottMap} that $b_\nu \in [b]$.  Therefore, $[b_\nu] \in B(M)_{\bas}^+$ represents $[b]\in B(G)_P$, as required for $b_\nu$ to be the standard representative.
\end{proof}

We now record the standard representatives for $\sigma$-conjugacy classes having one of the two associated parabolic subgroups studied in Example \ref{ex:Pnu}; compare \cite[Ex.~11.3.6]{GHKRadlvs}.

\begin{example}\label{ex:stdReps} Let $[b] \in B(G)_P$.  Denote by $\nu = \nu_G(b)$ the Newton point, and let $b_\nu$ denote the standard representative for $[b]$. 
\begin{enumerate}
\item If $P = B$, then $M=T\cong (\mathbb{G}_m)^n$ and $b_\nu$ is the translation $t^\nu=(t^{\varpi_1})^{\nu_1}\cdots (t^{\varpi_n})^{\nu_n} \in \Omega_M$, where $\nu = \sum \nu_i \varpi_i \mapsto (\nu_1, \dots, \nu_n)$ under the isomorphism $X_*(T) \cong \Z^n$.
\item If $P = G$, then $b_\nu = \id$ if $\nu=0$; otherwise $\nu=\nu_G(\omega)$ where $\omega = t^\eta v \in \Omega_G$ is such that $\eta \in Q^\vee$ projects onto a vertex of the base alcove $\fa_0$, and $b_\nu = \omega$.
\end{enumerate}
\end{example}

More generally, given $[b] \in B(G)_P$ with Newton point $\nu$, we shall typically express the standard representative as $b_\nu = t^\eta v \in \eW$, where $\eta \in Q^\vee$ and $v \in \sW_M$.   As discussed in \cite[Sec.~7]{GHKRadlvs}, the coweight $\eta$ is both $M$-dominant and $M$-minuscule, meaning that $\langle \alpha_i, \eta \rangle \in \{0,1\}$ for all $\alpha_i \in \Delta_M$.  The colored alcoves in Figure \ref{fig:C2StdReps} depict the standard representatives in type $\tilde{C}_2$, which are discussed in more detail in Example \ref{ex:C2stdreps} after we have explicit formulas.

\subsection{Affine Deligne--Lusztig varieties}\label{sec:ADLVs}

This section provides a brief review of the definition of affine Deligne--Lusztig varieties in the affine flag variety, together with the key results from \cite{GHKRadlvs} upon which our work builds.

Given any $x \in \eW$ and $b \in G(F)$, the \emph{affine Deligne--Lusztig variety} associated to this pair of elements is defined as 
\[ X_x(b) = \{ g \in G(F)/I \mid g^{-1}b\sigma(g) \in IxI \}. \]
In contrast to their classical counterparts, affine Deligne--Lusztig varieties may be empty, and it remains an open problem to fully characterize the precise conditions which determine nonemptiness.  

In \cite{GHKRadlvs}, both the nonemptiness problem and dimension calculations are rephrased in terms of intersections of $I$ and conjugates of $I_P$-orbits. Below we provide a slight reformulation of the main results from \cite[Sec.~11]{GHKRadlvs}, which we shall reinterpret later in Section \ref{sec:alcoveWalks}. We include a short proof for the sake of completeness, since the statements do not appear there in precisely this form.

\begin{thm}[\cite{GHKRadlvs}]\label{thm:GHKRThm}
Let $[b] \in B(G)_P$, and denote the corresponding Newton point by $\nu$ and the standard representative by $b_\nu$. For any $x \in \eW$, we have
\begin{enumerate}
\item $X_x(b) \neq \emptyset$ if and only if there exists $y \in \eW$ such that $Ix\fa_0 \cap I_P^y b_\nu^y \fa_0 \neq \emptyset$.
\item If $X_x(b) \neq \emptyset$, then
\begin{equation*}
\dim X_x(b) = \sup\limits_{y \in \eW} \left(\dim\left( Ix\fa_0 \cap I_P^y b_\nu^y \fa_0 \right)\right)-\langle 2\rho, \nu\rangle.
\end{equation*}
\end{enumerate}
\end{thm}

\begin{proof}
 Since $[b] \in B(G)_P$, we have by \cite[Eq.~(11.3.1)]{GHKRadlvs} that 
\begin{equation*}
    \dim X_x(b_\nu) = \sup\limits_{z \in \eW} \left(\dim \left( X_x(b_\nu)\cap I_P z\fa_0 \right) \right).
\end{equation*}
Next, \cite[Thm.~11.3.1]{GHKRadlvs} says that the dimensions on the right-hand side equal
\begin{equation*}
\dim(X_x(b_\nu) \cap I_Pz\fa_0) = \dim\left( Ix\fa_0 \cap I_P^{z^{-1}} b_\nu^{z^{-1}} \fa_0 \right)-\langle \rho, \nu_G(b_\nu)+\nu_M(b_\nu)\rangle.
\end{equation*}
By property (3) of Lemma \ref{lem:stdRepList}, since $b_\nu$ is the standard representative for $[b]$, then in fact $\nu_M(b_\nu) = \nu_G(b_\nu)=\nu$. Since $b_\nu \in [b]$, then $X_x(b) \cong X_x(b_\nu)$, in which case $\dim X_x(b) = \dim X_x(b_\nu)$.  Finally, note that the value $\langle 2\rho, \nu \rangle$ is independent of $z \in \eW$. The result now follows by replacing $z=y^{-1}$.
\end{proof}

\subsection{Galleries and chimneys}\label{sec:galleriesChimneys}

We direct the reader to~\cite{MST1} and to our more recent work~\cite{MNST} with Naqvi for all background concerning combinatorial galleries, labeled folded alcove walks, and chimneys. In this section, we briefly recall only those concepts which play a key role in this paper.

Every choice of a sub-root system, and hence choice of a standard spherical parabolic subgroup $P$, determines a $P$-\emph{chimney} in the $0$-sheet as defined in \cite[Def.~3.1]{MNST}. While formally a collection of half-apartments in the $0$-sheet, chimneys are represented by \emph{$P$-sectors}, which are regions of $\cA$ lying between pairs of adjacent hyperplanes for all positive roots in the defining sub-root system associated to $P$, and which go off to infinity in the other root directions. An important special case occurs when the sub-root system is the empty set and $P=B$, in which case chimneys are maximal simplices in the spherical building at infinity, represented by parallel sectors in $\cA$; equivalently, by Weyl chambers. See \cite[Rmk.~3.2]{MNST} for more background on the geometric interpretation.

\begin{figure}[ht]
\centering
\begin{overpic}[width=0.5\textwidth]{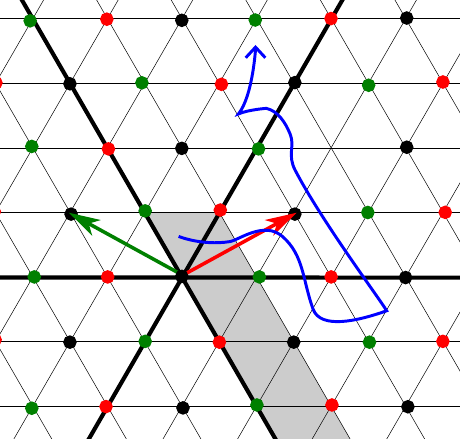}
\put(59,65){\color{blue}{$\gamma$}}
\put(14,8){\footnotesize{\color{black}$+$}}
\put(14,5){\footnotesize{\color{black}$-$}}
\put(31,8){\footnotesize{\color{black}$+$}}
\put(31,5){\footnotesize{\color{black}$-$}}
\put(46,8){\footnotesize{\color{black}$+$}}
\put(46,5){\footnotesize{\color{black}$-$}}
\put(62,8){\footnotesize{\color{black}$+$}}
\put(62,5){\footnotesize{\color{black}$-$}}
\put(79,8){\footnotesize{\color{black}$+$}}
\put(79,5){\footnotesize{\color{black}$-$}}
\put(95,8){\footnotesize{\color{black}$+$}}
\put(95,5){\footnotesize{\color{black}$-$}}
\put(14,36){\footnotesize{\color{black}$+$}}
\put(14,32){\footnotesize{\color{black}$-$}}
\put(30,36){\footnotesize{\color{black}$+$}}
\put(30,32){\footnotesize{\color{black}$-$}}
\put(46,36){\footnotesize{\color{black}$+$}}
\put(46,32){\footnotesize{\color{black}$-$}}
\put(62,36){\footnotesize{\color{black}$+$}}
\put(62,32){\footnotesize{\color{black}$-$}}
\put(80,36){\footnotesize{\color{black}$+$}}
\put(80,32){\footnotesize{\color{black}$-$}}
\put(95,36){\footnotesize{\color{black}$+$}}
\put(95,32){\footnotesize{\color{black}$-$}}
\put(14,64){\footnotesize{\color{black}$+$}}
\put(14,61){\footnotesize{\color{black}$-$}}
\put(30,64){\footnotesize{\color{black}$+$}}
\put(30,61){\footnotesize{\color{black}$-$}}
\put(46,64){\footnotesize{\color{black}$+$}}
\put(46,61){\footnotesize{\color{black}$-$}}
\put(64,64){\footnotesize{\color{black}$+$}}
\put(64,61){\footnotesize{\color{black}$-$}}
\put(79,64){\footnotesize{\color{black}$+$}}
\put(79,61){\footnotesize{\color{black}$-$}}
\put(95,64){\footnotesize{\color{black}$+$}}
\put(95,61){\footnotesize{\color{black}$-$}}
\put(14,92){\footnotesize{\color{black}$+$}}
\put(14,89){\footnotesize{\color{black}$-$}}
\put(30,92){\footnotesize{\color{black}$+$}}
\put(30,89){\footnotesize{\color{black}$-$}}
\put(46,92){\footnotesize{\color{black}$+$}}
\put(46,89){\footnotesize{\color{black}$-$}}
\put(64,92){\footnotesize{\color{black}$+$}}
\put(64,89){\footnotesize{\color{black}$-$}}
\put(79,92){\footnotesize{\color{black}$+$}}
\put(79,89){\footnotesize{\color{black}$-$}}
\put(95,92){\footnotesize{\color{black}$+$}}
\put(95,89){\footnotesize{\color{black}$-$}}
\put(5,22){\footnotesize{\color{black}$+$}}
\put(5,19){\footnotesize{\color{black}$-$}}
\put(22,22){\footnotesize{\color{black}$+$}}
\put(22,19){\footnotesize{\color{black}$-$}}
\put(38,22){\footnotesize{\color{black}$+$}}
\put(38,19){\footnotesize{\color{black}$-$}}
\put(55,22){\footnotesize{\color{black}$+$}}
\put(55,19){\footnotesize{\color{black}$-$}}
\put(71,22){\footnotesize{\color{black}$+$}}
\put(71,19){\footnotesize{\color{black}$-$}}
\put(87,22){\footnotesize{\color{black}$+$}}
\put(87,19){\footnotesize{\color{black}$-$}}
\put(5,50){\footnotesize{\color{black}$+$}}
\put(5,47){\footnotesize{\color{black}$-$}}
\put(22,50){\footnotesize{\color{black}$+$}}
\put(22,47){\footnotesize{\color{black}$-$}}
\put(38,50){\footnotesize{\color{black}$+$}}
\put(38,47){\footnotesize{\color{black}$-$}}
\put(54,50){\footnotesize{\color{black}$+$}}
\put(54,47){\footnotesize{\color{black}$-$}}
\put(71,50){\footnotesize{\color{black}$+$}}
\put(71,47){\footnotesize{\color{black}$-$}}
\put(87,50){\footnotesize{\color{black}$+$}}
\put(87,47){\footnotesize{\color{black}$-$}}
\put(5,78){\footnotesize{\color{black}$+$}}
\put(5,75){\footnotesize{\color{black}$-$}}
\put(22,78){\footnotesize{\color{black}$+$}}
\put(22,75){\footnotesize{\color{black}$-$}}
\put(38,78){\footnotesize{\color{black}$+$}}
\put(38,75){\footnotesize{\color{black}$-$}}
\put(55,78){\footnotesize{\color{black}$+$}}
\put(55,75){\footnotesize{\color{black}$-$}}
\put(71,78){\footnotesize{\color{black}$+$}}
\put(71,75){\footnotesize{\color{black}$-$}}
\put(87,78){\footnotesize{\color{black}$+$}}
\put(87,75){\footnotesize{\color{black}$-$}}
\put(-0.5,13){\footnotesize{\color{red}$+$}}
\put(4,13){\footnotesize{\color{red}$-$}}
\put(16,13){\footnotesize{\color{red}$+$}}
\put(20,13){\footnotesize{\color{red}$-$}}
\put(32,13){\footnotesize{\color{red}$+$}}
\put(36,13){\footnotesize{\color{red}$-$}}
\put(48,13){\footnotesize{\color{red}$+$}}
\put(53,13){\footnotesize{\color{red}$-$}}
\put(64,13){\footnotesize{\color{red}$-$}}
\put(69,13){\footnotesize{\color{red}$+$}}
\put(80,13){\footnotesize{\color{red}$-$}}
\put(85,13){\footnotesize{\color{red}$+$}}
\put(-0.5,41){\footnotesize{\color{red}$+$}}
\put(4,41){\footnotesize{\color{red}$-$}}
\put(16,41){\footnotesize{\color{red}$+$}}
\put(20,41){\footnotesize{\color{red}$-$}}
\put(32,41){\footnotesize{\color{red}$+$}}
\put(36,41){\footnotesize{\color{red}$-$}}
\put(48,41){\footnotesize{\color{red}$-$}}
\put(53,41){\footnotesize{\color{red}$+$}}
\put(64,41){\footnotesize{\color{red}$-$}}
\put(69,41){\footnotesize{\color{red}$+$}}
\put(80,41){\footnotesize{\color{red}$-$}}
\put(85,41){\footnotesize{\color{red}$+$}}
\put(-0.5,69){\footnotesize{\color{red}$+$}}
\put(4,69){\footnotesize{\color{red}$-$}}
\put(16,69){\footnotesize{\color{red}$+$}}
\put(20,69){\footnotesize{\color{red}$-$}}
\put(32,69){\footnotesize{\color{red}$-$}}
\put(36,69){\footnotesize{\color{red}$+$}}
\put(48,69){\footnotesize{\color{red}$-$}}
\put(53,69){\footnotesize{\color{red}$+$}}
\put(64,69){\footnotesize{\color{red}$-$}}
\put(69,69){\footnotesize{\color{red}$+$}}
\put(80,69){\footnotesize{\color{red}$-$}}
\put(85,69){\footnotesize{\color{red}$+$}}
\put(8,27){\footnotesize{\color{red}$+$}}
\put(12,27){\footnotesize{\color{red}$-$}}
\put(24,27){\footnotesize{\color{red}$+$}}
\put(28,27){\footnotesize{\color{red}$-$}}
\put(40,27){\footnotesize{\color{red}$+$}}
\put(44,27){\footnotesize{\color{red}$-$}}
\put(56,27){\footnotesize{\color{red}$-$}}
\put(61,27){\footnotesize{\color{red}$+$}}
\put(72,27){\footnotesize{\color{red}$-$}}
\put(76,27){\footnotesize{\color{red}$+$}}
\put(88,27){\footnotesize{\color{red}$-$}}
\put(93,27){\footnotesize{\color{red}$+$}}
\put(8,55){\footnotesize{\color{red}$+$}}
\put(12,55){\footnotesize{\color{red}$-$}}
\put(24,55){\footnotesize{\color{red}$+$}}
\put(28,55){\footnotesize{\color{red}$-$}}
\put(40,55){\footnotesize{\color{red}$-$}}
\put(44,55){\footnotesize{\color{red}$+$}}
\put(56,55){\footnotesize{\color{red}$-$}}
\put(61,55){\footnotesize{\color{red}$+$}}
\put(72,55){\footnotesize{\color{red}$-$}}
\put(76,55){\footnotesize{\color{red}$+$}}
\put(88,55){\footnotesize{\color{red}$-$}}
\put(93,55){\footnotesize{\color{red}$+$}}
\put(8,83){\footnotesize{\color{red}$+$}}
\put(12,83){\footnotesize{\color{red}$-$}}
\put(24,83){\footnotesize{\color{red}$-$}}
\put(28,83){\footnotesize{\color{red}$+$}}
\put(40,83){\footnotesize{\color{red}$-$}}
\put(44,83){\footnotesize{\color{red}$+$}}
\put(56,83){\footnotesize{\color{red}$-$}}
\put(61,83){\footnotesize{\color{red}$+$}}
\put(72,83){\footnotesize{\color{red}$-$}}
\put(76,83){\footnotesize{\color{red}$+$}}
\put(88,83){\footnotesize{\color{red}$-$}}
\put(93,83){\footnotesize{\color{red}$+$}}
\put(8,13){\footnotesize{\color{DarkGreen}$+$}}
\put(11.5,13){\footnotesize{\color{DarkGreen}$-$}}
\put(24,13){\footnotesize{\color{DarkGreen}$+$}}
\put(28,13){\footnotesize{\color{DarkGreen}$-$}}
\put(40,13){\footnotesize{\color{DarkGreen}$+$}}
\put(44,13){\footnotesize{\color{DarkGreen}$-$}}
\put(56,13){\footnotesize{\color{DarkGreen}$+$}}
\put(60,13){\footnotesize{\color{DarkGreen}$-$}}
\put(73,13){\footnotesize{\color{DarkGreen}$+$}}
\put(76,13){\footnotesize{\color{DarkGreen}$-$}}
\put(89,13){\footnotesize{\color{DarkGreen}$+$}}
\put(93,13){\footnotesize{\color{DarkGreen}$-$}}
\put(8,41){\footnotesize{\color{DarkGreen}$+$}}
\put(11.5,41){\footnotesize{\color{DarkGreen}$-$}}
\put(24,41){\footnotesize{\color{DarkGreen}$+$}}
\put(28,41){\footnotesize{\color{DarkGreen}$-$}}
\put(40,41){\footnotesize{\color{DarkGreen}$+$}}
\put(44,41){\footnotesize{\color{DarkGreen}$-$}}
\put(56,41){\footnotesize{\color{DarkGreen}$+$}}
\put(60,41){\footnotesize{\color{DarkGreen}$-$}}
\put(73,41){\footnotesize{\color{DarkGreen}$+$}}
\put(76,41){\footnotesize{\color{DarkGreen}$-$}}
\put(89,41){\footnotesize{\color{DarkGreen}$+$}}
\put(93,41){\footnotesize{\color{DarkGreen}$-$}}
\put(8,69){\footnotesize{\color{DarkGreen}$+$}}
\put(11.5,69){\footnotesize{\color{DarkGreen}$-$}}
\put(24,69){\footnotesize{\color{DarkGreen}$+$}}
\put(28,69){\footnotesize{\color{DarkGreen}$-$}}
\put(40,69){\footnotesize{\color{DarkGreen}$+$}}
\put(44,69){\footnotesize{\color{DarkGreen}$-$}}
\put(56,69){\footnotesize{\color{DarkGreen}$+$}}
\put(60,69){\footnotesize{\color{DarkGreen}$-$}}
\put(73,69){\footnotesize{\color{DarkGreen}$+$}}
\put(76,69){\footnotesize{\color{DarkGreen}$-$}}
\put(89,69){\footnotesize{\color{DarkGreen}$+$}}
\put(93,69){\footnotesize{\color{DarkGreen}$-$}}
\put(-0.5,27){\footnotesize{\color{DarkGreen}$+$}}
\put(4,27){\footnotesize{\color{DarkGreen}$-$}}
\put(16,27){\footnotesize{\color{DarkGreen}$+$}}
\put(20,27){\footnotesize{\color{DarkGreen}$-$}}
\put(32,27){\footnotesize{\color{DarkGreen}$+$}}
\put(36,27){\footnotesize{\color{DarkGreen}$-$}}
\put(48,27){\footnotesize{\color{DarkGreen}$+$}}
\put(52,27){\footnotesize{\color{DarkGreen}$-$}}
\put(64,27){\footnotesize{\color{DarkGreen}$+$}}
\put(68,27){\footnotesize{\color{DarkGreen}$-$}}
\put(80,27){\footnotesize{\color{DarkGreen}$+$}}
\put(84,27){\footnotesize{\color{DarkGreen}$-$}}
\put(-0.5,55){\footnotesize{\color{DarkGreen}$+$}}
\put(4,55){\footnotesize{\color{DarkGreen}$-$}}
\put(16,55){\footnotesize{\color{DarkGreen}$+$}}
\put(20,55){\footnotesize{\color{DarkGreen}$-$}}
\put(32,55){\footnotesize{\color{DarkGreen}$+$}}
\put(36,55){\footnotesize{\color{DarkGreen}$-$}}
\put(48,55){\footnotesize{\color{DarkGreen}$+$}}
\put(52,55){\footnotesize{\color{DarkGreen}$-$}}
\put(64,55){\footnotesize{\color{DarkGreen}$+$}}
\put(68,55){\footnotesize{\color{DarkGreen}$-$}}
\put(80,55){\footnotesize{\color{DarkGreen}$+$}}
\put(84,55){\footnotesize{\color{DarkGreen}$-$}}
\put(-0.5,83){\footnotesize{\color{DarkGreen}$+$}}
\put(4,83){\footnotesize{\color{DarkGreen}$-$}}
\put(16,83){\footnotesize{\color{DarkGreen}$+$}}
\put(20,83){\footnotesize{\color{DarkGreen}$-$}}
\put(32,83){\footnotesize{\color{DarkGreen}$+$}}
\put(36,83){\footnotesize{\color{DarkGreen}$-$}}
\put(48,83){\footnotesize{\color{DarkGreen}$+$}}
\put(52,83){\footnotesize{\color{DarkGreen}$-$}}
\put(64,83){\footnotesize{\color{DarkGreen}$+$}}
\put(68,83){\footnotesize{\color{DarkGreen}$-$}}
\put(80,83){\footnotesize{\color{DarkGreen}$+$}}
\put(84,83){\footnotesize{\color{DarkGreen}$-$}}
\end{overpic}
\caption{For $i = 1$ and $2$, the coroot $\alpha_i^\vee$ and the orientation induced by the $P_1$-chimney on $\alpha_i$-hyperplanes are shown in red and green, respectively.  The orientation induced on $(\alpha_1^\vee + \alpha_2^\vee)$-hyperplanes is shown in black.}
\label{fig:chimneyOrientations}
\end{figure}
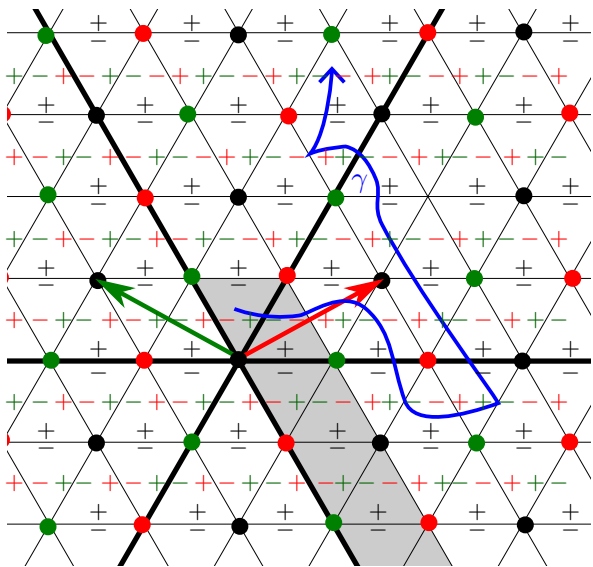

Recall that the $\omega$-sheets overlay each other such that the fundamental alcoves $\fa_\omega$ and the Weyl chambers coincide; see Figure \ref{fig:chimneyOrientations} for an illustration for $G=PGL_3$, where an extended alcove is pictured as a triangle together with a choice of a colored vertex. In particular, base alcoves $\fa_\omega$ and $\fa_{\omega'}$ for $\omega\neq\omega'$ both correspond to the same grey triangle, but are based at lattice elements of two different colors. We say that an extended alcove (or panel) \emph{sits above} or \emph{below} another extended alcove (or panel) if they project to the same alcove (or panel) in the 0-sheet. In Figure~\ref{fig:chimneyOrientations}, such alcoves (or panels) then correspond to the same simplex with a different base vertex. As for alcoves and panels, we may speak of walls, half-apartments, chimneys, and sectors as \emph{sitting above} or \emph{below} one another.

The affine Weyl group $\aW$ admits a natural action on the set of chimneys and on the set of sectors, which is induced by the action of $\aW$ on the set of half-apartments; see \cite[Def.~3.7]{MNST} for a formal definition. We extend this action to a left-action of $\eW$ as follows.  The image of a $P$-chimney under $\omega\in\Omega_G$ is the collection of half-apartments in the $\omega$-sheet which projects to the given $P$-chimney in the $0$-sheet; i.e.\ the $\omega$-image of the $P$-chimney is the chimney in the $\omega$-sheet sitting above it. For any given $y=z\omega \in\eW \cong \aW \rtimes \Omega_G$, the \emph{$(P,y)$-chimney} is defined to be the $y$-image of the $P$-chimney under this action; equivalently, the $z$-image inside the $\omega$-sheet of the chimney sitting above the $P$-chimney.

\begin{figure}[ht]
\centering
\begin{overpic}[width=0.5\textwidth]{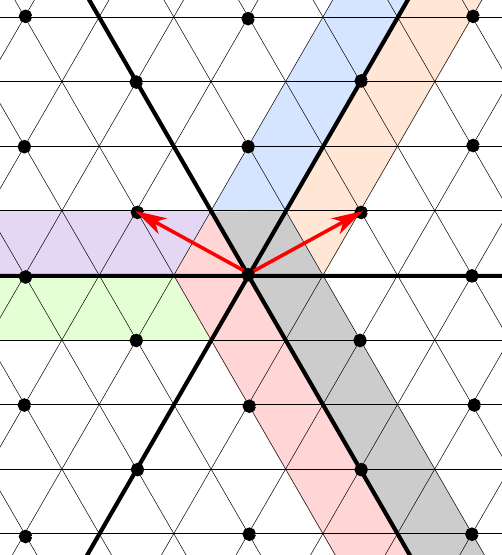}
\put(30,53){\color{red}$\alpha_2^\vee$}
\put(55,53){\color{red}$\alpha_1^\vee$}
\put(43,56){\footnotesize{\color{black}$\fa_0$}}
\put(36,53){\footnotesize{\color{black}$\s_1$}}
\put(50,53){\footnotesize{\color{black}$\s_2$}}
\put(35,46){\footnotesize{\color{black}$\s_1\s_2$}}
\put(49,46){\footnotesize{\color{black}$\s_2 \s_1$}}
\put(43,41){\footnotesize{\color{black}$\w_0$}}
\put(70,33){\footnotesize{\color{black}$(P_1,\id)$}}
\put(35,20){\footnotesize{\color{red}$(P_1,s_1)$}}
\put(11,64){\footnotesize{\color{Purple}$(P_1,s_2)$}}
\put(33,80){\footnotesize{\color{DarkBlue}$(P_1,s_1s_2)$}}
\put(11,35){\footnotesize{\color{DarkGreen}$(P_1,s_2s_1)$}}
\put(70,67){\footnotesize{\color{DarkOrange}$(P_1,w_0)$}}
\end{overpic}
\caption{A sector representing the $(P_1,u)$-chimney, for each $u \in \sW$.}
\label{fig:chimneys}
\end{figure}
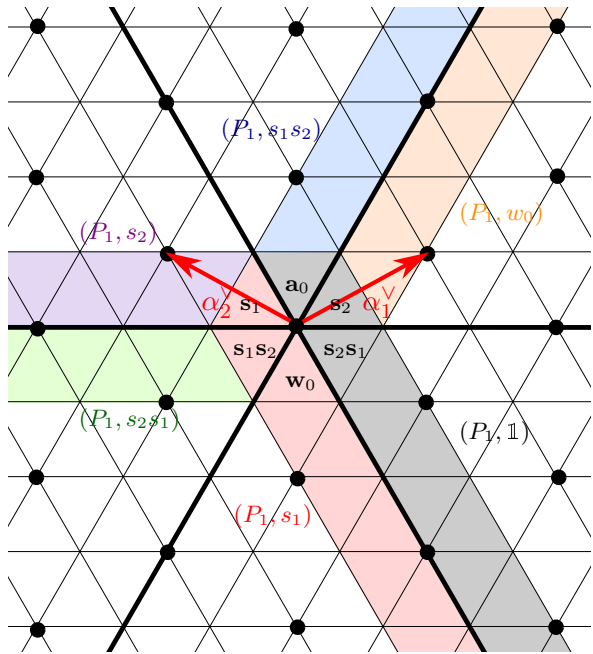

We illustrate the action of $\sW$ on a chimney in Figure \ref{fig:chimneys}. For each $u \in \sW$ we have shaded a sector representing the $(P_1,u)$-chimney. More generally, for any $y = t^\mu u \in \eW$ with $\mu \in Q^\vee$ and $u \in \sW$,  a sector representing the $(P_1,y)$-chimney can be obtained by translating by $\mu$ any sector representing the $(P_1,u)$-chimney.  Thus, every $(P_1,y)$-chimney is represented by a sector which is a translate to the $\omega$-sheet of one of the shaded regions in Figure~\ref{fig:chimneys} for some $\omega \in \Omega_G$.

Every chimney in a given $\omega$-sheet induces an orientation on that sheet as defined in \cite[Def.~3.23]{MNST}.  Formally, an orientation is a map that assigns either $+$ or $-$ to every pair of an extended alcove and a panel contained in it. 
Fix a chimney in the $\omega$-sheet, and denote by $\phi$ the orientation it induces. We extend $\phi$ to all other sheets of $\Omega_G \times \mathfrak{a}$ by assigning to a pair of an alcove in the $\omega'$-sheet and a panel contained in it the same value as for the pair sitting above or below it in the $\omega$-sheet. Figure~\ref{fig:chimneyOrientations} shows the orientation of pairs of extended alcoves and panels for all overlaying sheets together in one picture.
The \emph{opposite standard orientation} is the orientation on $\Omega_G \times \mathfrak{a}$ corresponding to $P=B$. For the rest of this section, we fix an orientation $\phi$ on all sheets of $\Omega_G \times \mathfrak{a}$, which is induced by a particular $(P,y)$-chimney.

A \emph{(combinatorial) gallery $\gamma$} is a sequence of (extended) alcoves $c_i$ and panels $p_i$, contained in $c_i$ and $c_{i-1}$. We say that $\gamma$ \emph{is contained in the $\omega$-sheet} if one (and hence all) of its alcoves is in the $\omega$-sheet.  The \emph{type} of $\gamma$ is the sequence in $\tilde S$ given by the types $s_{j_i}\in \tilde S$ of the panels $p_i$ in $\gamma$. More specifically, a minimal gallery $\gamma:\fa\rsa \x$ contained in the $\omega$-sheet has first alcove $\fa = \fa_\omega$ and end alcove $\x=x\fa_0$ for some $x = z\omega \in \eW \cong \aW \rtimes \Omega_G$. Writing the reduced expression $z=s_{i_1}s_{i_2} \cdots s_{i_k}\in \aW$ corresponding to the types in the 0-sheet of the panels sitting below those of $\gamma$, we have $x=s_{i_1}s_{i_2} \cdots s_{i_k}\omega\in \eW$. The type of $\gamma$ is then denoted by $\vec x = (s_{j_1}, s_{j_2}, \dots, s_{j_k})$, where $j_l = \omega(i_l)$ for all $1 \leq l \leq k$. There is thus a natural action of  $\eW \cong \aW \rtimes \Omega_G$ on the set of all combinatorial galleries. The subgroup $\aW$ acts in a type-preserving way on each individual sheet, while $\Omega_G$ permutes the sheets and acts on the type by the induced permutation action on the set of indices $\tilde I$.

A gallery is \emph{folded at $p_i$} if $c_i=c_{i-1}$, and we say that this fold is \emph{positive (with respect to the $(P,y)$-chimney)} if $c_i=c_{i-1}$ and the orientation $\phi$ assigns $+$  to the pair $(c_i, p_i)$. More generally, a gallery $\gamma$ is \emph{positively folded (with respect to the $(P,y)$-chimney)} if all of its folds are positive with respect to $\phi$. If $c_i \neq c_{i-1}$, the gallery has a \emph{crossing at $p_i$}.  A crossing is \emph{positive (with respect to the $(P,y)$-chimney)} if the orientation $\phi$ assigns $+$ to the pair $(c_i, p_i)$. For further details, see Section 2.2 in \cite{MNST}. 
Positively folded galleries may be constructed using PRS-folding sequences, which the authors introduced in \cite[Def.~3.16]{MST1}. A \emph{PRS-folding sequence} is a sequence of galleries where the next entry in the sequence is obtained from the previous gallery by reflecting a shorter and shorter tail at a particular panel.

We further illustrate these concepts for the group $G=PGL_3$. 
In Figure~\ref{fig:chimneyOrientations}, the shaded region is a sector representing the $P_1$-chimney, and we show using $+$ and $-$ signs the orientation induced by this chimney.  For each hyperplane, the $+$ signs are on the side of the hyperplane facing away from the ``point at infinity" corresponding to the chimney.  We also include in this figure an example of a gallery $\gamma$ which is positively folded with respect to the $P_1$-chimney.  Roughly speaking, this means that the blue gallery $\gamma$ is always ``folded away'' from the chimney.  Viewed as a gallery contained in the 0-sheet, the type of $\gamma$ is $(s_2,s_0,s_1,s_0,s_2,s_0,s_1,s_2,s_0,s_1,s_2,s_0,s_1,s_0,s_2,s_0)$. If we view $\gamma$ as a gallery contained in the $\varpi_1$-sheet based at the red vertex of $\fa_{\varpi_1}$, then the corresponding rotation of $\fa_0$ induces a 3-cycle on the set of indices $\tilde I = \{0,1,2\}$, and the type of $\gamma$ is instead $(s_1,s_2,s_0,s_2,s_1,s_2,s_0,s_1,s_2,s_0,s_1,s_2,s_0,s_2,s_1,s_2)$.


\section{Averaging operators and Newton points with rank 1 parabolics}\label{sec:NewtonPointsReps}

Apart from the special cases reviewed in Example \ref{ex:stdReps}, explicit formulas for  standard representatives of $\sigma$-conjugacy classes have only been recorded for the group $G=GL_m$ in matrix form; see \cite[Ex.~7.2.5]{GHKRadlvs}. Our first proposition provides explicit type-free formulas for standard representatives of $\sigma$-conjugacy classes having associated parabolic subgroups with a Levi subgroup of rank 1. In this case, $\Delta_M = \{\alpha_i\}$ for some $i \in [n]$, and we denote the corresponding \emph{parabolic subgroup of rank 1} by $P_i=MN$. The dimension of the corresponding open chamber $\mathfrak{a}_{P_i}^+$ then equals $n-1$.  

The primary purpose of this section is to prove Proposition \ref{prop:stdReps}, which presents formulas for those standard representatives $b_\nu$ such that $P_\nu = P_i$. In Section \ref{sec:averaging}, we provide a characterization of Newton points by means of an averaging projection,  which then allows us to compute explicit formulas for Newton points associated to rank 1 parabolics in Section~\ref{sec:integralNPRank1}. An explicit description of their standard representatives is then given in Section~\ref{sec:stdReps}. 

\subsection{Averaging projection operators}\label{sec:averaging}

The proof of Proposition \ref{prop:stdReps} requires several lemmas which aim to characterize the Newton points associated to rank 1 parabolic subgroups.  For this purpose, we first reformulate the method for calculating the Newton point of an element of the extended affine Weyl group as applying a certain projection operator.

\begin{definition}\label{defn:averaging}
Let $w \in \sW$ be an element of order $m$.  The \emph{averaging projection for $w$} is the linear operator on $\mathfrak{a}$ defined by
\[A_w := \frac{1}{m}\left(I+w+w^2 +\cdots + w^{m-1}\right), \]
where $I$ denotes the identity operator.
\end{definition}

For any $x = t^\lambda w \in \eW$, note by \eqref{eq:avgFormula} that  $\nu_G(x) = (A_w \lambda)^+$, which motivates the study of this family of operators.  We begin with a lemma relating the Newton point of a $\sigma$-conjugacy class to the image of an averaging operator determined by its standard representative.

\begin{lemma}\label{lem:stdRepAvDom} Let $[b] \in B(G)_P$, and denote the corresponding Newton point by $\nu = \nu_G(b)$. If $b_\nu = t^\eta v$ denotes the standard representative for $[b]$, then $\nu = A_v \eta$.  
\end{lemma}

\begin{proof} 
Given a Levi decomposition $P = MN$, we have $b_\nu = t^\eta v$ with $v \in W_M$. Therefore, $\nu_M(b_\nu) = (A_v\eta)^+_M$, where $(A_v \eta)^+_M$ denotes the unique element in the $W_M$-orbit of $A_v \eta$ which is $M$-dominant.  However, recall from \eqref{eq:StdRepMaps} that $b_\nu \in \Omega_M^+$ also satisfies $\nu_M(b_\nu) \in \mathfrak{a}_P^+$ and is thus $G$-dominant.  Therefore, $\nu_M(b_\nu) = (A_v\eta)^+_M = A_v \eta$.  Finally, since $v \in W_M$, note that $\nu_G(b_\nu) = \nu_M(b_\nu)$, and the result follows.
\end{proof}

Next, we gather several easy properties of averaging projections, including justifying the terminology.

\begin{lemma}\label{lem:AwProperties}
The averaging projection $A_w$ for any $w \in \sW$ satisfies the following:
\begin{enumerate}
    \item $A_w w = A_w$.
    \item $u A_w u^{-1} = A_{uwu^{-1}}$ for any $u \in \sW$.
    \item $A_w$ is a projection.
    \item If $w$ is a reflection, then $A_w$ is the orthogonal projection onto the hyperplane fixed by~$w$.
    \item $\Range(I-w) \subseteq \Ker(A_w)$.
    \item $\Range(A_w) \subseteq \Ker(I-w)$.
\end{enumerate}
\end{lemma}

\begin{proof}
Properties (1) and (2) are clear from the definition of $A_w$, and property (3) follows from (1). 
For (4), by definition $w = us_i u^{-1}$ for some $s_i, u \in \sW$.  Hence by (2), we have $A_w = u A_{s_i} u^{-1}$.  The group $\sW$ acts by isometries on $\mathfrak{a}$ and conjugates of orthogonal projections by isometries are orthogonal projections, so it suffices to show that $A_{s_i}$ is the orthogonal projection onto $H_{\alpha_i}$.  Now for any point $v$ of $\mathfrak{a}$, we have 
\[
A_{s_i}v = \frac{1}{2}(v + s_iv) = v - \frac{1}{2}(v - s_i v)
\]
which is the orthogonal projection onto $H_{\alpha_i}$.
For properties (5) and (6), note that $A_w(I-w)= (I-w)A_w = 0$.
\end{proof}

The orthogonal projections of the form $A_{s_i}$ will play an especially important role, so we also occasionally use the more intuitive notation $A_{s_i} = \proj_i$ to denote the orthogonal projection of the standard apartment $\mathfrak{a}$ onto the (simple) hyperplane $H_{\alpha_i}$.

Next, we prove using the properties in Lemma~\ref{lem:AwProperties}  above that the two containments in (5) and (6) are in fact equalities.

\begin{corollary}\label{cor:AwKerRange}
Let $w \in \sW$.  Then
\begin{align*}
\Ker(A_w) &= \Range(I-w), \ \text{and} \\
\Range(A_w) &= \Ker(I-w).
\end{align*}
\end{corollary}

\begin{proof}
First, compute directly using Lemma \ref{lem:AwProperties}(6) that $\Ker(I-w) \subseteq \Ker(I-A_w)$. Recalling from Lemma \ref{lem:AwProperties}(3) that $A_w$ is a projection operator, we also know that $\Ker(I-A_w) = \Range(A_w)$. Combining these observations with Lemma \ref{lem:AwProperties}(6) again, we have
\[\Range(A_w) \subseteq \Ker(I-w) \subseteq \Ker(I-A_w) = \Range(A_w). \]
In particular, $\Range(A_w) = \Ker(I-w)$, which proves the second desired equality.  Combining this equality with the Rank-Nullity Theorem applied twice to the pair of operators $A_w$ and $I-w$, we deduce that $\Ker(A_w)$ and $\Range(I-w)$ have the same dimension.  Applying Lemma \ref{lem:AwProperties}(5) then yields the first equality.
\end{proof}

\subsection{Formulas for Newton points with rank 1 associated parabolics}\label{sec:integralNPRank1}

In this section, we aim to explicitly describe those Newton points which have an associated parabolic subgroup of rank 1.  Using the results on the averaging operators from the previous section, we prove the following lemma, which provides several concrete formulas for such Newton points.

\begin{lemma}\label{lem:NPimageAsi}
Let $[b]\in B(G)_P$, and denote the corresponding Newton point by $\nu = \nu_G(b)$ and Kottwitz point by $\omega = \kappa_G(b)$.  If $P=P_i$ is a rank 1 parabolic subgroup, then 
\begin{equation}
\nu = A_{s_i}\lambda = \proj_i(\lambda)    
\end{equation}
for some $\lambda \in Q^\vee$ such that $\kappa_G(\lambda)=\omega$. More explicitly, for any such $\lambda \in Q^\vee$, 
\begin{equation}\label{eq:nuProj}
    \nu = \lambda - \frac{\langle \alpha_i, \lambda \rangle}{2}\alpha_i^\vee,
\quad \text{where} \quad
    \langle \alpha_i, \lambda \rangle = 
    \begin{cases}
    2k & \text{if $\nu$ is integral},\\
    2k+1 & \text{if $\nu$ is non-integral},
    \end{cases}
\end{equation}
for some $k \in \Z$. Equivalently, for any such $\lambda \in Q^\vee$, if we denote by $m_j = \langle \alpha_j, \lambda \rangle$, then
\begin{equation}\label{eq:nuCowt}
    \nu = \sum\limits_{j\neq i} \left(m_j - \frac{m_ic_{ji}}{2} \right)\varpi_j,
\end{equation}
where $C = (c_{ij})$ is the Cartan matrix.
\end{lemma}

\begin{proof}
By \eqref{eq:avgFormula}, the definition of the averaging operator $A_w$, and the fact that the elements of $\eW$ represent $B(G)$, we must have $\nu = \nu_G(t^\lambda w)= (A_w \lambda)^+$ for some $\lambda \in Q^\vee$ and $w \in \sW$. Further, in order that $t^\lambda w \in \eW$ represents $[b] \in B(G)_P$, we must also have $\kappa_G(\lambda) = \omega$.  Since $\nu = (A_w \lambda)^+$, then $\nu = u\left(A_w \lambda \right)$ for some $u \in \sW$. By part (4) of Lemma \ref{lem:AwProperties}, we thus see that $\nu = A_{uwu^{-1}}(u\lambda)$.  Since $u\lambda \in Q^\vee$ also satisfies $\kappa_G(u\lambda) = \kappa_G(\lambda) = \omega$, we thus assume without loss of generality that $\nu = A_w \lambda$, meaning that $A_w \lambda$ is already dominant.

Now suppose that $P$ is a rank 1 parabolic subgroup, so that $P = P_i$ for some $i \in [n]$.  By \eqref{eq:PnuRoots}, we have $\nu \in H_{\alpha_i}$, but $\nu \notin H_{\alpha_j}$ for any $j\neq i$.  Equivalently, 
\begin{align}\label{eq:rank1ineqs}
    \langle \alpha_i, \nu \rangle &= 0,\ \text{and} \nonumber \\
    \langle \alpha_j, \nu \rangle &>0\phantom{,} \ \text{for}\ j \neq i.
\end{align}
Since $\nu = A_w\lambda$, we have $\nu \in \Range(A_w)$ and so by Corollary \ref{cor:AwKerRange} we also have $\nu \in \Ker(I-w)$.  In particular,  $\nu$ is fixed by the action of $w$. Combining this observation with \eqref{eq:rank1ineqs} and the $\sW$-invariance of the evaluation pairing, we have
\begin{align*}
   \langle w\alpha_i, \nu \rangle &= 0,\ \text{and} \\
   \langle w\alpha_j, \nu \rangle &>0\phantom{,} \ \text{for}\ j \neq i. \nonumber
\end{align*}
Therefore, we must have $w\alpha_i = -\alpha_i$ and so in fact $w=s_i$.  In words, $\nu$ is simply the image of the orthogonal projection of $\lambda$ onto $H_{\alpha_i}$. Altogether, we have thus shown that there exists $\lambda \in Q^\vee$ such that  $A_{s_i}\lambda = \nu$ and $\kappa_G(\lambda) = \omega$.  Throughout the remainder of the proof, fix any such $\lambda \in Q^\vee$. 

Using that $s_i(\lambda) = \lambda - \langle \alpha_i, \lambda \rangle \alpha_i^\vee$, we compute directly that 
\begin{equation}\label{eq:nuProjv1}
    \nu = A_{s_i}\lambda= \frac{1}{2}\left(\lambda+s_i \lambda \right)= \lambda - \frac{\langle \alpha_i, \lambda \rangle}{2}\alpha_i^\vee,
\end{equation}
which proves the first equality in \eqref{eq:nuProj}.  Dualizing Equation (14) from \cite[Ch.~VI $\S$1.10]{Bourbaki4-6}, the simple coroots can be written in terms of the basis of fundamental coweights as  
\begin{equation}\label{eq:cortcowt}
\alpha_i^\vee = \sum c_{ji}\varpi_j,     
\end{equation}
where $C = (c_{ij})$ is the Cartan matrix. Write $\lambda = \sum m_j \varpi_j$ in terms of the basis of fundamental coweights, where $m_j = \langle \alpha_j, \lambda \rangle$ for all $j \in [n]$ by duality. Using the fact that $c_{ii}=2$, we thus see from \eqref{eq:nuProjv1} and \eqref{eq:cortcowt} that
\begin{equation*}
    \nu = \sum_{j=1}^n m_j \varpi_j - \frac{m_i}{2}\sum_{j=1}^n c_{ji}\varpi_j = \sum\limits_{j\neq i} \left(m_j - \frac{m_ic_{ji}}{2} \right)\varpi_j,
\end{equation*}
confirming \eqref{eq:nuCowt}.

Finally, we consider the parity of $\langle \alpha_i, \lambda \rangle$.  By \eqref{eq:nuProjv1}, all $\lambda \in Q^\vee$ such that $A_{s_i}\lambda = \nu$ are of the form \begin{equation}\label{eq:lambdaForm}
    \lambda = \nu+\frac{m}{2}\alpha_i^\vee
\end{equation} for some $m \in \Z$.   First suppose that $\nu$ is integral, so that $\nu \in Q^\vee$.  Since $\lambda \in Q^\vee$ as well, \eqref{eq:lambdaForm} implies that $m$ must be even in this case. Conversely, if $m$ is even, then $\lambda$ and $\nu$ differ by an element of $R^\vee$, and so $\kappa_G(\nu) = \kappa_G(\lambda) = \omega$. In addition, if $m$ is even, then \eqref{eq:lambdaForm} implies that $\nu \in Q^\vee$.  Therefore, if $m$ is even, then $\nu$ is integral in the sense of Definition \ref{defn:integralNP}. Altogether, we have thus proved that $m$ is even if and only if $\nu$ is integral, confirming the parity statement in \eqref{eq:nuProj} and concluding the proof.
\end{proof}

\subsection{Standard representatives for Newton points with rank 1 parabolics}\label{sec:stdReps}

We are now prepared to prove the following proposition concerning standard representatives for $\sigma$-conjugacy classes having Newton points with associated parabolic subgroups of rank 1.

\begin{prop}\label{prop:stdReps} Let $[b] \in B(G)_P$.  Denote by $\nu = \nu_G(b)$ the Newton point, and let $b_\nu$ denote the standard representative for $[b]$.  
\begin{enumerate}
\item If $P = P_i$ and $\nu$ is integral, then $b_\nu$ is the translation $t^\nu$.
\item If $P = P_i$ and $\nu$ is non-integral, then $b_\nu = t^\eta s_i$ where 
\[
\eta = \nu+ \frac{1}{2}\alpha_i^\vee.
\]
Equivalently, $\eta$ is the unique element of $H_{\alpha_i,1} \cap Q^\vee$ such that \begin{equation*}
\nu = A_{s_i}\eta=\proj_i (\eta).
\end{equation*}
\end{enumerate}
\end{prop}

\begin{proof} 
Recall that $\nu = \nu_G(b)$ is the Newton point of $[b] \in B(G)_P$, and denote the Kottwitz point by $\omega = \kappa_G(b)$. Throughout the proof, we assume that $P=P_i$ is a rank 1 parabolic subgroup, in which case the Levi subgroup $M$ is isomorphic to $\mathbb{G}_m^{i-1}\times G_1 \times \mathbb{G}_m^{n-i}$, where $G_1$ is a rank 1 group of type $A_1$. Since $\Delta_M = \{\alpha_i\}$, then $\Omega_M$ is generated by the translations $t^{\varpi_j}$ for all $j \neq i$, together with the element $\tau_i :=t^{\varpi_i}s_i$.  Moreover, the action of $s_i$ is trivial on all $t^{\varpi_j}$ such that $j \neq i$.

We first consider the case where $\nu$ is integral. We need only verify that $t^\nu$ satisfies the three properties from Lemma \ref{lem:stdRepList}.  Recall by Definition \ref{defn:integralNP} that $\nu$ has the same image as $\omega$ in $Q^\vee/R^\vee$ since $\nu$ is integral, and so $\kappa_G(t^\nu) = \omega$, verifying property (1) for $b_\nu = t^\nu$.  Since $\nu = \sum_{j\neq i} \nu_j \varpi_j \in Q^\vee$ for some $\nu_j \in \Z$ by \eqref{eq:nuCowt}, we see that $t^\nu = (t^{\varpi_1})^{\nu_1}\cdots \id \cdots (t^{\varpi_n})^{\nu_n}$, where $\id$ occurs in the $i^{\text{th}}$ coordinate.  In particular, $t^\nu \in \Omega_M$, confirming property (2) of Lemma \ref{lem:stdRepList}.  Finally, using \eqref{eq:avgFormula}, since $t^\nu$ is a translation and $\nu \in \mathfrak{a}_P^+$, we have $\nu_M(t^\nu) = \nu_G(t^\nu) = \nu$, confirming property (3).  Therefore, $b_\nu=t^\nu$ is the standard representative for $[b]$ in the case where $\nu$ is integral.

Now consider the case where $\nu$ is non-integral. By Lemma \ref{lem:NPimageAsi}, we have $\nu = A_{s_i}\lambda$ for some $\lambda \in Q^\vee$ such that both $\kappa_G(\lambda) = \omega$ and $\langle \alpha_i, \lambda \rangle = 2k+1$ for some $k \in \Z$. Consider $\eta = \lambda-k\alpha^\vee_i$.  Then $A_{s_i}\eta = \nu$ as well, since $A_{s_i}\lambda = \nu$ and $\lambda$ and $\eta$ differ by a multiple of $\alpha_i^\vee$.  In addition, 
\begin{equation}\label{eq:eta1}
    \langle \alpha_i, \eta \rangle = \langle \alpha_i, \lambda - k\alpha_i^\vee \rangle = \langle \alpha_i, \lambda \rangle - k\langle \alpha_i, \alpha_i^\vee \rangle = (2k+1)-k(2) = 1,
\end{equation}
so that $\eta \in H_{\alpha_i,1} \cap Q^\vee$. In particular, $\eta = \nu+\frac{1}{2}\alpha_i^\vee.$  In addition, note that by definition, 
\begin{equation}\label{eq:kappaEta}
    \kappa_G(\eta) = \kappa_G\left(\lambda - k\alpha_i^\vee \right) = \kappa_G(\lambda).
\end{equation}

Now define $b_\nu = t^\eta s_i$.  By \eqref{eq:kappaEta} and our choice of $\lambda$, we have that $\kappa_G(b_\nu) = \kappa_G(\eta) = \kappa_G(\lambda) = \omega,$ verifying property (1) of Lemma \ref{lem:stdRepList}.  Recalling \eqref{eq:cortcowt} and applying  \eqref{eq:nuCowt} to $\eta$, we have
\begin{equation*}
    \eta = \nu+\frac{1}{2}\alpha_i^\vee =  \sum\limits_{j\neq i} \left(m_j - \frac{m_ic_{ji}}{2} \right)\varpi_j + \frac{1}{2} \left( \sum\limits_{j=1}^n c_{ji}\varpi_j\right) = \varpi_i+ \sum\limits_{j\neq i}m_j\varpi_j,
\end{equation*}
where $m_j = \langle \alpha_j, \eta \rangle$ and $m_i=1$ by \eqref{eq:eta1}. 
Therefore, $b_\nu = t^\eta s_i = (t^{\varpi_1})^{m_1}\cdots \tau_i \cdots (t^{\varpi_n})^{m_n} \in \Omega_M$ since $s_i$ acts trivially on $t^{\varpi_j}$ for all $j \neq i$, confirming property (2).  Finally, $\nu_M(b_\nu) = A_{s_i}\eta = \nu$, so that $b_\nu$ satisfies property (3). By Lemma \ref{lem:stdRepList}, we have that $b_\nu=t^\eta s_i$ is the standard representative for $[b]$ in the case where $\nu$ is non-integral.
\end{proof}

The next result concerns the location of $\eta$ where $b = t^\eta s_i$ is a standard representative.  We will use this result in Section~\ref{sec:remainingTargets}.

\begin{corollary}\label{cor:stdRepDominant}  Let $\nu$ be any non-integral Newton point with associated spherical standard parabolic subgroup $P_i$, and let $b_\nu = t^\eta s_i$ be its standard representative.   Then $\eta \in Q^\vee$ is in the dominant Weyl chamber $\Cf$, except if $G$ is of type $\tilde G_2$ and $\nu = \frac{1}{2}\alpha_1^\vee + \alpha_2^\vee$.
\end{corollary}

\begin{proof}  Since $\eta \in Q^\vee$, we may write $\eta = \sum n_k \varpi_k$ where $n_k = \langle \alpha_k, \eta\rangle \in \Z$ by duality.  We thus aim to prove that $n_k \geq 0$ for all $k$, outside of the unique case identified in the statement.  By Proposition~\ref{prop:stdReps}, we have $\eta \in H_{\alpha_i,1}$, and so indeed $n_i =1 \geq 0$.

We now consider any $j \neq i$.  Applying  \eqref{eq:nuCowt} to $\lambda = \eta$, we have
\begin{equation*}
   \nu =  \sum\limits_{j\neq i} \left(n_j - \frac{c_{ji}}{2} \right)\varpi_j.
\end{equation*}
Since $\langle \alpha_j, \nu \rangle  = n_j - \frac{c_{ji}}{2} >0$ by \eqref{eq:rank1ineqs}, then for $c_{ji} \in \{0,-1,-2\}$, we have $n_j  = \langle \alpha_j, \eta \rangle \geq 0$ because $n_j \in \Z$.  In all of these cases, $\eta$ is thus dominant as required.
 
The remaining possibility is $c_{ji} = -3$, in which case $G$ is type $G_2$ with $j = 2$ and $i = 1$.  In 
this case we have $\nu = (n_2 + \frac{3}{2})\varpi_2$, and so by \eqref{eq:rank1ineqs}, we have $n_2 + \frac{3}{2} >0$.  Since $n_2 \in \Z$, then this inequality implies that $n_2 \geq -1$.  For $n_2 \geq 0$, we have $\langle \alpha_2, \eta \rangle \geq 0$, and so $\eta$ is dominant. The only exception is $n_2 = -1$, in which case $\nu = \frac{1}{2}\varpi_2 = \frac{1}{2}\alpha_1^\vee + \alpha_2^\vee$ via \eqref{eq:cortcowt}, as required.
\end{proof}

We now complete our discussion of Figure \ref{fig:C2StdReps} from Section \ref{sec:NewtonMap} illustrating Proposition \ref{prop:stdReps}.

\begin{example}\label{ex:C2stdreps}
Figure \ref{fig:C2StdReps} illustrates the four families of standard representatives in type $\tilde{C}_2$ obtained from Proposition \ref{prop:stdReps}. The parabolic $P=P_1$ corresponds to the Newton points in the $\Z_{>0}$-span of $\varpi_2$ highlighted in green, all of which are contained in the hyperplane $H_{\alpha_1}$.  The dark blue alcoves correspond to standard representatives having Kottwitz point $0 \in Q^\vee/R^\vee$, whereas elements for the teal alcoves have Kottwitz point $\varpi_2 \in Q^\vee/R^\vee$. The Newton point $\nu$ is integral for precisely those alcoves whose translation part lies on the hyperplane $H_{\alpha_1}$; these translation alcoves are depicted as teal and dark blue in alternation.  The standard representatives for the non-integral Newton points having $P=P_1$ are similarly shaded in alternating dark blue and teal, though now along the affine hyperplane $H_{\alpha_1,1}$, and located in the $s_1$-spherical position. Note that each green Newton point along the hyperplane $H_{\alpha_1}$ is represented by two distinct $\sigma$-conjugacy classes, one for each element of $Q^\vee/R^\vee$. 

There are again two families of standard representatives corresponding to Newton points such that $P=P_2$.  The pink alcoves along $H_{\alpha_2}$ depict the translation alcoves representing the integral Newton points highlighted in dark pink.  The orange alcoves based on $H_{\alpha_2,1}$ are the standard representatives for the non-integral Newton points highlighted in dark orange.  For $P=P_2$, each Newton point along $H_{\alpha_2}$ represents a unique $\sigma$-conjugacy class, in contrast to the situation for $P=P_1$. This distinction will play an important role later in Lemma \ref{lem:BCintegral}.
\end{example}


\section{Affine Deligne--Lusztig varieties and labeled folded alcove walks}\label{sec:alcoveWalks}

This section provides a precise connection between galleries which are positively folded with respect to a chimney and affine Deligne--Lusztig varieties in the affine flag variety.  The main result is Theorem \ref{thm:ADLVChimneys}, restated and then proved in Section \ref{sec:dimGallery}, which takes the problems of determining nonemptiness and calculating dimensions from \cite{GHKRadlvs} summarized here as Theorem \ref{thm:GHKRThm}, and translates these questions into the language of positively folded galleries with respect to a chimney.  We then briefly illustrate in Section \ref{sec:integral} how to use the folded galleries constructed in \cite{MST1} to reprove nonemptiness for regular integral Newton points as a special case of Theorem \ref{thm:ADLVChimneys}.

\subsection{Dimension formulas and folded galleries}\label{sec:dimGallery}

The orbits which appear in Theorem \ref{thm:GHKRThm} are precisely those which were studied by the authors in recent joint work with Naqvi on retractions based at chimneys. We thus continue by reviewing one of the main theorems from \cite{MNST}, which permits the translation of Theorem \ref{thm:GHKRThm} into gallery combinatorics. 

We remark that the proofs in \cite{MNST} equally apply to all alcoves in the same fixed $\omega$-sheet of extended alcoves for any $\omega \in \Omega_G$.  Since a necessary condition for the intersection $Ix\fa_0 \cap I_P^y b^y\fa_0$ to be nonempty is that the alcoves $\x$ and $\bb^y$ are in the same sheet of extended alcoves, we may freely state Theorem \ref{thm:DoubleCosetIntersection} in this slightly greater level of generality.

\begin{thm}[\cite{MNST}, Thm.~1.2]\label{thm:DoubleCosetIntersection} Let $b, x, y \in \eW$. Let $P$ be a standard spherical parabolic subgroup of $G(F)$.  There is a bijection between the points of the intersection 
\[Ix\fa_0 \cap I_P^y b^y \fa_0\]
and the set of labeled folded alcove walks of type $\vec{x}$ from $\fa_0$ to $\bb^y$ which are positively folded with respect to the $(P,y)$-chimney.
\end{thm}

We proceed to generalize the discussion of \cite[Sec.~4]{MST1} by defining the dimension of a gallery which is positively folded with respect to a chimney. For any gallery $\gamma$ which is positively folded with respect to the $(P,y)$-chimney, denote by $p_P^y(\gamma)$ the number of positive crossings, and by $f_P^y(\gamma)$ the number of folds in the gallery $\gamma$.

\begin{definition}\label{defn:GalleryDim}
Let $P$ be a standard spherical parabolic subgroup of $G(F)$, and let $y \in \eW$.  Given any gallery $\gamma$ which is positively folded with respect to the $(P,y)$-chimney, the \emph{dimension of $\gamma$ with respect to the $(P,y)$-chimney}, denoted $\dim_P^y(\gamma)$, is defined to be
\[ \dim_P^y(\gamma) = p_P^y(\gamma) + f_P^y(\gamma). \]
When the choice of chimney is clear, we abbreviate these statistics as $\dim(\gamma), p(\gamma),$ and $f(\gamma)$, respectively.
\end{definition}

Roughly speaking, the crossings in a gallery which move in a direction ``away'' from the chimney are those which contribute to the dimension, as the next example illustrates.

\begin{example}
Consider the blue gallery $\gamma$ in Figure~\ref{fig:chimneyOrientations}, which is positively folded with respect to the $P_1$-chimney shaded in grey.  Before the first fold, only the 2 crossings of hyperplanes parallel to $H_{\alpha_1}$ are positive.  By contrast, all 7 crossings after the first fold are positive crossings, each of which takes a step ``away'' from the chimney.  Together with the 2 folds, we thus see that $\dim(\gamma) = p(\gamma)+f(\gamma)= 9+2=11.$ Moreover, this dimension calculation is identical independent of the sheet of extended alcoves containing $\gamma$.
\end{example}

The choice of terminology is justified by the following lemma, which says that the dimension of a gallery which is positively folded with respect to the $(P,y)$-chimney coincides with the dimension of the intersection of the Iwahori and $I_P^y$-orbits occurring in Theorem \ref{thm:DoubleCosetIntersection}.

\begin{lemma}\label{lem:GalleryDim}
Let $b,x,y \in \eW$. Let $P$ be a standard spherical parabolic subgroup of $G(F)$.  If $Ix\fa_0 \cap I_P^y b^y \fa_0\neq \emptyset$, then
\begin{equation}\label{E:GalleryDim}
\dim \left( Ix\fa_0 \cap I_P^y b^y \fa_0 \right) = \max\{ \dim_P^y (\gamma) \},
\end{equation}
where $\gamma:\fa \rsa \bb^y$ is a gallery of type $\vec{x}$ which is positively folded with respect to the $(P,y)$-chimney.
\end{lemma}

\begin{proof}
The proof of \cite[Prop.~5.9]{MNST} shows that the chimney orientations correspond to a stratification on the intersection $Ix\fa_0 \cap I_P^y b^y \fa_0$ by products of the form $\mathbb{A}^{p(\gamma)} \times (\mathbb{A}\backslash \{0\})^{f(\gamma)}$.  The same argument as in the proof of \cite[Prop.~5.6]{MST1} thus applies, replacing $U^-$ by $I_P$.
\end{proof}

We now repeat Theorem~\ref{thm:ADLVChimneys} from the introduction and give a proof. 

\ADLVChimneyRestate*

\begin{proof}  Combine part (1) of Theorem \ref{thm:GHKRThm} with Theorem~\ref{thm:DoubleCosetIntersection} for the nonemptiness statement. Combine part (2) of Theorem \ref{thm:GHKRThm} with Lemma \ref{lem:GalleryDim} to obtain the dimension formula.
\end{proof}

\subsection{Integral Newton points}\label{sec:integral}

In this section we briefly explain how to apply results from~\cite{MST1} to prove Theorem~\ref{thm:w0ShrunkenDominant} in the case that the Newton point $\nu_b$ is integral. We mention that the vector $\check{\rho}$ was mislabeled as $\rho$ throughout \cite{MST1}, and we correct this error in the statement of Theorem \ref{thm:translations} below; all statements in \cite{MST1} remain true after this relabeling. Since we also have a dimension formula from \cite{MST1}, we prove the following.

\begin{thm}[\cite{MST1}]\label{thm:translations}   Let $b = t^\mu \in \eW$, and let $\nu_b$ be the Newton point of the $\sigma$-conjugacy class $[b]$. Let $x_0=t^\lambda w_0 \in \eW$ be such that $\kappa_G(x_0) = \kappa_G(b)$. Assume that:
\begin{enumerate}
    \item $\x_0$ is contained in $\Cfs$; and 
    \item $\nu_b \in \Conv( \sW (\lambda - 2\check{\rho}) ).$
\end{enumerate}
Then 
$X_{x_0}(b) \neq \emptyset$, and $\dim X_{x_0}(b) = \langle \rho, \lambda - \nu_b \rangle$.
\end{thm}

\begin{proof}  By $\sigma$-conjugating (equivalently conjugating) the given $b$ by an element of $\sW$, we may assume without loss of generality that $\mu$ is contained in the dominant Weyl chamber $\Cf$, so that $\nu_b = \mu$. The intersection of $\Cf$ with the polytope $\Conv( \sW (\lambda - 2\check{\rho}) )$ is equal to the intersection of $\Cf$ with the negative cone based at $\lambda - 2\check{\rho}$. Since the Kottwitz points of $x_0$ and $b$ agree, we may consider galleries in a single fixed sheet of extended alcoves.  Therefore, by \cite[Prop.~8.7]{MST1}, for all $\mu$ in this intersection, we have $X_{x_0}(b) \neq \emptyset$ and $\dim X_{x_0}(b) = \langle \rho, \lambda - \mu \rangle = \langle \rho, \lambda - \nu_b \rangle$.  
\end{proof}

\begin{remark}
If $\nu_b = \mu$ also lies in the shrunken dominant Weyl chamber $\Cfs$, then the result of Theorem \ref{thm:translations} can be obtained as an application of Theorem \ref{thm:ADLVChimneys}.  Recall from Example \ref{ex:Pnu} that if the Newton point $\nu$ for $[b] \in B(G)$ is integral and regular, then $P_\nu = B$.  By Example \ref{ex:stdReps}, the standard representative $b_\nu = t^\nu$ in this case, and the gallery $\gamma_0: \fa \rsa \bb_\nu^{w_0}$ from \cite[Prop.~6.7]{MST1}, constructed to prove the nonemptiness and dimension conclusions in \cite[Prop.~8.7]{MST1} cited in the proof above, has type $\vec{x}_0$ and is positively folded with respect to the $(B,w_0)$-chimney.  In particular, the dominant Weyl chamber is a representative sector for this chimney, and the gallery $\gamma_0$ is folded away from the chimney, as required.
\end{remark}


\section{Conjugacy classes of standard representatives}\label{sec:conjugacy}

In this section we consider $\eW$-conjugacy classes of standard representatives for $\sigma$-conjugacy classes, motivated by the conjugates appearing in Theorem~\ref{thm:ADLVChimneys}. Figure~\ref{fig:conjugates}, repeated from the introduction here with more detailed annotation, illustrates an example of the definitions and results of this section in type $\tilde{A}_2$. In Section~\ref{sec:fillOut}, we give a precise description of the $\eW$-conjugacy classes of standard representatives for non-integral Newton points with a rank 1 parabolic subgroup. This description is based on the notion of transverse subspaces introduced in Section~\ref{sec:transverse}.

\subsection{Transverse subspaces and conjugacy classes} \label{sec:transverse}

For each Newton point $\nu$, we now introduce a collection of associated affine subspaces of the standard apartment $\mathfrak{a}$.  We then relate these subspaces to the $\eW$-conjugacy class of the standard representative $b_\nu$.  

\begin{definition}[Transverse subspaces]\label{defn:transverse}  Let $[b] \in B(G)_P$, and denote by $\nu$ the corresponding Newton point and by $b_\nu = t^\eta v$ the standard representative.  The affine subspace of $\mathfrak{a}$  given by 
\[
\cT_\nu := \nu + \Ker A_v
\]
is called \emph{the subspace transverse to $\nu$}.  For any $u \in \sW$, the affine subspace $u\cT_\nu$ is called \emph{a transverse subspace of $\nu$}.  
\end{definition}

In the special cases we treat in this paper, the next two results provide a concrete description of these transverse subspaces. See Figure~\ref{fig:conjugates} for an example of transverse subspaces in parabolic rank $1$, illustrating Corollary \ref{cor:nonintegralTransverse} below.

\begin{corollary}\label{cor:integralTransverse}  If either $[b_\nu] \in B(G)_P$ with $P=B$, or $[b_\nu] \in B(G)_P$ with $P = P_i$ and $\nu$ integral, then the subspace transverse to $\nu$ is the single point $\cT_\nu = \{\nu\}$.
\end{corollary}

\begin{proof}
By Example~\ref{ex:Pnu}(1) and Proposition~\ref{prop:stdReps}(1), respectively, in these two cases $b_\nu = t^\nu$.  Thus $A_v = A_\id = \Id$, and so $\Ker A_v$ is trivial.
\end{proof}

\begin{corollary}\label{cor:nonintegralTransverse}
Suppose $[b_\nu] \in B(G)_P$ with $P = P_i$ and $\nu$ non-integral, and write  $b_\nu = t^\eta s_i$ as in Proposition~\ref{prop:stdReps}(2).  Then 
\[
\cT_\nu = \nu + \{ r \alpha_i^\vee \mid r \in \R \}.
\]
\end{corollary}

\begin{proof}
The result follows from Definition \ref{defn:transverse} and the fact that $\Ker A_{s_i} = \{ r \alpha_i^\vee \mid r \in \R \}$, since by Lemma~\ref{lem:AwProperties}(4), the operator $A_{s_i}$ is the orthogonal projection onto $H_{\alpha_i}$.
\end{proof}

For general Newton points, we observe the following.

\begin{corollary}\label{cor:stdRepTransverse}
Let $[b] \in B(G)_P$, and denote by $\nu$ the Newton point and by $\omega$ the Kottwitz point.  Let $b_\nu = t^\eta v$ be the standard representative for $[b]$.  Then $\eta \in \cT_\nu \cap (\omega+R^\vee)$.
\end{corollary}

\begin{proof}
First recall by property (1) of the standard representative from Lemma \ref{lem:stdRepList} that $\kappa_G(b_\nu) = \kappa_G(\eta) = \omega$, and so $\eta \in \omega+R^\vee$.  By definition, $\eta \in \cT_\nu$ if and only if $\eta - \nu \in \Ker A_v$.  Thus we want to show that $A_v \eta = A_v \nu$. We have $A_v \eta = \nu$ by Lemma~\ref{lem:stdRepAvDom}, so it suffices to prove that $A_v$ fixes $\nu$.  Let $P = MN$ be the parabolic subgroup associated to $\nu$.  Then $\nu \in \cap_{\alpha_i \in \Delta_M} H_{\alpha_i}$.  Now $v \in W_M$, and it follows that $v$ fixes $\nu$. From the definition of $A_v$, it is now easy to see that $A_v$ fixes $\nu$, as desired.
\end{proof}

We now explain the connection between transverse subspaces and $\eW$-conjugacy classes of standard representatives.

\begin{lemma}\label{lem:containedTransverse} Let $[b] \in B(G)_P$, and denote by $\nu$ the Newton point and by $\omega$ the Kottwitz point.  Let $b_\nu = t^\eta v$ be the standard representative for $[b]$. Then for any $y = t^\mu u \in \eW$ with $\mu \in Q^\vee$ and $u \in \sW,$ the conjugate $b_\nu^y$ is of the form $t^\xi u v u^{-1},$ where $\xi \in (u\cT_\nu) \cap (\omega + R^\vee)$. 
\end{lemma}

\begin{proof}  First compute that $b_\nu^y = t^\xi u v u^{-1}$
where
\[
\xi = \mu + u \eta - u v u^{-1} \mu.
\]
Recall from Corollary \ref{cor:stdRepTransverse} that $\eta \in \omega + R^\vee$.  Since the action of $\sW$ on $\mathfrak{a}$ preserves $R^\vee$, then $\xi \in \omega + R^\vee$ as well, independent of $\mu \in Q^\vee$.

It remains to show that $\mu + u \eta - u v u^{-1} \mu$ lies in $u\cT_\nu$.  Multiplying on the left by $u^{-1}$ and replacing $u^{-1}\mu$ by $\mu'$, it is then enough to show that for all $\mu' \in R^\vee$, we have $\mu' + \eta - v \mu' \in \cT_\nu$.
Now by Corollary~\ref{cor:stdRepTransverse}, we have $\eta \in \cT_\nu$.  It follows that $\nu + \Ker A_{v} = \eta + \Ker A_{v}$.  It now suffices to show that $\mu' - v\mu' \in \Ker A_{v}$.  But $\mu' - v\mu' = (\Id - v)\mu' \in \Range(\Id -v) \subseteq \Ker A_{v}$ by Lemma~\ref{lem:AwProperties}(5), so we are done. 
\end{proof}

We rephrase the previous lemma in the following immediate corollary.

\begin{corollary}\label{cor:containedTransverse}  With all notation as in Lemma~\ref{lem:containedTransverse}, the $\eW$-conjugacy class of $b_\nu$ is contained in the set
\[
\bigcup_{u \in \sW} \left\{ t^\xi u v u^{-1} \mid \xi \in (u\cT_\nu) \cap (\omega+R^\vee) \right\}.
\]
\end{corollary}

\begin{figure}[ht]
\centering
\begin{overpic}[width=0.75\textwidth]{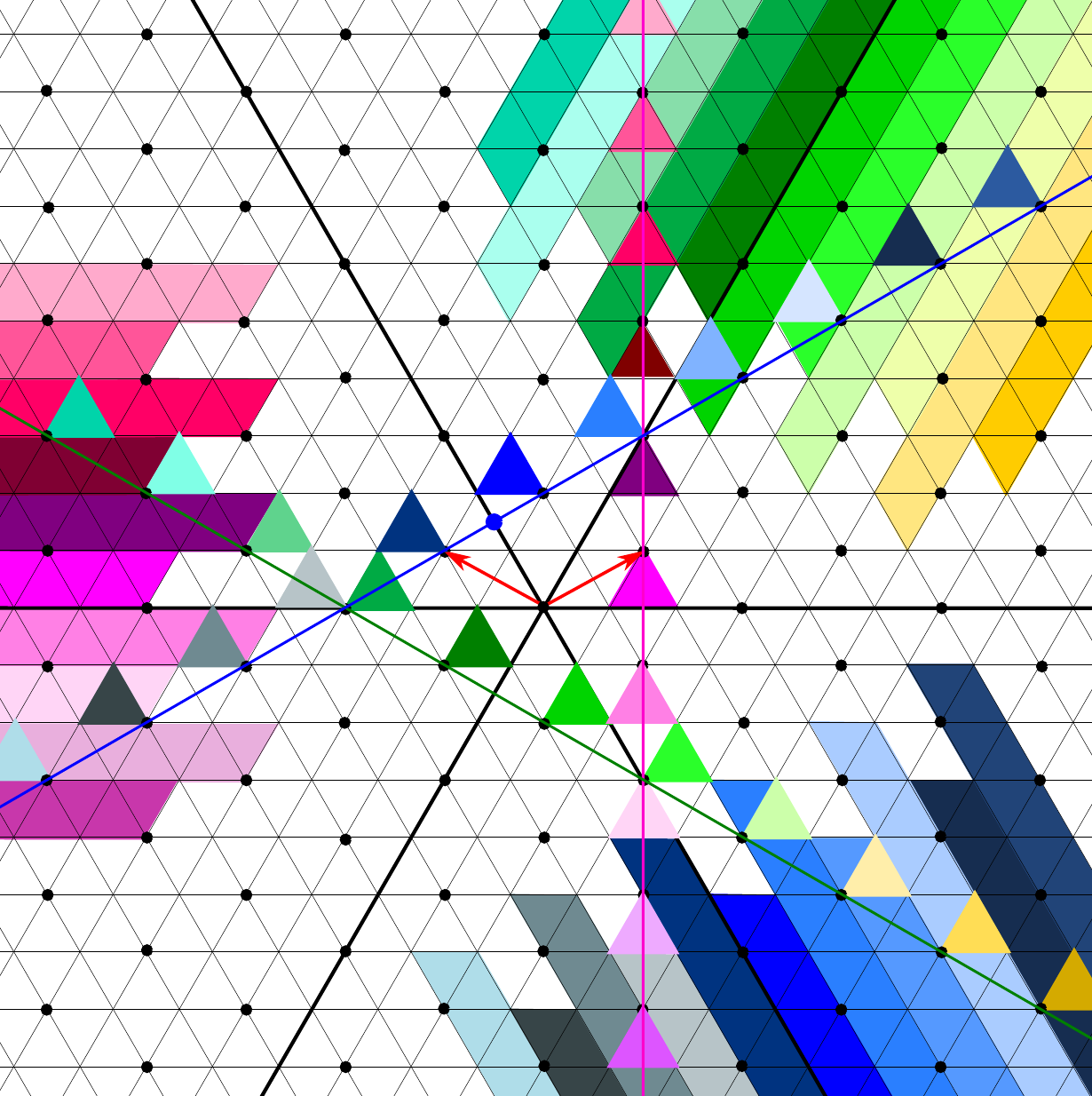}
\put(41,46){\footnotesize{\color{red}$\alpha_2^\vee$}}
\put(54,46){\footnotesize{\color{red}$\alpha_1^\vee$}}
\put(49,52){\footnotesize{\color{black}$\check{\rho}$}}
\put(42,58){\footnotesize{\color{black}$b_\nu$}}
\put(45,53.5){\footnotesize{\color{blue}$\nu$}}
\put(49,57){\footnotesize{\color{blue}$\eta$}}
\put(18,35){\footnotesize{\color{blue}$\cT_\nu = s_1 T_\nu$}}
\put(39,30){\footnotesize{\color{DarkGreen}$s_1 s_2\cT_\nu = w_0 T_\nu$}}
\put(59,52){\footnotesize{\color{Fuchsia}$s_2 \cT_\nu = s_2 s_1 T_\nu$}}
\end{overpic}
\caption{The $\eW$-conjugacy class of $b_\nu = t^{\check{\rho}} s_1$, with sectors in the same color representing the corresponding chimneys.}
\label{fig:conjugates}
\end{figure}

We note that in general the union in the statement of Corollary~\ref{cor:containedTransverse} will not be disjoint.  For example, in Figure~\ref{fig:conjugates}, if $\xi \in \cT_\nu \cap R^\vee$ and $u = s_1$ then $\xi \in u \cT_\nu = \cT_\nu$ and $t^\xi s_1 = t^\xi u s_1 u^{-1}$.  In this figure, the colored alcoves, which lie along the transverse subspaces,  are exactly those which are (labeled by) the elements of $\eW$ given in the statement of  Corollary~\ref{cor:containedTransverse}.

Motivated by Corollary~\ref{cor:containedTransverse}, we now make the following definition.

\begin{definition}\label{defn:fillsOut}  Let $[b] \in B(G)_P$, and denote by $\nu$ the Newton point and by $\omega$ the Kottwitz point.  Let $b_\nu = t^\eta v$ be the standard representative for $[b]$. We say that the $\eW$-conjugacy class of $b_\nu$ \emph{fills out the transverse subspaces} if it is equal to
\[
\bigcup_{u \in \sW} \left\{ t^\xi u v u^{-1} \mid \xi \in (u\cT_\nu) \cap (\omega + R^\vee) \right\}.
\]
\end{definition}

The next lemma gives a convenient condition equivalent to verifying the set equality in Definition \ref{defn:fillsOut}. Recall that by Corollary~\ref{cor:AwKerRange}, for any $v \in \sW$ we have $\Ker A_v = \Range(\Id - v)$. However, we will see later in Examples \ref{ex:C2} and \ref{ex:s1s3} that Condition~\eqref{eq:kerRangeZ} below is a stronger requirement, and by no means automatic from this equality.

\begin{lemma}\label{lem:fillsOut}   
Let $[b] \in B(G)_P$, and denote by $\nu$ the corresponding Newton point. Let $b_\nu = t^\eta v$ be the standard representative for $[b]$. The $\eW$-conjugacy class of $b_\nu$ fills out the transverse subspaces if and only if the following condition holds: 
\begin{equation}\label{eq:kerRangeZ}
\Ker A_{v} \cap R^\vee \subseteq (\Id - v)R^\vee.
\end{equation}
\end{lemma}

\begin{proof} Throughout the proof, denote the Kottwitz point of $[b]$ by $\omega$. First suppose that the $\eW$-conjugacy class of $b_\nu$ fills out the transverse subspaces.  Then by definition, for all $u \in \sW$ and  $\xi \in (u\cT_\nu) \cap (\omega + R^\vee)$, there exists a $\mu \in Q^\vee$ such that for $y = t^\mu u\in \eW$, we have $b_\nu^y = t^\xi u v u^{-1}$.  As in the proof of Lemma~\ref{lem:containedTransverse}, we have
\[
\xi = \mu + u \eta - u v u^{-1} \mu.
\]
Multiply on the left by $u^{-1}$ and replace $u^{-1}\xi$ by $\xi'$ and $u^{-1}\mu$ by $\mu'$.  Then for all $\xi' \in \cT_\nu \cap (\omega+R^\vee)$, there exists $\mu' \in R^\vee$ such that 
\[
\xi' -\eta = \mu' - v \mu'.
\]
In other words, for all $\xi' \in \cT_\nu \cap (\omega+R^\vee)$, we have $\xi' - \eta \in (\Id - v)R^\vee$.

Now consider an arbitrary element $\sum_{j = 1}^n c_j \alpha_j^\vee \in \Ker A_{v} \cap R^\vee$, where $c_j \in \Z$. Define $\zeta:= \eta + \sum_{j = 1}^n c_j \alpha_j^\vee$.  Since $\eta \in \cT_\nu \cap (\omega+R^\vee)$ by Corollary~\ref{cor:stdRepTransverse}, it follows that $
\zeta \in \cT_\nu \cap (\omega+R^\vee)$.  
So by the previous paragraph, $\sum_{j = 1}^n c_j \alpha_j^\vee = \zeta - \eta \in (\Id - v)R^\vee.$
Therefore Condition~\eqref{eq:kerRangeZ} holds.

For the reverse containment, suppose Condition~\eqref{eq:kerRangeZ} holds, and let $u \in \sW$ and $\xi \in (u\cT_\nu) \cap (\omega+R^\vee)$.  We will construct an element $\mu \in Q^\vee$ such that for $y = t^\mu u \in \eW$, we have $b_\nu^y = t^\xi u v u^{-1}$.
Write $\xi' = u^{-1}\xi$, so that $\xi' \in \cT_\nu \cap (\omega+R^\vee)$.   Since $\eta \in \cT_\nu \cap (\omega+R^\vee)$ by Corollary \ref{cor:stdRepTransverse}, it follows that 
\[ \xi' - \eta \in \Ker A_{v} \cap R^\vee.\]  
Then by Condition~\eqref{eq:kerRangeZ}, there exists $\mu' \in R^\vee$ such that
\[
\xi' - \eta = (\Id - v)\mu'.
\]

Finally, let $\mu = u\mu'$, and recall that $\xi = u\xi'$.  Then 
\[
\xi = \mu + u\eta - u v u^{-1}\mu.
\]
Thus if $y = t^\mu u$, we obtain $b_\nu^y = t^\xi u v u^{-1}$, as in the proof of Lemma~\ref{lem:containedTransverse}.  So the $\eW$-conjugacy class of $b_\nu$ fills out the transverse subspaces, as required.
\end{proof}

 Condition~\eqref{eq:kerRangeZ} from the statement of Lemma~\ref{lem:fillsOut} is quite subtle, as indicated by the following examples.  These examples use the indexing from the plates in Bourbaki~\cite{Bourbaki4-6}.

\begin{example}\label{ex:C2}  Suppose that $G$ is of type $\tilde{C}_2$, and consider $v=s_2$.  By definition, we then have $A_{s_2}= \frac{1}{2}\left( I + s_2\right)$.  Compute that 
\[ A_{s_2}(\alpha_2^\vee) = \frac{1}{2}\left(I+s_2\right)\alpha_2^\vee = \frac{1}{2}\left(\alpha_2^\vee + s_2(\alpha_2^\vee)\right) = 0,\]
and so $\alpha_2^\vee \in \Ker(A_{s_2}) \cap R^\vee.$ (Alternatively, directly apply Lemma \ref{lem:AwProperties}(4) or Equation \eqref{eq:nuProj} from Lemma \ref{lem:NPimageAsi}.)

On the other hand, for any $c_1\alpha_1^\vee + c_2\alpha_2^\vee \in R^\vee$, we have
\begin{equation}\label{eq:s2ActC2} s_2(c_1\alpha_1^\vee + c_2\alpha_2^\vee) =  c_1s_2(\alpha_1^\vee) + c_2s_2(\alpha_2^\vee) =   c_1(\alpha_1^\vee+2\alpha_2^\vee) -c_2\alpha_2^\vee.
\end{equation}
Therefore, 
\[(I-s_2)(c_1\alpha_1^\vee + c_2\alpha_2^\vee) = (c_1\alpha_1^\vee + c_2\alpha_2^\vee)-(c_1(\alpha_1^\vee+2\alpha_2^\vee) -c_2\alpha_2^\vee)=2(c_2-c_1)\alpha_2^\vee. \]
Since $c_1,c_2 \in \Z$, the image $(I-s_2)R^\vee$ consists only of even multiples of $\alpha_2^\vee$, and in particular $\alpha_2^\vee \notin (I-s_2)R^\vee$. Since $\alpha_2^\vee  \in \Ker(A_{s_2}) \cap R^\vee$ but $\alpha_2^\vee \notin (I-s_2)R^\vee$,  Condition~\eqref{eq:kerRangeZ} does not hold for $v=s_2$ in type $\tilde{C}_2$. (Dually, Condition~\eqref{eq:kerRangeZ} does not hold for $v=s_1$ in type $\tilde{B}_2$.) 
\end{example}

As the next lemma suggests, the fact that Condition~\eqref{eq:kerRangeZ} fails for $v=s_2$ in type $\tilde{C}_2$ is not problematic for our purposes.  In type $\tilde{C}_2$, for all $[b] \in B(G)_P$ with $P=P_2$, Lemma \ref{lem:BCintegral} below says that the corresponding Newton point $\nu$ is integral, and so in fact $b_\nu = t^\nu$ and we are instead in the case of Corollary \ref{cor:integralTransverse}, rendering Lemma \ref{lem:fillsOut} for $v=s_2$ (and thus Example \ref{ex:C2}) irrelevant for our purposes. This feature is illustrated in Figure \ref{fig:C2StdReps}, where we see that all of the pink Newton points  $\nu \in H_{\alpha_2}$ are represented by the pink translation alcoves based at $\nu$. 
We remark that one can prove a version of Lemma \ref{lem:BCintegral} in greater generality, but we include here only those cases which arise directly in the proof of Proposition \ref{prop:conjRank1}.

\begin{lemma}\label{lem:BCintegral}
Suppose $[b] \in B(G)_P$ with $P=P_i$, and denote by $\nu$ the Newton point.  If the Kottwitz point $\kappa_G(b) = 0$, and either of the following conditions hold:
\begin{enumerate}
    \item $G$ is of type $\tilde{B}_2$, and $i=1$; or
    \item $G$ is of type $\tilde{C}_2$, and $i=2$; 
\end{enumerate}
then $\nu$ is integral.
\end{lemma}

\begin{proof}
We include the proof for $G$ of type $\tilde{C}_2$ and $i=2$, since the argument for type $\tilde{B}_2$ with $i=1$ follows by duality.  Since $\kappa_G(b)=0$, the Newton point $\nu$ is the image of an element of $R^\vee$.  It thus suffices to prove that $A_{s_2}(R^\vee) \subseteq R^\vee$.  To this end, consider any $c_1\alpha_1^\vee+c_2\alpha_2^\vee \in R^\vee$, and compute using \eqref{eq:s2ActC2} that 
\[ (I+s_2)(c_1\alpha_1^\vee+c_2\alpha_2^\vee) = (c_1\alpha_1^\vee+c_2\alpha_2^\vee)+(c_1(\alpha_1^\vee+2\alpha_2^\vee) -c_2\alpha_2^\vee))=2c_1(\alpha_1^\vee+\alpha_2^\vee).\]
Therefore, $\nu = A_{s_2}(c_1\alpha_1^\vee+c_2\alpha_2^\vee) = \frac{1}{2}(I+s_2)(c_1\alpha_1^\vee+c_2\alpha_2^\vee) = c_1(\alpha_1^\vee+\alpha_2^\vee) \in R^\vee$. Since $\nu \in R^\vee$ and $\kappa_G(b)=0$, then $\nu$ is integral.
\end{proof}

The lucky phenomenon demonstrated by Lemma \ref{lem:BCintegral}, in which Condition~\eqref{eq:kerRangeZ} fails precisely when no non-integral Newton points actually exist, does not always persist when the rank of the parabolic subgroup is greater than 1.  In fact, Example \ref{ex:s1s3} below illustrates the primary obstruction to providing a straightforward generalization of our methods to treat parabolic subgroups of higher rank.  More specifically, above parabolic rank 1, Example \ref{ex:s1s3} shows that the $\eW$-conjugacy class of a standard representative does not always fill out the transverse subspaces.

\begin{example}\label{ex:s1s3}  Suppose that $G = SL_4$, and consider the element $b = t^{(1,0,0,-1)}s_1s_3 \in \eW$. Compute using \eqref{eq:avgFormula} that $\nu := \nu_G(b) = \left(\frac{1}{2},\frac{1}{2},-\frac{1}{2},-\frac{1}{2}\right)$.  By \eqref{eq:PnuRoots}, we have $P_\nu = MN$ where the Levi subgroup $M$ is isomorphic the product $GL_2 \times GL_2$ inside $SL_4$.  To distinguish these two copies of $GL_2$, we write $M_1 \times M_3$, where $\sW_{M_i} = \langle s_i \rangle$ are each of type $A_1$.

We claim that $b$ is the standard representative of $[b]$, which we prove using Lemma \ref{lem:stdRepList}. First note that $\kappa_G(b) = 0$ since $(1,0,0,-1) \in R^\vee$ and $G = SL_4$, verifying property (1).  Next, note that $b = (t^{\varpi_1}s_1) \cdot (t^{\varpi_3}s_3)^{-1} \in \Omega_M$, verifying property (2).  Therefore, $\nu_M(b) = \nu_G(b) = \nu$, satisfying property (3).   By Lemma \ref{lem:stdRepList}, we now have that $b = t^{(1,0,0,-1)}s_1s_3$ is the standard representative of $[b]$; compare \cite[Ex.~7.2.5]{GHKRadlvs}.

For $v=s_1s_3$, we have shown thus far that $v$ is the spherical direction of the standard representative $b_\nu =t^{(1,0,0,-1)}s_1s_3$.  Since $\nu \in H_{\alpha_1} \cap H_{\alpha_3}$, note that the associated parabolic subgroup is not of rank~$1$.  We proceed to show that Condition~\eqref{eq:kerRangeZ} does not hold for $v = s_1 s_3$. With respect to the basis $\{ \alpha_1^\vee, \alpha_2^\vee, \alpha_3^\vee \}$, the matrix for $v$ is
\[
v = \begin{pmatrix} 
-1 & 1 & 0 \\
0 & 1 & 0  \\
0 & 1 & -1 
\end{pmatrix}.
\]
Since $v$ has order $2$, 
\[
A_v =  \frac{1}{2}\left(\Id + v\right) =  \frac{1}{2}\begin{pmatrix} 
0 & 1 & 0 \\
0 & 2 & 0  \\
0 & 1 & 0 
\end{pmatrix}
\quad \mbox{and} \quad 
\Id - v = \begin{pmatrix} 
2 & -1 & 0 \\
0 & 0 & 0 \\
0 & -1 & 2 
\end{pmatrix}.
\]  
Thus,
\[
\Ker A_v \cap R^\vee = \{ c_1 \alpha_1^\vee + c_3 \alpha_3^\vee \mid c_1, c_3 \in \Z \}.
\]  
In \eqref{eq:kerRangeZ}, we are thus considering whether for all $c_1, c_3 \in \Z$ the equation
\[
c_1 \alpha_1^\vee + c_3 \alpha_3^\vee = (\Id - v)\left( \sum_{i=1}^3 d_i \alpha_i^\vee \right)
\]
has a solution with $d_1, d_2, d_3 \in \Z$.  
This is equivalent to the system
\begin{eqnarray*}
2 d_1 - 2d_3 & = & c_1 - c_3 \\
d_2 - 2d_3 & = & -c_3,
\end{eqnarray*}
and this system has a solution $d_1, d_2, d_3 \in \Z$ if and only if the integer $c_1 - c_3$ is even.   Thus Condition~\eqref{eq:kerRangeZ} does not hold for $v = s_1 s_3$ in $\sW$ in type $\tilde{A}_3$, showing that $\eW$-conjugacy classes of standard representatives do not necessarily fill out the transverse subspaces if the rank of the parabolic subgroup is larger than 1.
\end{example}

\subsection{Conjugacy classes of standard representatives in parabolic rank 1}\label{sec:fillOut}

We now consider $\eW$-conjugacy classes of standard representatives $b_\nu$ where $\nu$ is non-integral and the associated parabolic subgroup has rank $1$.  The main results are   Corollary~\ref{cor:conjRank1}, which gives a precise description of these $\eW$-conjugacy classes in terms of the transverse subspaces introduced in Section~\ref{sec:transverse}, and Proposition~\ref{prop:inverseRank1}, which partly solves the ``inverse problem" for an element of the $\eW$-conjugacy class of $b_\nu$.

In contrast to Example~\ref{ex:s1s3}, when the parabolic subgroup is rank 1, we have the following.

\begin{prop}\label{prop:conjRank1}  Suppose $[b_\nu] \in B(G)_P$ with $P = P_i$ and $\nu$ non-integral.  Then 
\[
\Ker A_{s_i} \cap R^\vee \subseteq (\Id - s_i)R^\vee;
\]
that is, Condition~\eqref{eq:kerRangeZ} holds. 
\end{prop}

\begin{proof} Since $\Ker A_{s_i} \cap R^\vee = \{ c_i \alpha_i^\vee \mid c_i \in \Z \}$, it suffices to show that there are integers $d_j \in \Z$ such that 
\begin{equation}\label{eq:rank1}
\alpha_i^\vee = (\Id - s_i) \left( \sum_{j = 1}^n d_j\alpha_j^\vee\right).
\end{equation}
First note that 
\[
(\Id - s_i)\alpha_i^\vee = \alpha_i^\vee - s_i(\alpha_i^\vee) = 2\alpha_i^\vee.
\]
Recall more generally that $s_i(\alpha_j^\vee) = \alpha_j^\vee - \langle \alpha_i, \alpha_j^\vee \rangle \alpha_i^\vee$ by definition.  Using the conventions in \cite{Bourbaki4-6}, the value of the pairing $\langle \alpha_i, \alpha_j^\vee \rangle = c_{ij}$ equals the $ij$-entry of the Cartan matrix. 
Now, if there exists an $\alpha_j \in \Delta$ such that $\langle \alpha_i, \alpha_j^\vee \rangle = -1$, then
\[
(\Id - s_i)\alpha_j^\vee= \alpha_j^\vee - s_i(\alpha_j^\vee) =\alpha_j^\vee - \left( \alpha_j^\vee - \langle \alpha_i, \alpha_j^\vee \rangle \alpha_i^\vee\right) =  \alpha_j^\vee - (\alpha_j^\vee + \alpha_i^\vee) = -\alpha_i^\vee,
\]
and so Equation~\eqref{eq:rank1} holds.  

Using the plates from Bourbaki~\cite{Bourbaki4-6}, the only cases in which $c_{ij}\neq -1$ for any $j \in [n]$ are $G$ of type $\tilde{B}_2$ with $i=1$; type $\tilde{C}_2$ with $i=2$; and type $\tilde{G}_2$ with $i=2$. In type $\tilde{B}_2$ with $i=1$, Lemma \ref{lem:BCintegral} says that $\nu$ is always integral, so there is nothing to prove, since $\nu$ is non-integral by hypothesis.  Similarly, in type $\tilde{C}_2$ with $i=2$, there is nothing to check.

Finally, suppose that $G$ of type $\tilde{G}_2$ and $i=2$.  We then have
\[(\Id - s_2)\alpha_1^\vee = \alpha_1^\vee - s_2 (\alpha_1^\vee) = \alpha_1^\vee - (\alpha_1^\vee - \langle \alpha_2, \alpha_1^\vee \rangle\alpha_2^\vee) = \alpha_1^\vee - (\alpha_1^\vee +3\alpha_2^\vee)=-3\alpha_2^\vee. \]
Recalling from above that $(I-s_2)\alpha_2^\vee = 2\alpha_2^\vee$, we have
 $\{ (\Id - s_2)\alpha_j^\vee \mid \alpha_j \in \Delta \} = \{ -3\alpha_2^\vee, 2 \alpha_2^\vee \}$ in this case.  Taking for example $d_1 = d_2 = -1$, we see that Equation~\eqref{eq:rank1} holds, completing the proof.
\end{proof}

We immediately obtain the following precise description of the $\eW$-conjugacy classes of the standard representatives in parabolic rank $1$. For example, in Figure~\ref{fig:conjugates}, the (labels of) the colored alcoves are exactly the elements of the $\eW$-conjugacy class of $b_\nu$.

\begin{corollary}\label{cor:conjRank1}  Suppose $[b_\nu] \in B(G)_P$ with $P = P_i$ and $\nu$ non-integral, and let $\omega$ denote the Kottwitz point.  Then the $\eW$-conjugacy class of $b_\nu$ is equal to
\[
\bigcup_{u \in \sW} \{ t^\xi u s_i u^{-1} \mid \xi \in (u\cT_\nu) \cap (\omega+R^\vee) \}.
\]
That is, the $\eW$-conjugacy class of $b_\nu$ fills out the transverse subspaces.
\end{corollary}

\begin{proof}
Combine Lemma~\ref{lem:fillsOut} and Proposition~\ref{prop:conjRank1}.
\end{proof}

Given an element $z$ of the $\eW$-conjugacy class of $b_\nu$, a natural problem is to find all $y\in \eW$ such that $z = b_\nu^y$.  The next result provides one explicit solution $y \in \eW$ to this ``inverse problem" for every such $z$. We do not require a complete solution to the inverse problem for our current purposes, and so do not pursue this further.  We simply remark that, given one solution, finding all solutions would be equivalent to understanding the centralizer of $b_\nu$ in $\eW$. As far as we are aware, this centralizer is not known (even in the parabolic rank $1$ case).

\begin{prop}\label{prop:inverseRank1}
Suppose $[b_\nu] \in B(G)_P$ with $P = P_i$ and $\nu$ non-integral, and let $\omega$ denote the Kottwitz point.  For any $u \in \sW$ and $\zeta \in (u\cT_\nu) \cap (\omega+R^\vee)$, note by Corollary~\ref{cor:conjRank1} that $z = t^\zeta us_i u^{-1}$ is in the $\eW$-conjugacy class of $b_\nu$.  Then:
\begin{enumerate}
    \item there exists an integer $d_i \in \Z$ such that $u^{-1}\zeta - \eta = d_i \alpha_i^\vee$; and
    \item if $y = t^\mu u \in \eW$ with $\mu = d_i u \eta\in Q^\vee$ for $d_i \in \Z$ as in (1), then $z = b_\nu^y$.
\end{enumerate}
\end{prop}

\begin{proof}
We have $u^{-1}\zeta \in \cT_\nu \cap (\omega+R^\vee)$ by hypothesis, and $\eta \in \cT_\nu$ by Corollary~\ref{cor:stdRepTransverse}.  Since $\eta \in  \omega + R^\vee$ by Corollary \ref{cor:stdRepTransverse}, using Corollary~\ref{cor:nonintegralTransverse}  we see that $u^{-1}\zeta - \eta = d_i \alpha_i^\vee$ for some $d_i \in \Z$, and so (1) holds.

For (2), let $y = t^\mu u$, where $\mu = d_i u \eta$ and $d_i\in \Z$ is as in part (1).  Then the translation part of $b_\nu^y$ is equal to 
\[
\mu + u \eta - u s_i u^{-1}\mu = d_i u \eta + u \eta - d_i u s_i \eta = d_i u (\eta - s_i \eta) + u \eta.
\]
Since $\eta \in H_{\alpha_i,1}$ by Proposition~\ref{prop:stdReps}, then $s_i\eta = \eta -\alpha_i^\vee$, and so $\eta - s_i \eta = \alpha_i^\vee$.  Therefore the translation part of $b_\nu^y$ equals
\[
u( d_i \alpha_i^\vee + \eta) = u(u^{-1}\zeta) = \zeta
\]
by part (1). Since the spherical direction of $b_\nu^y$ is equal to $us_i u^{-1}$, we have that $b_\nu^y= z$, as required.
\end{proof}

In Figure~\ref{fig:conjugates}, for each alcove $\z$ corresponding to an element of the $\eW$-conjugacy class of $b_\nu = t^{\check{\rho}}s_1$, the shaded region of the same color is a sector representing the $(P_1,y)$-chimney, for $y$ as in the statement of Proposition~\ref{prop:inverseRank1}.  (We do not indicate the alcoves $\y$, as this would make the picture too cluttered.)


\section{The first target}\label{sec:firstTarget}

Since if $G$ has rank $1$ all Newton points are integral and Theorem \ref{thm:translations} applies, from now on we assume that $G$ has rank $n \geq 2$.  Let $x_0 = t^\lambda w_0 \in \eW$ be such that the alcove $\x_0$ is in the shrunken dominant Weyl chamber.
In Section~\ref{sec:minGallery}, we construct a minimal gallery $\gamma:\fa \rightsquigarrow \x_0$, and then  we carry out a PRS-folding sequence on $\gamma$ in Section~\ref{sec:folding}.  We then use these constructions in Section~\ref{sec:conjugationFinal} to prove the following result.  

\begin{thm}\label{thm:firstTarget} Let $x_0=t^\lambda w_0\in \eW$, and assume that $\x_0$ is in the shrunken dominant Weyl chamber $\Cfs$.  
Suppose that
\[
\proj_i(\lambda - 2\check{\rho}) = \nu,
\]
where $\nu$ is a non-integral Newton point with associated standard spherical parabolic subgroup $P_i$ of rank $1$. 
Then for any $b \in \eW$ such that $\nu_b = \nu$ and $\kappa_G(b)=\kappa_G(x_0)$, we have
$X_{x_0}(b) \neq \emptyset.$
\end{thm}

We will also use the constructions from this section later in Section~\ref{sec:remainingTargets}.

\subsection{Construction of minimal gallery}\label{sec:minGallery}

In this section, we construct a minimal gallery $\gamma: \fa \rightsquigarrow \x_0$ using a specific choice of reduced word for the longest element $w_0 \in \sW$. 
An example of the construction of $\gamma$ in type $\tilde{A}_2$ is given in Figure~\ref{fig:gamma}. Although our constructions apply equally well to any fixed sheet of extended alcoves, we illustrate all explicit gallery constructions in the 0-sheet.

\begin{figure}[ht]
\centering
\begin{overpic}[width=0.4\textwidth]{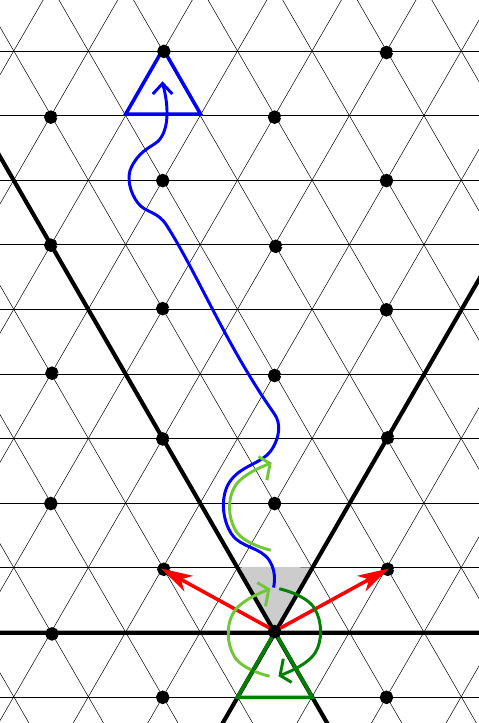}
\put(21.5,94.5){\color{blue}$\lambda$}
\put(19,76){\color{blue}$\lambda - \check{\rho}$}
\put(21,89){\color{blue}$\x_0$}
\put(32,53){\color{blue}$\gamma$}
\put(37,17){\color{black}$\fa$}
\put(37,27){\color{black}$\check{\rho}$}
\put(47,15){\color{red}$\alpha_1^\vee$}
\put(25,15){\color{red}$\alpha_2^\vee$}
\put(54,51){\color{black}$H_{\alpha_2}$}
\put(9,51){\color{black}$H_{\alpha_1}$}
\put(36,4.5){\color{DarkGreen}$\w_0$}
\put(44,8){\color{DarkGreen}$\tilde\tau$}
\put(30,8){\color{LimeGreen}$\tau$}
\put(33,31){\color{LimeGreen}$\tau'$}
\put(47,25){\color{DarkGreen}$H_1$}
\put(53,14){\color{DarkGreen}$H_2$}
\put(45,1){\color{DarkGreen}$H_3$}
\end{overpic}
\caption{An example of the construction of the minimal gallery $\gamma:\fa \rightsquigarrow \x_0$.}
\label{fig:gamma}
\end{figure}
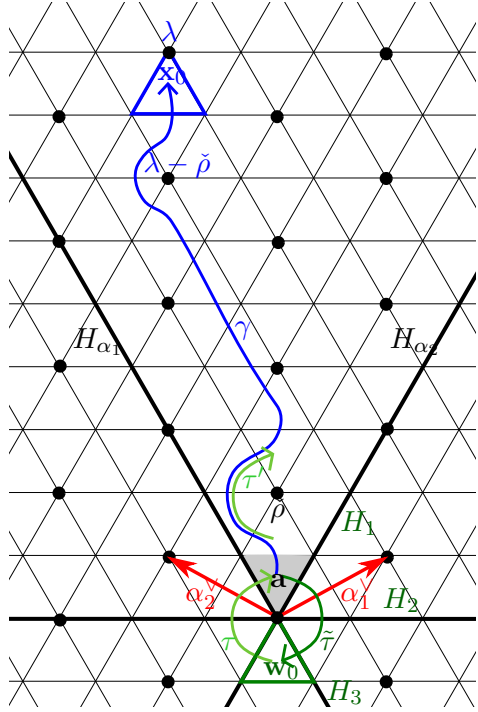

Fix $i \in \{1,\dots,n\}$.  The reduced word for $w_0$ is chosen as follows.  Since conjugation by $w_0$ permutes the (spherical)  simple reflections, $w_0 s_i w_0$ is again a simple reflection, which we denote by $s_{i_1}$.  Thus $w_0 \alpha_i = -\alpha_{i_1}$.  We then denote by $s_{i_2},\dots,s_{i_n}$ the remaining spherical simple reflections $s_j$ for $j \in \{1,\dots,n\} \setminus \{ i_1\}$, written in any order.  Now choose a reduced word $\vec w_0 = s_{k_1} \cdots s_{k_\ell}$ for $w_0$, where $\ell = \ell(w_0)$, which ends with the Coxeter element $s_{i_n}s_{i_{n-1}} \cdots s_{i_1}$.  That is, $\vec w_0$ has last letter $s_{k_\ell} = s_{i_1}$, and $s_{k_{\ell-1}}=s_{i_2}$, \dots, $s_{k_{\ell - n + 1}} = s_{i_n}$ are the remaining spherical simple reflections.  In the example depicted in Figure~\ref{fig:gamma}, we have $i = 1$, and hence $i_1 = 2$, $i_2 = 1$, and $\vec w_0 = s_2 s_1 s_2$.

Recall that $x_0 =t^\lambda w_0$, and suppose throughout the rest of the paper that $\lambda \in \omega + R^\vee$, or equivalently that $\x_0$ is in the $\omega$-sheet of extended alcoves with base alcove $\fa = \fa_\omega$. We construct a minimal gallery $\gamma: \fa \rightsquigarrow \x_0$ as follows. Denote by $\w_0 = w_0\fa$. Let the gallery $\tilde\tau:\fa \rightsquigarrow \w_0$ be of type $(s_{\omega k_1},\dots,s_{\omega k_\ell})$; that is, $\tilde\tau$ has the same type as $\vec w_0$, appropriately relabeled according to the action of $\omega$ on $\fa_0$. Denote by $H_1, H_2, \dots, H_\ell$ the (ordered) list of hyperplanes crossed by $\tilde\tau$, which is independent of $\omega$. One can compute from the chosen word $\vec w_0$ for $w_0$ that $H_1=H_{\alpha_{k_1}}$ and $H_j= s_{k_1}s_{k_2}\cdots s_{k_{j-1}}H_{\alpha_{k_j}}$ for all $j=2, \dots, \ell$.  In Figure~\ref{fig:gamma}, the gallery $\tilde \tau$ is the dark green gallery from $\fa$ to $\w_0$, of $(s_2,s_1,s_2)$, and the hyperplanes $H_1$, $H_2$, and $H_3$ are also labeled in dark green.

Now let $\tau: \w_0\rightsquigarrow \fa$ be the minimal gallery that crosses the same list of hyperplanes as $\tilde\tau,$ in the same order. In other words, the alcoves and panels in $\tau$ are opposite those in $\tilde\tau$ at each step, and $\tau$ can be thought of as the image of $\tilde \tau$ under the opposition map in the star (or link) of the origin. The type of $\tau$ may or may not be the same as the type of $\tilde\tau$, depending on $G$.  For example, if $G$ is of type $\tilde{B}_2$ then $\tau$ will have the same type as $\tilde \tau$.  In Figure~\ref{fig:gamma}, in which $G$ is type $\tilde{A}_2$, the gallery $\tau$ is the light green gallery from $\w_0$ to $\fa$, and $\tau$ has type $(1,2,1)$, which is different from the type of $\tilde \tau$.

Next, let $\tau'$ be the gallery from the $w_0$-position at $\check{\rho}$ to the identity position at $\check{\rho}$ obtained by applying the translation $t^{\check{\rho}}$ to $\tau$.   Note that the translation $t^{\check{\rho}}$ is not necessarily type-preserving; however, for each simple root $\alpha_k$, it takes the hyperplane $H_{\alpha_k,0}$ through the origin to the hyperplane $H_{\alpha_k,1}$ through $\check{\rho}$.  In Figure~\ref{fig:gamma}, the gallery $\tau' = t^{\check{\rho}} \tau$ is also depicted in light green.

Now let $\gamma:\fa \rightsquigarrow \x_0$ be the gallery obtained by concatenating the following five galleries:
\begin{enumerate}
    \item a minimal gallery from $\fa$ to the $w_0$-position at $\check{\rho}$;
    \item the gallery $\tau'$;
    \item a minimal gallery from the identity position at $\check{\rho}$ to the $w_0$-position at $\lambda - \check{\rho}$;
    \item a minimal gallery from the $w_0$-position at $\lambda - \check{\rho}$ to the identity position at $\lambda-\check{\rho}$, whose final crossing is in a panel of the $\alpha_i$-hyperplane passing through $\lambda - \check{\rho}$;
    \item a minimal gallery from the identity position at $\lambda - \check{\rho}$ to the alcove $\x_0$, which is in the $w_0$-position at $\lambda$.
\end{enumerate}
In Figure~\ref{fig:gamma}, such a gallery $\gamma$ is shown in blue.  The resulting gallery $\gamma$ is minimal by similar arguments to those in \cite[Lemma 6.2]{MST1}, and we denote its type by $\vec{x}_0$.  For the constructions in this section we will only need some features of $\gamma$, and the remaining features will be used in Section~\ref{sec:remainingTargets}.

\subsection{Folding the gallery}\label{sec:folding}

In this section we carry out a PRS-folding sequence on the minimal gallery $\gamma:\fa \rightsquigarrow \x_0$ from Section~\ref{sec:minGallery} to obtain a gallery $\gamma_{\check{\rho}}$ of type $\vec{x}_0$, and we determine the first and final alcoves of $\gamma_{\check{\rho}}$.  To illustrate this folding sequence, we continue the example from Figure~\ref{fig:gamma} in Figure~\ref{fig:gamma_rho} below.  We also establish a sufficient condition on certain elements $y \in \eW$ for $\gamma_{\check{\rho}}$ to be positively folded with respect to the $(P_i,y)$-chimney.

\begin{figure}[ht]
\centering
\begin{overpic}[width=0.7\textwidth]{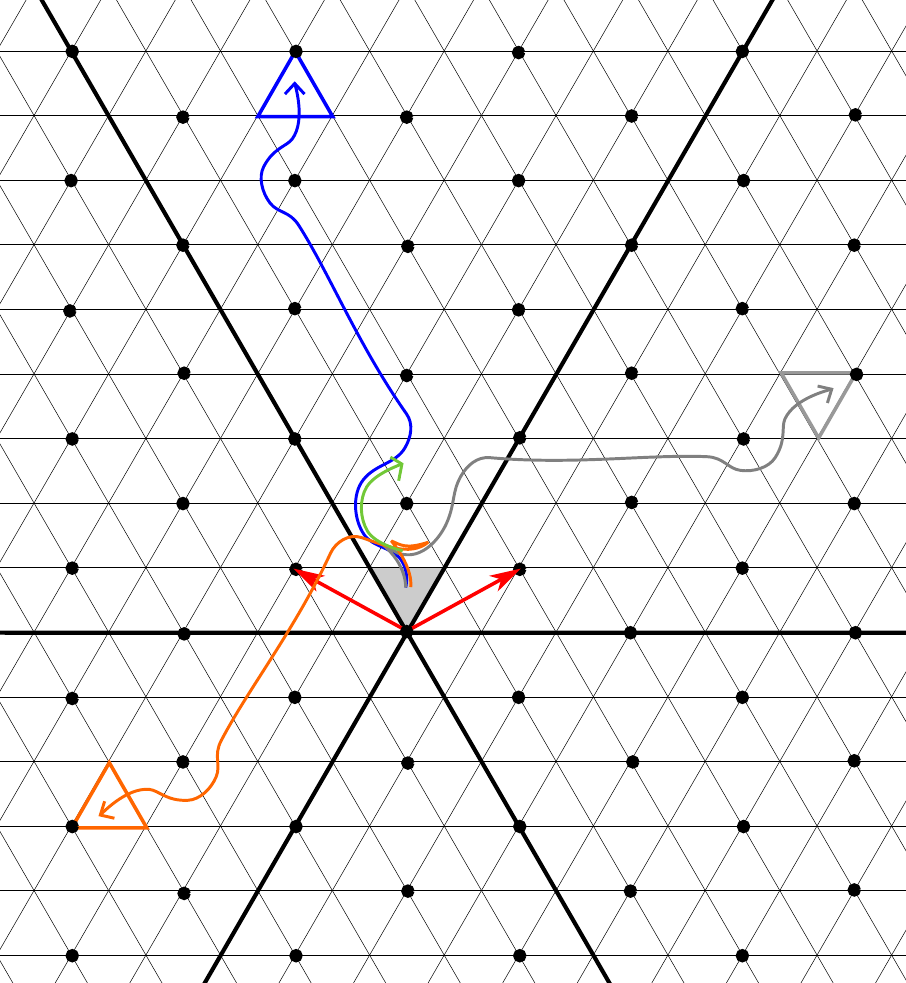}
\put(29,96){\color{blue}$\lambda$}
\put(24,70){\color{blue}$\lambda - 2\check{\rho}$}
\put(29,91.7){\color{blue}$\x_0$}
\put(37,65){\color{blue}$\gamma$}
\put(68,91){\color{black}$H_{\alpha_2}$}
\put(10,91){\color{black}$H_{\alpha_1}$}
\put(40.5,38.5){\color{black}$\fa$}
\put(37,37){\color{black}$\s_1$}
\put(44,37){\color{black}$\s_2$}
\put(35,33.5){\color{black}$\s_1\s_2$}
\put(43,33.5){\color{black}$\s_2 \s_1$}
\put(40,30){\color{black}$\w_0$}
\put(40.5,46){\color{black}$\check{\rho}$}
\put(49,37){\color{red}$\alpha_1^\vee$}
\put(31,37){\color{red}$\alpha_2^\vee$}
\put(37.5,48){\color{LimeGreen}$\tau'$}
\put(10.5,19){\color{OrangeRed}$\z$}
\put(4.5,14){\color{OrangeRed}$\zeta$}
\put(4.5,57){\color{OrangeRed}$w_0\zeta$}
\put(25,26){\color{OrangeRed}$\gamma_{\check{\rho}}$}
\end{overpic}
\caption{An example of the construction of the folded gallery $\gamma_{\check{\rho}}$.}
\label{fig:gamma_rho}
\end{figure}

To obtain $\gamma_{\check{\rho}}$ from $\gamma$, we fold the first $\ell(w_0) -1$ crossings of the subgallery $\tau' = t^{\check{\rho}} \tau$ of $\gamma$ using a PRS-folding sequence, so that there are $\ell(w_0)  - 1$ folds in total, all occurring in panels of the $w_0$-position at $\check{\rho}$.  (Note that since $G$ has rank $n \geq 2$, we have $\ell(w_0) > 1$.)   This folding sequence consists of exactly the first $\ell(w_0) - 1$ folds of the sequence described in more detail in \cite[Sec.~6.2]{MST1}. In Figure~\ref{fig:gamma_rho}, where $\ell(w_0) - 1 = 2$, recall that the subgallery $\tau'$ is shown in light green and $\gamma$ is blue. The gray gallery is that obtained from $\gamma$ by carrying out the first fold in this folding sequence, and $\gamma_{\check{\rho}}$ is the orange gallery obtained from the gray gallery by carrying out the second and final fold.

We note that $\gamma_{\check{\rho}}$, like $\gamma$, has first alcove $\fa$, since this folding sequence leaves fixed the initial minimal subgallery of $\gamma$ from $\fa$ to the $w_0$-position at $\check{\rho}$. We now determine the final alcove $\z$ of the folded gallery $\gamma_{\check{\rho}}$.

\begin{lemma}\label{lem:sphericalZ}  The spherical direction of $\z$ is $s_{i_1} = w_0 s_i w_0$.
\end{lemma}

\noindent For example, in Figure~\ref{fig:gamma_rho}, we have $i_1 = 2$ and the spherical direction of $\z$ is in fact $s_2$.

\begin{proof} By construction, the final alcove $\x_0$ of $\gamma$ has spherical direction $w_0$.  The subgallery $\tau' = t^{\check{\rho}} \tau$ of $\gamma$ crosses the hyperplanes $t^{\check{\rho}} H_1, \dots, t^{\check{\rho}} H_\ell$, in this order, and all folds of $\gamma_{\check{\rho}}$ occur within $\tau'$.  Thus the spherical direction of $\z$ can be found by first translating $\x_0$ by $-\check{\rho}$, then folding the first $\ell(w_0) - 1$ crossings of $\tau$, then applying the translation $t^{\check{\rho}}$ to the final alcove of the resulting gallery.  Since translations do not alter spherical directions of alcoves, it suffices to consider the effect of the reflections which correspond to   folding the first $\ell(w_0) - 1$ crossings of $\tau$.  

The hyperplanes crossed by $\tau$ are $H_1,\dots,H_\ell$, with $H_j= s_{k_1}s_{k_2}\cdots s_{k_{j-1}}H_{\alpha_{k_j}}$.  The reflection corresponding to the first fold of $\tau$ is thus $s_{\alpha_{k_1}}$, and we have $s_{\alpha_{k_1}}H_2 = H_{\alpha_{k_2}}$, so the second fold of $\tau$ occurs in $H_{\alpha_{k_2}}$.  Continuing in this way we obtain that the hyperplanes in which the folds of $\tau$ occur are $H_{\alpha_{k_1}}$, $H_{\alpha_{k_2}}$, \dots, $H_{\alpha_{k_{\ell - 1}}}$, in this order.  Hence the spherical direction of $\z$ is $s_{k_{\ell -1}}\cdots s_{k_1}w_0$.

Using $\vec w_0 = s_{k_1} \cdots s_{k_\ell}$ we then obtain
\[s_{k_{\ell -1}}\cdots s_{k_1}w_0 = s_{k_{\ell -1}}\cdots s_{k_1} \cdot s_{k_1}\cdots s_{k_\ell}=s_{k_\ell}=s_{i_1}=w_0 s_i w_0, \] 
as claimed.
\end{proof}

We now determine the translation vertex of the final alcove of the gallery $\gamma_{\check{\rho}}$.  

\begin{lemma}\label{lem:translationZ}  The translation vertex of $\z$ is given by 
\begin{equation}\label{eq:zeta}
\zeta = \check{\rho} + w_0 s_i (\lambda - \check{\rho}) \;\in \omega + R^\vee.
\end{equation}
\end{lemma}

\noindent For example, in Figure~\ref{fig:gamma_rho}, we have $i = 1$ so $w_0 s_i = s_1 s_2$, and both $\lambda \in R^\vee$ and $\zeta = \check{\rho} + s_1 s_2 (\lambda - \check{\rho}) \in R^\vee$.

\begin{proof} 
First, note that although $\check{\rho}$ may not have the same image as $\lambda$ in $Q^\vee/R^\vee$, the final vertex $\zeta$ of $\gamma_{\check{\rho}}$ will, since reflections in hyperplanes through $\check{\rho}$ preserve this image.

Since $\x_0$ has translation vertex $\lambda$, by similar arguments to those in the proof of Lemma~\ref{lem:sphericalZ}, the final vertex of $\gamma_{\check{\rho}}$ can be found by first translating $\lambda$ by $-\check{\rho}$, then carrying out the reflections in the successive hyperplanes $H_{\alpha_{k_1}},\dots,H_{\alpha_{k_{\ell - 1}}}$, and finally translating the result by $\check{\rho}$.  So it suffices to show that $
s_{k_{\ell - 1}}\cdots s_{k_1} = w_0 s_i$.  Using the reverse of our chosen word $\vec{w}_0$, we obtain
\[
s_{k_{\ell - 1}}\cdots s_{k_1} = s_{k_\ell} s_{k_\ell} s_{k_{\ell - 1}}\cdots s_{k_1} = s_{k_\ell} w_0 = s_{i_1} w_0 = w_0 s_i w_0 w_0 = w_0 s_i,
\]
which completes the proof.
\end{proof}

We point out that the gallery $\gamma_{\check{\rho}}$ is not positively folded with respect to the $P_i$-chimney.  For example, compare the orientations for the $P_1$-chimney from Figure \ref{fig:chimneyOrientations} to see that the two folds in $\gamma_{\check{\rho}}$ from Figure \ref{fig:gamma_rho} both occur on the negative side of the respective hyperplane.  The next lemma determines an infinite family of elements $y \in \eW$ such that $\gamma_{\check{\rho}}$ is positively folded with respect to the $(P_i,y)$-chimney.

\begin{lemma}\label{lem:posFolded}  
Let $y = t^\mu w_0 s_i \in \eW$, where $\mu \in Q^\vee$.  If $\langle \alpha_{i_1}, \mu \rangle \geq 1$, then the gallery  $\gamma_{\check{\rho}}$ is positively folded with respect to the $(P_i,y)$-chimney.
\end{lemma}

\begin{proof}  Since the spherical direction of $\y$ is $w_0 s_i$, and $w_0 s_i \alpha_i = w_0(-\alpha_i) = \alpha_{i_1}$, any $(P_i,y)$-sector lies between the hyperplanes $H_{\alpha_{i_1},k}$ and $H_{\alpha_{i_1},k+1}$ for some $k \in \Z$.  More precisely, as any $P_i$-sector lies between $H_{\alpha_{i},0}$ and $H_{\alpha_{i},1}$, any $(P_i, y)$-sector lies between $H_{\alpha_{i_1},k}$ and $H_{\alpha_{i_1},k+1}$ where $k = \langle \alpha_{i_1}, \mu \rangle$.

Let $\beta \in \Phi^+ \setminus \{ \alpha_i \}$.  Then by definition the $P_i$-chimney contains all half-apartments bounded by $\beta$-hyperplanes which contain subsectors of the antidominant Weyl chamber.  Write $\beta' = w_0 s_i \beta$.  Then $s_i \beta$ is positive, hence $\beta'$ is negative.  Now since $\alpha_{i_1} = w_0 s_i \alpha_i$, it follows that $-\beta' \in \Phi^+ \setminus \{ \alpha_{i_1} \}$.  Moreover the map $\Phi^+ \setminus \{ \alpha_i \} \to \Phi^+ \setminus \{ \alpha_{i_1} \}$ given by $\beta \mapsto -\beta'$ is a bijection.

Therefore for all $\beta \in \Phi^+ \setminus \{ \alpha_{i_1} \}$ and all $y = t^\mu w_0 s_i$, any $(P_i,y)$-chimney contains all half-apartments bounded by $\beta$-hyperplanes which contain subsectors of the dominant Weyl chamber.  Thus for all $\beta \in \Phi^+ \setminus \{ \alpha_{i_1} \}$ and all $y = t^\mu w_0 s_i$, the periodic orientation on $\beta$-hyperplanes with respect to the opposite standard orientation agrees with the orientation induced by the $(P_i,y)$-chimney.

Now by construction, the gallery $\gamma_{\check{\rho}}$ has all folds positive with respect to the opposite standard orientation, and its only folds in an $\alpha_{i_1}$-hyperplane are in $H_{\alpha_{i_1},1}$.   So $\gamma_{\check{\rho}}$ will be positively folded with respect to any $(P_i,y)$-chimney represented by a sector lying between $H_{\alpha_{i_1},k}$ and $H_{\alpha_{i_1},k+1}$ where $k \geq 1$.  This is exactly the condition $\langle \alpha_{i_1}, \mu \rangle \geq 1$ from the statement.
\end{proof}

\begin{remark}\label{rem:chimney}
Conversely, we point out that if $\langle \alpha_{i_1}, \mu \rangle <1,$ then in fact $\gamma_{\check{\rho}}$ is not positively folded with respect to the $(P_i,y)$-chimney.  In particular, the folds in $H_{\alpha_{i_1},1}$ occur on the negative side of this hyperplane if $\langle \alpha_{i_1}, \mu \rangle <1$.  As we shall see in Section \ref{sec:remainingTargets}, we will need to handle these two cases separately.
\end{remark}

We will need the next result in Section~\ref{sec:conjugationFinal} and in Section~\ref{sec:remainingTargets}, which says that both $\lambda-2\check{\rho}$ and $w_0\zeta$ project to the same point on $H_{\alpha_i}$. Note that since $\sW$ preserves $R^\vee$ and $2\check{\rho} \in R^\vee$ by definition, then $w_0\zeta$ and $\lambda-2\check{\rho}$ have the same image in $Q^\vee/R^\vee$, which equals that of $\lambda$. See Figure~\ref{fig:gamma_rho} for an example, and recall that $i = 1$.

\begin{lemma}\label{lem:top}
For any $\lambda \in Q^\vee$ and $i \in [n]$ as in Theorem \ref{thm:firstTarget}, and $\zeta$ defined by Lemma \ref{lem:translationZ},
\[\proj_i(w_0\zeta) = \proj_i(\lambda - 2\check{\rho}).\]
\end{lemma}

\begin{proof}
By the definition of $\zeta$, properties of $\proj_i$, and the fact that $w_0 \check{\rho} = - \check{\rho}$,
\begin{eqnarray*}
\proj_i(w_0 \zeta) & = & \proj_i(w_0 \check{\rho} + s_i(\lambda - \check{\rho})) \\
& = & \proj_i(w_0 \check{\rho}) + \proj_i(s_i(\lambda - \check{\rho})) \\ 
& = & \proj_i(-\check{\rho}) + \proj_i(\lambda - \check{\rho}) \\
& = & \proj_i(\lambda -2\check{\rho}), 
\end{eqnarray*}
as claimed.
\end{proof}

\subsection{Reaching the first target}\label{sec:conjugationFinal}

We are now prepared to prove Theorem~\ref{thm:firstTarget}.  We thus assume that $\nu = \proj_i(\lambda - 2\check{\rho})$ is a non-integral Newton point with associated parabolic subgroup $P=P_i$ of rank 1.   We will show that the final alcove $\z$ of the folded gallery $\gamma_{\check{\rho}}$ from Section~\ref{sec:folding} is (labeled by) a certain $\eW$-conjugate of the standard representative $b_\nu$.  This allows us to conclude that $\gamma_{\check{\rho}}$ is positively folded with respect to a suitable chimney, such that we can apply Theorem~\ref{thm:ADLVChimneys}. Figure~\ref{fig:firstTarget}, which continues the example from Figures~\ref{fig:gamma} and~\ref{fig:gamma_rho}, illustrates the proof of Theorem \ref{thm:firstTarget}.

\begin{figure}[ht]
\centering
\begin{overpic}[width=0.6\textwidth]{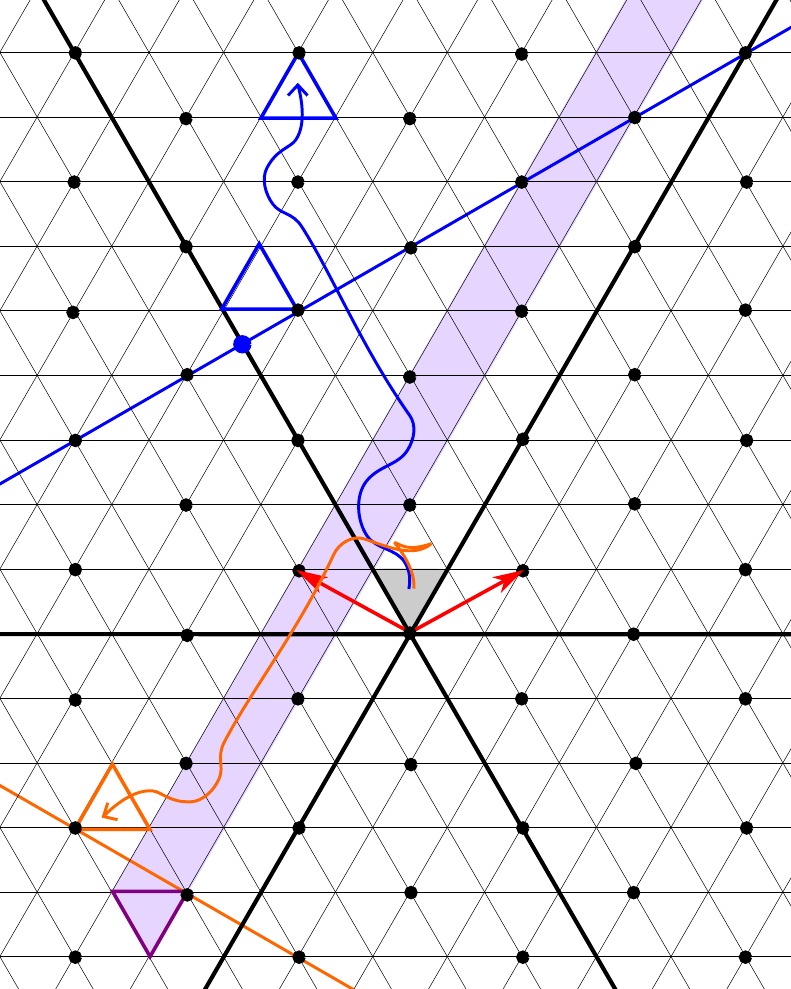}
\put(29,96){\color{blue}$\lambda$}
\put(25,70){\color{blue}$b_\nu$}
\put(29,91.7){\color{blue}$\x_0$}
\put(37,65){\color{blue}$\gamma$}
\put(29.5,66){\color{blue}$\eta$}
\put(23.5,62.6){\color{blue}$\nu$}
\put(69,93){\color{blue}$\cT_\nu$}
\put(40.5,38.5){\color{black}$\fa$}
\put(37,37){\color{black}$\s_1$}
\put(44,37){\color{black}$\s_2$}
\put(35,33.5){\color{black}$\s_1\s_2$}
\put(43,33.5){\color{black}$\s_2 \s_1$}
\put(40,30){\color{black}$\w_0$}
\put(40.5,46){\color{black}$\check{\rho}$}
\put(49,37){\color{red}$\alpha_1^\vee$}
\put(31,37){\color{red}$\alpha_2^\vee$}
\put(10.5,19){\color{OrangeRed}$\z$}
\put(4.5,14){\color{OrangeRed}$\zeta$}
\put(4.5,57){\color{OrangeRed}$w_0\zeta$}
\put(36,77){\color{OrangeRed}$s_i w_0\zeta$}
\put(21.5,33){\color{OrangeRed}$\gamma_{\check{\rho}}$}
\put(25.5,6){\color{OrangeRed}$w_0 s_i \cT_\nu$}
\put(18,11){\color{DarkMagenta}$\mu$}
\put(14.5,7){\color{DarkMagenta}$\y$}
\end{overpic}
\caption{An example illustrating the proof of Theorem~\ref{thm:firstTarget}.}
\label{fig:firstTarget}
\end{figure}
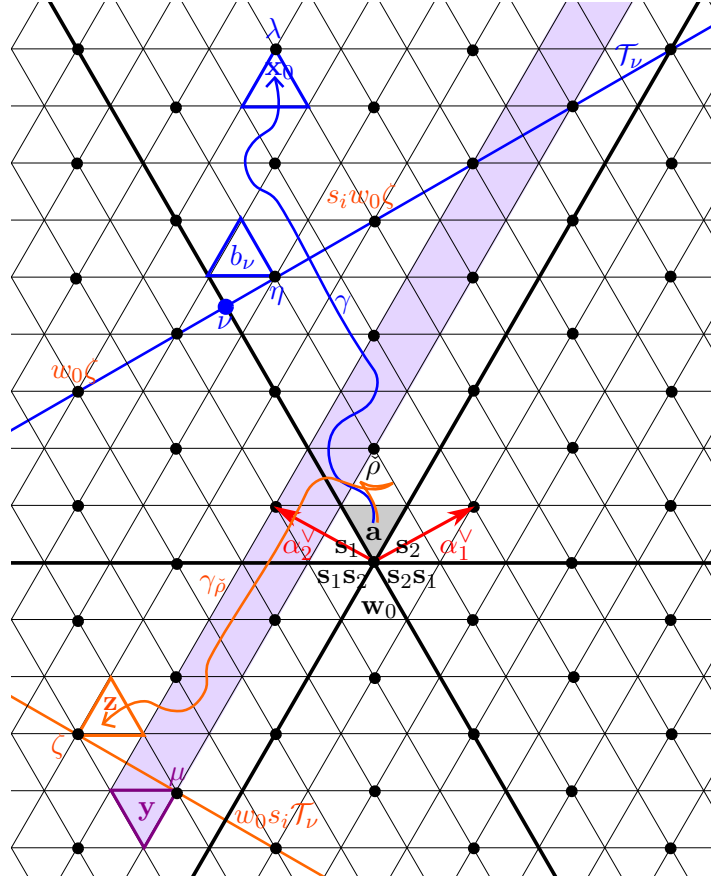

We start by using some results on conjugacy classes from Section~\ref{sec:conjugacy}.  Recall that the subspace transverse to the Newton point $\nu$ with standard representative $b_\nu = t^\eta v$ is defined as $\cT_\nu = \nu + \Ker A_v$.  In Figure~\ref{fig:firstTarget}, in which $i = 1$, the affine subspace $\cT_\nu$ is the blue line, and the orange line is $w_0 s_i \cT_\nu = s_1 s_2 \cT_\nu$.

\begin{lemma}\label{lem:firstTransverse} Suppose that $\nu = \proj_i(\lambda - 2\check{\rho})$ is a non-integral Newton point, with standard representative $b_\nu = t^\eta s_i$, and denote the corresponding Kottwitz point by $\omega$.  
Then \[\zeta \in (w_0s_i \cT_\nu) \cap (\omega+R^\vee).\]
\end{lemma}

\begin{proof} 
Recall that $A_{s_i} = \proj_i$ by Lemma~\ref{lem:AwProperties}(4), in which case $s_i(\Ker A_{s_i}) = \Ker A_{s_i}$.  Since $\nu \in H_{\alpha_i}$ and hence is fixed by $s_i$, it follows by the definition of $\cT_\nu = \nu + \Ker A_{s_i}$ that $s_i \cT_\nu = \cT_\nu$.  Since $\zeta \in \omega + R^\vee$ by Lemma \ref{lem:translationZ}, it now suffices to show that $w_0 \zeta \in \cT_\nu$.  By Lemma~\ref{lem:top} and assumption, we have $\proj_i(w_0 \zeta) = \proj_i(\lambda - 2\check{\rho}) = \nu$. Combining this observation with  Proposition~\ref{prop:stdReps}(2), we get that $\proj_i(w_0 \zeta) = \nu = \proj_i(\eta)$.  Thus $w_0 \zeta - \eta \in \Ker A_{s_i}$.  From Corollary~\ref{cor:stdRepTransverse} we have $\eta \in \cT_\nu$, and hence $w_0 \zeta \in \cT_\nu$ as required.
\end{proof}

We now identify $z=t^\zeta s_{i_1}$ as a particular $\eW$-conjugate of the standard representative $b_\nu$.

\begin{corollary}\label{cor:firstConjugation}
Let $z = t^\zeta s_{i_1} \in \eW$.  Then:
\begin{enumerate}
    \item there exists an integer $d_i \in \Z$ such that  $s_i w_0 \zeta - \eta = d_i \alpha_i^\vee$; and
    \item if $y = t^\mu w_0 s_i \in \eW$ with $\mu = d_i w_0 s_i \eta \in Q^\vee$ for $d_i \in \Z$ as in (1), then $z=b_\nu^y$.
\end{enumerate}
\end{corollary}

\noindent For example, in Figure~\ref{fig:firstTarget} we have $d_i = 1$, and hence $\mu = w_0 s_i \eta = s_1 s_2 \eta$ with $y = t^\mu s_1 s_2$.  The vertex $\mu \in R^\vee$ and the alcove $\y$ are labeled in purple.  (Note that, coincidentally in this example,  $\lambda - 2\check{\rho} = \eta$.)

\begin{proof}
Put $u = w_0 s_i$.  Then $s_{i_1} = w_0 s_i w_0 = w_0 s_i s_i s_i w_0 = u s_i u^{-1}$. We also have $\zeta \in (u \cT_\nu) \cap (\omega +R^\vee)$ by Lemma~\ref{lem:firstTransverse}.  The remaining claims are then immediate by Proposition~\ref{prop:inverseRank1}.
\end{proof}

Our final corollary in this section combines the previous results to conclude that the gallery $\gamma_{\check{\rho}}$ is positively folded with respect to the chimney suggested by Corollary \ref{cor:firstConjugation}.

\begin{corollary}\label{cor:firstTarget} Let $x_0 =t^\lambda w_0 \in \eW$, and assume that $\x_0$ is in the shrunken dominant Weyl chamber $\Cfs$.
Suppose that  
\[
\proj_i(\lambda - 2\check{\rho}) = \nu,
\]
where $\nu$ is a non-integral Newton point with associated standard spherical parabolic subgroup $P_i$ of rank $1$, and denote the corresponding standard representative by $b_\nu = t^\eta s_i$.   

Let $\mu = d_i w_0 s_i \eta$ and $y = t^\mu w_0 s_i \in \eW$ be as in the statement of Corollary~\ref{cor:firstConjugation}. Then the gallery $\gamma_{\check{\rho}}:\fa \rightsquigarrow \bb_\nu^y$ has type $\vec{x}_0$ and is positively folded with respect to the $(P_i,y)$-chimney.
\end{corollary}

\noindent For example, a $(P_i,y)$-sector is shaded in purple in Figure~\ref{fig:firstTarget} with the alcove $\y$ based at $\mu$ outlined, and the orange gallery $\gamma_{\check{\rho}}$ is positively folded with respect to the $(P_i,y)$-chimney.

\begin{proof}  
By construction, $\gamma_{\check{\rho}}: \fa \rsa \z$ has type $\vec{x}_0$. Corollary \ref{cor:firstConjugation} says that $\z=z\fa_0$ where $z=b^y_\nu$, and so $\bb_\nu^y = b_\nu^y\fa_0$ is indeed the final alcove of $\gamma_{\check{\rho}}$. In order for $\gamma_{\check{\rho}}$ to be positively folded with respect to the $(P_i,y)$-chimney, by Lemma~\ref{lem:posFolded} it suffices to show that $\langle \alpha_{i_1}, \mu \rangle \geq 1$.  
We have
\[
\langle \alpha_{i_1}, \mu \rangle = \langle w_0 s_i \alpha_i, d_i w_0 s_i \eta \rangle = d_i \langle \alpha_i, \eta \rangle = d_i,
\]
since $\eta \in H_{\alpha_i,1}$ by Proposition \ref{prop:stdReps}. We thus aim to show that $d_i \geq 1$.  

As $s_i w_0 \zeta - \eta = d_i \alpha_i^\vee$ by Corollary \ref{cor:firstConjugation}(1), we obtain that $s_i w_0 \zeta - \eta \in H_{\alpha_i,2d_i}$, and hence $s_i w_0 \zeta \in H_{\alpha_i,2d_i +1}$ again using that $\eta \in H_{\alpha_i,1}$.  Thus
\[
2d_i + 1 = \langle \alpha_i, s_i w_0 \zeta \rangle = \langle -\alpha_i, -\check{\rho} + s_i(\lambda - \check{\rho}) \rangle = \langle -\alpha_i, -\check{\rho} \rangle + \langle \alpha_i, \lambda \rangle + \langle \alpha_i, -\check{\rho} \rangle = \langle \alpha_i, \lambda \rangle. 
\]
Now recall by Lemma \ref{lem:NPimageAsi} that $\proj_i(\lambda - 2\check{\rho})$ is a non-integral Newton point if and only if $\langle \alpha_i, \lambda-2\check{\rho} \rangle = 2k+1$ for some $k \in \Z$, which in turn happens if and only if $\lambda$ itself lies in an $\alpha_i$-hyperplane of odd index, since $\langle \alpha_i, 2\check{\rho} \rangle = 2$.  Finally, since $\lambda - 2\check{\rho}$ is dominant, we know that $\langle \alpha_i, \lambda \rangle \geq 2$.  Combining this observation with the fact that $\lambda$ lies in an $\alpha_i$-hyperplane of odd index, in fact we have $\langle \alpha_i, \lambda \rangle \geq 3$.  Therefore $d_i \geq 1$ above, 
as required.
\end{proof}

The main result of this section now immediately follows.

\begin{proof}[Proof of Theorem~\ref{thm:firstTarget}] Combine Corollary~\ref{cor:firstTarget} with Theorem~\ref{thm:ADLVChimneys}(1).  
\end{proof}


\section{Reaching the remaining targets}\label{sec:remainingTargets}

In this section, we modify the gallery $\gamma_{\check{\rho}}$ constructed in Section \ref{sec:firstTarget}, which is designed to reach the maximum Newton point associated to $\lambda-2\check{\rho}$ having standard parabolic subgroup $P_i$, in order to reach all other Newton points $\nu' \in \Conv(\sW(\lambda-2\check{\rho}))$ associated to the same parabolic subgroup.  The main idea is to apply the root operators of \cite{GaussentLittelmann} as in \cite[Sec.~6]{MST1}, in order to prove Theorem \ref{thm:remainingTargets}. The preparatory results which guarantee the availability of the root operators are developed in Section \ref{sec:rootOps}.  These root operators are then applied in Section \ref{sec:lower}, although the positivity of the folds only holds for larger Newton points.  These galleries are then further modified in Section \ref{sec:lowest}, in order to reach the smaller Newton points.

\begin{thm}\label{thm:remainingTargets}
Let $x_0=t^\lambda w_0\in \eW$, and assume that $\x_0$ is in the shrunken dominant Weyl chamber $\Cfs$.  Suppose that $\nu'$ is a non-integral Newton point with associated spherical standard parabolic subgroup $P_i$ of rank 1, and suppose that \[ \nu' \in \Conv(\sW(\lambda - 2\check{\rho})).\] Then for any $b' \in \eW$ such that $\nu_{b'}=\nu'$ and $\kappa_G(b') = \kappa_G(x_0)$, we have
$X_{x_0}(b') \neq \emptyset.$
\end{thm}

The assumptions on $\x_0$ throughout the remainder of this section are those from the statement of Theorem~\ref{thm:remainingTargets}; in particular, we do not assume that $\proj_i(\lambda - 2\check{\rho})$ is non-integral, as we did in Section~\ref{sec:conjugationFinal}.  As in Section \ref{sec:firstTarget}, although $\lambda \in \omega + R^\vee$ throughout, all explicit gallery constructions are illustrated in the 0-sheet of extended alcoves.

\subsection{Application of root operators}\label{sec:rootOps}

Let $\gamma_{\check{\rho}}:\fa \rightsquigarrow \z$ be the folded gallery of type $\vec{x}_0$ constructed in Section~\ref{sec:folding}.  In this section, we obtain a family of galleries from $\gamma_{\check{\rho}}$ using the root operators of~\cite{GaussentLittelmann}. Figures~\ref{fig:rootOps}--\ref{fig:rootOps3} continue the example from Figures~\ref{fig:gamma}--\ref{fig:firstTarget}, and one new gallery obtained by implementing the construction described in this section is shown in pink in Figure~\ref{fig:rootOps3}.

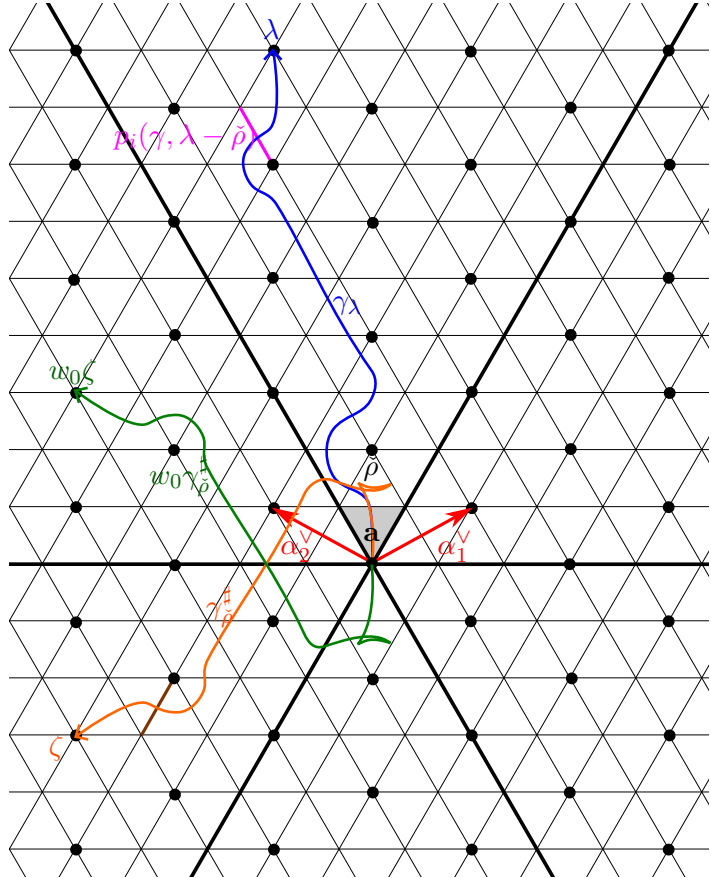
\begin{figure}[ht]
\centering
\begin{overpic}[width=0.6\textwidth]{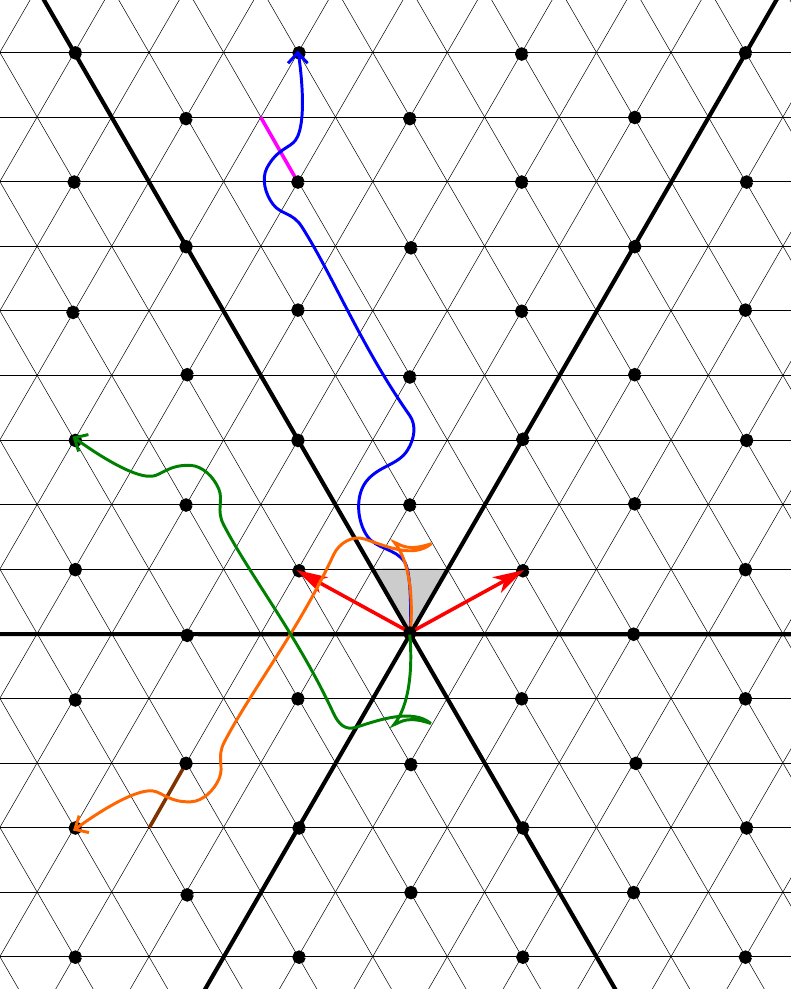}
\put(29,96){\color{blue}$\lambda$}
\put(37,65){\color{blue}$\gamma_\lambda$}
\put(12,84){\color{Fuchsia}$p_i(\gamma, \lambda - \check{\rho})$}
\put(40.5,38.5){\color{black}$\fa$}
\put(40.5,46){\color{black}$\check{\rho}$}
\put(49,37){\color{red}$\alpha_1^\vee$}
\put(31,37){\color{red}$\alpha_2^\vee$}
\put(4.5,14){\color{OrangeRed}$\zeta$}
\put(4.5,57){\color{DarkGreen}$w_0\zeta$}
\put(22.5,30){\color{OrangeRed}$\gamma_{\check{\rho}}^\sharp$}
\put(16,45){\color{DarkGreen}$w_0\gamma_{\check{\rho}}^\sharp$}
\end{overpic}
\caption{An example illustrating the gallery $w_0\gamma_{\check{\rho}}^{\sharp}$ to begin the application of root operators.}
\label{fig:rootOps}
\end{figure}

Let $\gamma_{\check{\rho}}^\sharp$ be the canonical vertex-to-vertex gallery associated to $\gamma_{\check{\rho}}$, with initial vertex the origin and final vertex $\zeta$ defined in \eqref{eq:zeta}.  By the same argument as in the proof of \cite[Prop.~6.7]{MST1}, $\gamma_{\check{\rho}}^\sharp$ has the type of a minimal gallery $\gamma_\lambda$ from the origin to $\lambda$, and $\gamma_\lambda$ is minimal in the sense defined in~\cite{GaussentLittelmann}.  In Figure~\ref{fig:rootOps}, the minimal gallery $\gamma_\lambda$ is blue, and $\gamma_{\check{\rho}}^\sharp$ is orange.

As in the proof of \cite[Prop.~6.7]{MST1}, we now act on $\gamma_{\check{\rho}}^{\sharp}$ on the left by $w_0$ to obtain the gallery $w_0 \gamma_{\check{\rho}}^\sharp$, which is positively folded with respect to the standard orientation, also has type $\gamma_\lambda$, and has final vertex $w_0 \zeta$.  In Figure~\ref{fig:rootOps}, the gallery $w_0 \gamma_{\check{\rho}}^\sharp$ is dark green.  

We now consider the application of root operators to $w_0\gamma_{\check{\rho}}^\sharp$. We first identify where the folds of $w_0\gamma_{\check{\rho}}^\sharp$ occur.  Recall that we denote by $\alpha_{i_2},\dots,\alpha_{i_n}$ the elements of $\Delta \setminus \{ \alpha_{i_1} \}$ in any order, where $\alpha_{i_1} = -w_0 \alpha_i$.  For $1 \leq j \leq n$, we now write $\alpha'_{i_j}$ for the positive simple root $-w_0 \alpha_{i_j}$.  Then $\alpha_i = \alpha'_{i_1}$, and we have $\Delta = \{ \alpha_{i_1}', \alpha'_{i_2},\dots,\alpha'_{i_n} \}$.  In Figure~\ref{fig:rootOps}, we have $i = 1$, $i_1 = 2$, and $i_2 = 1$, and so $\alpha'_{i_1} = \alpha_1$ and $\alpha'_{i_2} = \alpha_2$.

\begin{lemma}\label{lem:indexing1}  All folds in the gallery $w_0 \gamma_{\check{\rho}}^\sharp$ occur in hyperplanes of the form $H_{\alpha'_{i_j},-1}$, where $j \in \{1,\dots,n\}$.  In particular, the last $n-1$ folds of $w_0 \gamma_{\check{\rho}}^\sharp$ occur in the hyperplanes $H_{\alpha'_{i_n},-1}$, \dots, $H_{\alpha'_{i_2},-1}$, in that order.
\end{lemma}

\begin{proof}  By construction, the folds of $\gamma_{\check{\rho}}$, and thus of $\gamma_{\check{\rho}}^\sharp$, are in the hyperplanes $H_{\alpha_{k_1},1}$, \dots, $H_{\alpha_{k_{\ell - 1}},1}$, where $\vec{w_0} = s_{k_1} \cdots s_{k_\ell}$ is the reduced word for $w_0$ fixed in Section \ref{sec:minGallery}.  In particular, the last $n-1$ folds of $\gamma_{\check{\rho}}$, and thus of $\gamma_{\check{\rho}}^\sharp$, are in $H_{\alpha_{k_{\ell - n + 1}},1}, \dots, H_{\alpha_{k_{\ell - 1}},1}$, in that order, where we recall that $\alpha_{k_{\ell - n + 1}} = \alpha_{i_n}, \dots, \alpha_{k_{\ell - 1}} = \alpha_{i_2}$ by definition.  Thus, all folds of $\gamma_{\check{\rho}}^\sharp$ lie in hyperplanes $H_{\alpha_{i_j},1}$, where $j \in \{1,\dots,n\}$, and the last $n-1$ folds of $\gamma_{\check{\rho}}^\sharp$ are in $H_{\alpha_{i_n},1}, \dots, H_{\alpha_{i_2},1}$, in that order.  Then since $w_0 H_{\alpha_{i_j},1} = H_{w_0\alpha_{i_j},1} = H_{-\alpha'_{i_j},1} = H_{\alpha'_{i_j},-1}$, the folds of $w_0 \gamma_{\check{\rho}}^\sharp$ are as claimed.
\end{proof}

As in \cite[Sec.~4.2]{MST1}, for $\alpha$ any positive simple root, $m(w_0\gamma_{\check{\rho}}^\sharp, \alpha)$ denotes the minimal integer $m$ such that the gallery $w_0\gamma_{\check{\rho}}^\sharp$ contains a panel or a first or final vertex in $H_{\alpha,m}$.  The next result will be used to determine how many times certain root operators can be applied to $w_0\gamma_{\check{\rho}}^\sharp$.

\begin{lemma}\label{lem:indexingM}
For each $2 \leq j \leq n$:
\begin{enumerate}
\item $m(w_0\gamma_{\check{\rho}}^\sharp, \alpha_{i_j}') =-1$; and 
\item $M_j' := \langle \alpha_{i_j}', w_0\zeta \rangle \geq 0$, so in particular, $M_j' \neq m(w_0\gamma_{\check{\rho}}^\sharp, \alpha_{i_j}')$.
\end{enumerate}
\end{lemma}

\noindent For example, in Figure~\ref{fig:rootOps}, where $\alpha_{i_2'} = \alpha_2$, we have $m(w_0\gamma_{\check{\rho}}^\sharp,\alpha_{i_2}') = -1$ and $\langle \alpha_{i_2}',w_0\zeta \rangle = 6$.

\begin{proof}  By the second statement in Lemma~\ref{lem:indexing1}, for $2 \leq j \leq n$ the gallery $w_0 \gamma_{\check{\rho}}^\sharp$ has at least one fold in $H_{\alpha_{i_j}',-1}$.  Hence $m(w_0\gamma_{\check{\rho}}^\sharp, \alpha_{i_j}') \leq -1$.  Suppose by contradiction that $m(w_0\gamma_{\check{\rho}}^\sharp, \alpha_{i_j}') < -1$.  By the first statement in Lemma~\ref{lem:indexing1}, this implies that $w_0\gamma_{\check{\rho}}^\sharp$ must either cross or have final vertex in an $\alpha_{i_j}'$-hyperplane of index strictly less than $-1$.  Hence $\gamma_{\check{\rho}}^\sharp$ must either cross or have a final vertex in an $\alpha_{i_j}$-hyperplane of index strictly greater than~$1$.  

All crossings of $\gamma_{\check{\rho}}^\sharp$ are crossings of $\gamma_{\check{\rho}}$, so we start by considering $\gamma_{\check{\rho}}$.  The initial minimal subgallery of $\gamma_{\check{\rho}}$ from $\fa$ to the $w_0$-position at $\check{\rho}$ does not cross any simple hyperplanes.  The part of $\gamma_{\check{\rho}}$ after all its folds (all of which occur in panels of the $w_0$-position at $\check{\rho}$) consists of a minimal subgallery $\gamma_1$ starting from the $w_0$-position at $\check{\rho}$, running to the alcove at $\check{\rho}$ which is the image of the identity position (again at $\check{\rho}$) under the reflections in $H_{\alpha_{k_1,1}},\dots,H_{\alpha_{k_{\ell - 1}}, 1}$, followed by a minimal subgallery $\gamma_2$ from this alcove to $\z$.  Thus the two minimal subgalleries $\gamma_1$ and $\gamma_2$ meet at the alcove in the $w$-position at $\check{\rho}$, where
\[
w = s_{k_{\ell - 1}} \cdots s_{k_1}.
\]
From the proof of Lemma~\ref{lem:sphericalZ}, we have $s_{k_{\ell - 1}} \cdots s_{k_1} w_0 = s_{i_1}$, and so $w = s_{i_1}w_0=w_0s_i$.

We now show that the minimal gallery $\gamma_1$, from the $w_0$-position to the $w$-position at $\check{\rho}$, does not cross $H_{\alpha_{i_j},1}$ for any $2 \leq j \leq n$.  After applying $t^{-\check{\rho}}$ and then $w_0$ on the left, this is equivalent to showing that a minimal gallery from the base alcove $\fa$ to the alcove $w_0w\fa$ does not cross the hyperplane $H_{w_0\alpha_{i_j}}$.  But $w_0 w = w_0 s_{i_1} w_0 = s_i$, so a minimal gallery from $\fa$ to $w_0w\fa$ only crosses $H_{\alpha_i} = H_{-\alpha_i}$.  As $w_0 \alpha_{i_j}$ is negative, we obviously cannot have $w_0\alpha_{i_j} = \alpha_i$. If $w_0\alpha_{i_j} = -\alpha_i$, this contradicts $\alpha_{i_j} \in \Delta \setminus \{ \alpha_{i_1}\}$.  Hence, $\gamma_1$ does not cross $H_{\alpha_{i_j},1}$ for any $2 \leq j \leq n$. 

The minimal subgallery $\gamma_2$, from the $w$-position at $\check{\rho}$ to its final alcove $\z$, is the image of a minimal gallery from $t^{\check{\rho}} \fa$ to $\x_0$ under a sequence of reflections in hyperplanes through $\check{\rho}$.  Thus $\gamma_2$ lies in the image of the Weyl chamber $\Cw$ under the translation $t^{\check{\rho}}$.  So the subgallery $\gamma_2$ of $\gamma_{\check{\rho}}$ does not contain any panels in hyperplanes through $\check{\rho}$.  

Combining these observations, every panel of $\gamma_{\check{\rho}}$ (and thus of $\gamma_{\check{\rho}}^\sharp$) contained in $H_{\alpha_{i_j},1}$ is a panel of the alcove $t^{\check{\rho}} w_0\fa$ in which $\gamma_{\check{\rho}}$ has a fold.  Therefore $\gamma_{\check{\rho}}^\sharp$ cannot cross an $\alpha_{i_j}$-hyperplane of index strictly greater than $1$.   Since $\lambda$ does not lie on any hyperplane through $\check{\rho}$ by hypothesis, and $\zeta$ is the image of $\lambda$ under a sequence of reflections in hyperplanes through $\check{\rho}$, we also note that the final vertex $\zeta$ of $\gamma_{\check{\rho}}^\sharp$ does not lie on any hyperplane through $\check{\rho}$. It follows that $\zeta$ is not contained in an $\alpha_{i_j}$-hyperplane of index strictly greater than $1$.  This completes the proof of (1).

For (2), as $\alpha_{i_j}' = -w_0 \alpha_{i_j}$, this is equivalent to showing $\langle \alpha_{i_j}, \zeta \rangle \leq 0$.  Now from the previous paragraph, $\gamma_{\check{\rho}}^\sharp$ does not cross $H_{\alpha_{i_j},1}$, and $\zeta$ is not contained in $H_{\alpha_{i_j},1}$.  Since $\gamma_{\check{\rho}}^\sharp$ has first vertex the origin and final vertex $\zeta$, it follows that $\langle \alpha_{i_j}, \zeta \rangle \leq 0$, as required.
\end{proof}

Before continuing with our results concerning the application of root operators, we record the following lemma for later use in Section \ref{sec:lowest}, since its proof uses some arguments from the proof of Lemma~\ref{lem:indexingM}.  Recall from the construction of the minimal gallery $\gamma:\fa \rightsquigarrow \x_0$ in Section~\ref{sec:minGallery} that $\gamma$ ends with a minimal subgallery from the identity position at $\lambda - \check{\rho}$ to $\x_0$, and that the crossing of $\gamma$ immediately before this subgallery is in the hyperplane through $\lambda - \check{\rho}$ parallel to $H_{\alpha_i}$.  Write $p_i(\gamma,\lambda - \check{\rho})$ for the (unique) panel of $\gamma$ at which this crossing occurs.  In Figure~\ref{fig:rootOps}, the panel $p_i(\gamma,\lambda - \check{\rho})$ is shown in pink, and its image in $\gamma_{\check{\rho}}$ is shown in brown.

\begin{lemma}\label{lem:tail_gammaRho}  The image of the panel $p_i(\gamma,\lambda - \check{\rho})$ of $\gamma$ in the gallery $\gamma_{\check{\rho}}$ lies in the hyperplane $H_{\alpha_{i_1},k}$ where $k = \langle \alpha_{i_1},\zeta\rangle - 1$.  Moreover, $\gamma_{\check{\rho}}$ crosses this image from the antidominant to the dominant side, and has no folds after this crossing.
\end{lemma}

\noindent For example, in Figure \ref{fig:rootOps}, compute $\langle \alpha_{i_1},\zeta\rangle =\langle \alpha_2, \zeta\rangle =3$, and note that the brown panel is indeed contained in $H_{\alpha_2, 2}$.

\begin{proof} By construction of $\gamma$, the panel $p_i(\gamma,\lambda - \check{\rho})$ is contained in its minimal subgallery from  $t^{\check{\rho}}\fa$ to $\x_0$.  Now by construction of $\gamma_{\check{\rho}}$, the image in $\gamma_{\check{\rho}}$ of this minimal subgallery of $\gamma$ is a minimal subgallery which runs from the $w$-position at $\check{\rho}$ to $\z$, where $w = s_{i_1}w_0 = w_0 s_i$ is as in the proof of Lemma~\ref{lem:indexingM}.  Therefore, $\gamma_{\check{\rho}}$ crosses the image of $p_i(\gamma,\lambda - \check{\rho})$ and has no folds after this crossing. 

We now determine the hyperplane containing the image of $p_i(\gamma,\lambda - \check{\rho})$.  Using the $\sW$-invariance of the pairing, this hyperplane will be $wH_{\alpha_i,j} = w_0 s_i H_{\alpha_i,j}$ where $j = \langle \alpha_i,\lambda - \check{\rho} \rangle$.  Since $w_0 s_i \alpha_i = \alpha_{i_1}$, it remains to show that $\langle \alpha_{i_1},\zeta\rangle - 1 = \langle \alpha_i,\lambda - \check{\rho} \rangle$.  We compute that
\begin{eqnarray*}
\langle \alpha_{i_1}, \zeta \rangle & = & \langle \alpha_{i_1}, \check{\rho} + w_0 s_i(\lambda - \check{\rho}) \rangle \\
& = & \langle \alpha_{i_1}, \check{\rho} \rangle + \langle  \alpha_i, \lambda - \check{\rho} \rangle \\
& = & 1 + \langle \alpha_i, \lambda - \check{\rho} \rangle,
\end{eqnarray*}
as required.

Finally, we show that the crossing of $\gamma_{\check{\rho}}$ at the image of $p_i(\gamma,\lambda - \check{\rho})$ is from the antidominant to the dominant side.  Now as observed in the proof of Lemma~\ref{lem:indexingM}, the minimal subgallery of $\gamma_{\check{\rho}}$ starting in the $w$-position at $\check{\rho}$ is contained in the image of the Weyl chamber $\Cw$ under $t^{\check{\rho}}$.  Since $w = s_{i_1}w_0$, the Weyl chamber $\Cw$ is separated from the antidominant Weyl chamber exactly by an $\alpha_{i_1}$-hyperplane.  Therefore any crossing of an $\alpha_{i_1}$-hyperplane by the subgallery of $\gamma_{\check{\rho}}$ starting in the $w$-position at $\check{\rho}$ is from the antidominant to the dominant side.
\end{proof}

We are now prepared to apply the available root operators to the vertex-to-vertex gallery $w_0 \gamma_{\check{\rho}}^\sharp$.  We use all notation as in \cite[Sec.~4.2]{MST1}, and we and refer the reader in particular to Definition 4.14 for a reminder about the precise definition of $f_\alpha$, and to and Lemma 4.17 for several key properties of root operators, following \cite{GaussentLittelmann}. 

We first consider the application of the single root operator $f_{\alpha_{i_2}'}$ to $w_0 \gamma_{\check{\rho}}^\sharp$.  For $M$ a non-negative integer, we write $f_{\alpha_{i_2}'}^M$ for an $M$-fold application of $f_{\alpha_{i_2}'}$.  Recall that $\gamma_{\check{\rho}}$ has $\ell(w_0) - 1$ folds, and hence $w_0 \gamma_{\check{\rho}}^\sharp$ has $\ell(w_0) - 1$ folds as well.  In Figure~\ref{fig:rootOps2}, we have $f_{\alpha_{i_2}'} = f_{\alpha_2}$, $M_2'= \langle \alpha'_{i_2}, w_0\zeta \rangle = \langle \alpha_2,w_0\zeta \rangle = 6$, and $\ell(w_0)-1 =2$.  We depict  $f_{\alpha_2}(w_0 \gamma_{\check{\rho}}^\sharp)$ in light green, $f_{\alpha_2}^2(w_0 \gamma_{\check{\rho}}^\sharp)$ in light blue, and $f_{\alpha_2}^6(w_0 \gamma_{\check{\rho}}^\sharp)$ in purple. 

\begin{figure}[ht]
\centering
\begin{overpic}[width=0.7\textwidth]{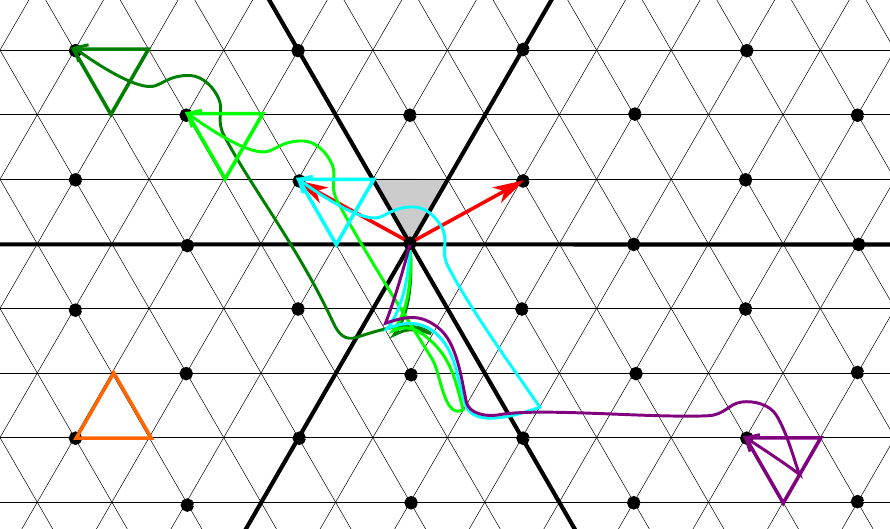}
\put(45,36){\color{black}$\fa$}
\put(44.5,26){\color{black}$\w_0$}
\put(54,33){\color{red}$\alpha_1^\vee$}
\put(36,33){\color{red}$\alpha_2^\vee$}
\put(5,7.5){\color{OrangeRed}$\zeta$}
\put(12,12){\color{OrangeRed}$\z$}
\put(4.5,55.5){\color{DarkGreen}$w_0\zeta$}
\put(18,52){\color{DarkGreen}$w_0\gamma_{\check{\rho}}^\sharp$}
\end{overpic}
\caption{The application of the root operators $f^M_{\alpha_{i_2}'} = f^M_{\alpha_2}$ to $w_0 \gamma_{\check{\rho}}^\sharp$ for $M \in \{1,2,6\}$.}
\label{fig:rootOps2}
\end{figure}
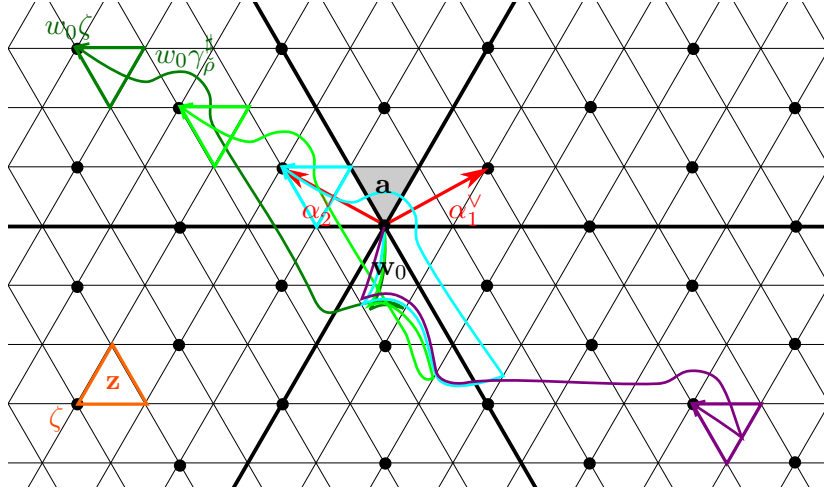

\begin{lemma}\label{lem:singleRootOp}  For $1 \leq k \leq M_2'+1$ the root operator $f_{\alpha_{i_2}'}^{k}$ is defined for the gallery $w_0 \gamma_{\check{\rho}}^\sharp$, while $f_{\alpha_{i_2}'}^{M_2' + 2}$ is not defined for this gallery.  Moreover for $1 \leq k \leq M_2'$, the gallery $f_{\alpha_{i_2}'}^{k}(w_0\gamma_{\check{\rho}}^\sharp)$ has:
\begin{enumerate}
\item first alcove $\w_0$; 
\item final alcove the translate of $w_0\z$ by $-k(\alpha_{i_2}')^\vee$; and
\item the same first $\ell(w_0) - 2$ folds as $w_0 \gamma_{\check{\rho}}^\sharp$.
\end{enumerate}
\end{lemma}

\begin{proof}  Let $m_2 = m(w_0\gamma_{\check{\rho}}^\sharp,\alpha_{i_2}')$.  The gallery $w_0\gamma_{\check{\rho}}^\sharp$ has final vertex $w_0 \zeta$, and so by \cite[Lem.~4.17]{MST1} the operator $f_{\alpha_{i_2}'}$ is not defined for this gallery if and only if $m_2 = M_2'$.  Hence Lemma~\ref{lem:indexingM}(2) establishes that $f_{\alpha_{i_2}'}$ can be applied to $w_0\gamma_{\check{\rho}}^\sharp$ at least once.

By Lemma~\ref{lem:indexingM}(1) we have $m_2 = -1$.  The operator $f_{\alpha_{i_2}'}$ thus has the effect of fixing all parts of $w_0\gamma_{\check{\rho}}^\sharp$ which lie before its last panel contained in $H_{\alpha_{i_2}',-1}$, say $p_l$.  Thus, the first alcove of $f_{\alpha_{i_2}'}(w_0\gamma_{\check{\rho}}^\sharp)$ remains $\w_0$.  Also, by Lemma~\ref{lem:indexing1}, the last fold of $w_0 \gamma_{\check{\rho}}^\sharp$ is in $H_{\alpha_{i_2}',-1}$, and so this fold occurs in $p_l$, and the gallery $f_{\alpha_{i_2}'}(w_0\gamma_{\check{\rho}}^\sharp)$ has the same first $\ell(w_0) - 2$ folds  as $w_0 \gamma_{\check{\rho}}^\sharp$.  Now the operator $f_{\alpha_{i_2}'}$ has the effect of reflecting in $H_{\alpha_{i_2}',-1}$ the portion of $w_0\gamma_{\check{\rho}}^\sharp$ lying between $p_l$ and its next panel (or final vertex) contained in $H_{\alpha_{i_2}',0}$, say $p_{l'}$, and translating by $-(\alpha_{i_2}')^\vee$ the remainder of $w_0\gamma_{\check{\rho}}^\sharp$.  The final vertex $w_0\zeta$ of $w_0\gamma_{\check{\rho}}^\sharp$ lies in $H_{\alpha_{i_2}',M_2'}$, and by  Lemma~\ref{lem:indexingM}(2), we have $M_2' \geq 0$.  It follows that the minimal index of an $\alpha_{i_2}'$-hyperplane containing a panel or final vertex of $f_{\alpha_{i_2}'}(w_0\gamma_{\check{\rho}}^\sharp)$ is 
\[
m(f_{\alpha_{i_2}'}(w_0\gamma_{\check{\rho}}^\sharp),\alpha_{i_2}') = m(w_0\gamma_{\check{\rho}}^\sharp, \alpha_{i_2}') -1 = m_2-1 = -2,
\]
and the gallery $f_{\alpha_{i_2}'}(w_0\gamma_{\check{\rho}}^\sharp)$ has final vertex $w_0\zeta - (\alpha_{i_2}')^\vee$.  Note that 
\[
\langle \alpha_{i_2}', w_0\zeta - (\alpha_{i_2}')^\vee \rangle = M_{2}' - 2.
\]

We now inductively assume that the operator $f_{\alpha_{i_2}'}^{k-1}$ can be applied to $w_0\gamma_{\check{\rho}}^\sharp$, where $2 \leq k \leq M_2' + 1$, so that the minimal index of an $\alpha_{i_2}'$-hyperplane containing a panel or final vertex of $f^{k-1}_{\alpha_{i_2}'}(w_0\gamma_{\check{\rho}}^\sharp)$ is 
\begin{equation}\label{eq:mk}
    m(f^{k-1}_{\alpha_{i_2}'}(w_0\gamma_{\check{\rho}}^\sharp),\alpha_{i_2}') = m(w_0\gamma_{\check{\rho}}^\sharp, \alpha_{i_2}') -(k-1) = m_2-(k-1) = -k,
\end{equation}
and the gallery $f^{k-1}_{\alpha_{i_2}'}(w_0\gamma_{\check{\rho}}^\sharp)$ has final vertex $w_0\zeta - (k-1)(\alpha_{i_2}')^\vee$ satisfying  
\[
\langle \alpha_{i_2}', w_0\zeta - (k-1)(\alpha_{i_2}')^\vee \rangle = M_2' - 2(k-1).
\]
Then again by \cite[Lem.~4.17]{MST1}, the root operator $f_{\alpha_{i_2}'}$ can be applied to the gallery $f^{k-1}_{\alpha_{i_2}'}(w_0\gamma_{\check{\rho}}^\sharp)$ provided that $
-k \neq M_2' - 2(k-1)$, or equivalently $k \neq M_2' + 2$.  Therefore, by induction $f^k_{\alpha_{i_2}'}$ can be applied to $w_0\gamma_{\check{\rho}}^\sharp$ for all $1 \leq k \leq M_{2}' + 1$, while $f_{\alpha_{i_2}'}^{M_2' + 2}$ is not defined for this gallery.

Suppose now that $1 \leq k \leq M_2'$.  The root operator $f_{\alpha_{i_2}'}$ fixes the first vertex of any gallery to which it can be applied, and from \eqref{eq:mk} we have 
\[
m(f^{k-1}_{\alpha_{i_2}'}(w_0\gamma_{\check{\rho}}^\sharp),\alpha_{i_2}') = -k < 0.
\]
Hence, $f^{k-1}_{\alpha_{i_2}'}(w_0\gamma_{\check{\rho}}^\sharp)$ contains a panel (or has a final vertex) which lies in $H_{\alpha_{i_2}',-k}$ with $-k < 0$.  Therefore $f_{\alpha_{i_2}'}$ fixes the first alcove of $f^{k-1}_{\alpha_{i_2}'}(w_0\gamma_{\check{\rho}}^\sharp)$, and so by induction the first alcove of $f^k_{\alpha_{i_2}'}(w_0\gamma_{\check{\rho}}^\sharp)$ is $\w_0$, and $f^k_{\alpha_{i_2}'}(w_0\gamma_{\check{\rho}}^\sharp)$ has the same first $\ell(w_0) - 2$ folds as $w_0\gamma_{\check{\rho}}^\sharp$.

For the final alcove, the final vertex of $f^{k-1}_{\alpha_{i_2}'}(w_0\gamma_{\check{\rho}}^\sharp)$ lies in the $\alpha_{i_2}'$-hyperplane of index $M_2' - 2(k-1)$.  By the same argument as in the proof of \cite[Lem.~6.10]{MST1}, provided that $M_{2}' - 2(k-1) > -k + 1$, or equivalently $M_2' > k-1$, then the effect of $f_{\alpha_{i_2}'}$ on the final alcove of $f^{k-1}_{\alpha_{i_2}'}(w_0\gamma_{\check{\rho}}^\sharp)$ will be to simply translate it by $-(\alpha_{i_2}')^\vee$.  Since we are assuming $k \leq M_2'$ in this part of the lemma, the result follows by induction.
\end{proof}

The next result, which now concerns the application of root operators corresponding to multiple simple roots, is proved using similar inductive arguments to Lemma~\ref{lem:singleRootOp}, and so we leave the details as an exercise for the reader.  For $2 \leq j \leq n$, recall from Lemma \ref{lem:indexingM} that $M_j' = \langle \alpha_{i_j}', w_0 \zeta \rangle \geq 0$, and now denote by $\sigma_1 = w_0 \gamma_{\check{\rho}}^\sharp$.

\begin{corollary}\label{cor:allRootOps}  
Let $2 \leq j \leq n$ and suppose for all $2 \leq k \leq j$ that $c_k$ is an integer satisfying $0 \leq c_k \leq M'_{k}$.   Then the operator $f_{\alpha_{i_j}'}^{c_j}$ can be applied to the gallery  
\[ \sigma_{j-1} := f_{\alpha_{i_{j-1}}'}^{c_{j-1}}\left(f_{\alpha_{i_{j-2}}'}^{c_{j-2}}\left(\cdots\left(f_{\alpha_{i_{2}}'}^{c_{2}}\left(w_0\gamma_{\check{\rho}}^\sharp\right)\right)\cdots \right)\right). \] 
Moreover, the gallery $\sigma_j:= f_{\alpha_{i_j}'}^{c_j}(\sigma_{j-1})$ has:
\begin{enumerate}
\item first alcove $\w_0$; 
\item final alcove the translate of $w_0\z$ by $-\sum_{k=2}^j c_k (\alpha_{i_k}')^\vee$; and
\item the same first $\ell(w_0) - j$ folds as $w_0 \gamma_{\check{\rho}}^\sharp$.
\end{enumerate}
\end{corollary}

\begin{proof}
 Follow the proof of Lemma \ref{lem:singleRootOp}.  More precisely, apply the same arguments as in the proofs for Lemma~6.11 and Corollary~6.12 of~\cite{MST1}, after swapping $e$- and $f$-operators, and the corresponding functions $p$ and $q$ determining the number of times they can be applied. 
\end{proof}

We now fix a collection of integers $(c_2,\dots,c_n)$ such that $0 \leq c_j \leq M'_j$ for each $j$.  Let $\zeta' = \zeta'(c_2,\dots,c_n)$ be given by 
\begin{equation}\label{eq:w0zeta'}
w_0\zeta' = w_0\zeta - \sum_{j=2}^n c_j (\alpha'_{i_j})^\vee
\end{equation}
and define $\sigma = \sigma_n = f_{\alpha_{i_{n}}'}^{c_{n}}\left(f_{\alpha_{i_{n-1}}'}^{c_{n-1}}\left(\cdots\left(f_{\alpha_{i_{2}}'}^{c_{2}}\left(w_0\gamma_{\check{\rho}}^\sharp\right)\right)\cdots\right)\right)$.  By Corollary \ref{cor:allRootOps}, the gallery $\sigma$ has first vertex the origin and final vertex $w_0\zeta'$,  in addition to
\begin{enumerate}
    \item first alcove $\w_0$;
    \item final alcove the translate of $w_0 \z$ given by $w_0\z - \sum_{j=2}^n c_j(\alpha'_{i_j})^\vee;$ and
    \item the same first $\ell(w_0) - n$ folds as $w_0 \gamma_{\check{\rho}}^\sharp$.
\end{enumerate}
In Figure~\ref{fig:rootOps2}, letting $c_2 = 2$ for example, the gallery $\sigma$ is the light blue gallery. Since $w_0\gamma_{\check{\rho}}^\sharp$ is positively folded with respect to the standard orientation, and root operators take positively folded galleries to positively folded galleries (see \cite[Lem.~4.17]{MST1}), then $\sigma$ is positively folded with respect to the standard orientation as well.  

Now define $\gamma_{\check{\rho}}(c_2,\dots,c_n)$ to be the gallery obtained from $\sigma$ by first applying $w_0$ on the left, and then removing the first and last vertex of the resulting gallery running from the origin to $\zeta'$.  Applying $w_0$ to the properties of $\sigma$ listed above,  we see that $\gamma_{\check{\rho}}(c_2,\dots,c_n)$ has
\begin{enumerate}
    \item first alcove $\fa$;
    \item final alcove
\begin{equation}\label{eq:z'}
    \z'= \z'(c_2,\dots,c_n):=\z - w_0\sum_{j=2}^n c_j(\alpha'_{i_j})^\vee = \z + \sum_{j=2}^n c_j \alpha_{i_j}^\vee = t^\zeta s_{i_1} + \sum_{j=2}^n c_j \alpha_{i_j}^\vee = t^{\zeta'}s_{i_1},
\end{equation}
where we define
\begin{equation}\label{eq:zeta'}
    \zeta' = \zeta + \sum_{j=2}^n c_j \alpha_{i_j}^\vee; \quad \text{and}
\end{equation}
\item the same first $\ell(w_0) - n$ folds as $\gamma_{\check{\rho}}$.
\end{enumerate}
In particular, note by comparing Lemmas \ref{lem:sphericalZ} and \ref{lem:translationZ} that the final alcove $\z'$ is simply translated from the final alcove $\z$ of $\gamma_{\check{\rho}}$ by the combination of simple roots prescribed by $(c_2, \dots, c_n)$. In Figure~\ref{fig:rootOps3}, where $c_2 = 2$, $i_1 = 2$,  and $i_2 = 1$, the gallery $\gamma_{\check{\rho}}(2)$, its final alcove $\z' = \z + 2\alpha_1^\vee$, and its final vertex $\zeta' = \zeta + 2\alpha_1^\vee$ are shown in pink.

Moreover, $\gamma_{\check{\rho}}(c_2,\dots,c_n)$ is positively folded with respect to the opposite standard orientation, as was shown for $\gamma_{\check{\rho}}$ in the proof of Lemma \ref{lem:posFolded}.  To determine whether or not $\gamma_{\check{\rho}}(c_2, \dots, c_n)$ remains positively folded with respect to the same $(P_i,y)$-chimney as $\gamma_{\check{\rho}}$, we thus need to identify those folds which occur in $\alpha_{i_1}$-hyperplanes.

\begin{figure}[ht]
\centering
\begin{overpic}[width=0.5\textwidth]{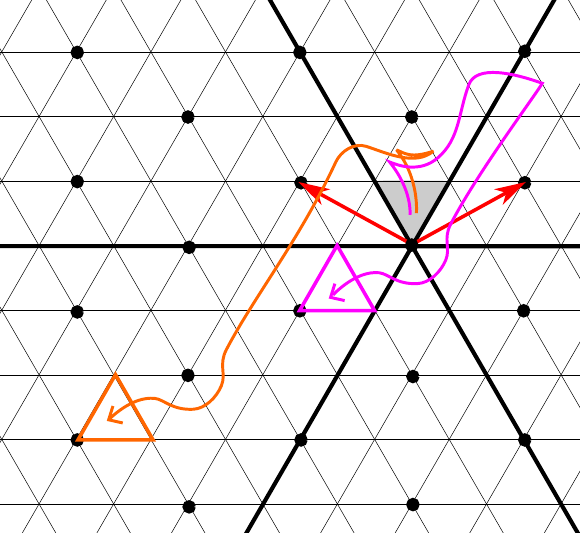}
\put(70,55){\color{black}$\fa$}
\put(83,52){\color{red}$\alpha_1^\vee$}
\put(55,52){\color{red}$\alpha_2^\vee$}
\put(9,13){\color{OrangeRed}$\zeta$}
\put(16,17){\color{OrangeRed}$\z$}
\put(47,35){\color{DeepPink}$\zeta'$}
\put(54,40){\color{DeepPink}$\z'$}
\put(40,42){\color{OrangeRed}$\gamma_{\check{\rho}}$}
\put(78,81){\color{DeepPink}$\gamma_{\check{\rho}}(2)$}
\end{overpic}
\caption{The gallery $\gamma_{\check{\rho}}(2)$ obtained from $\gamma_{\check{\rho}}$ by application of root operators.}
\label{fig:rootOps3}
\end{figure}
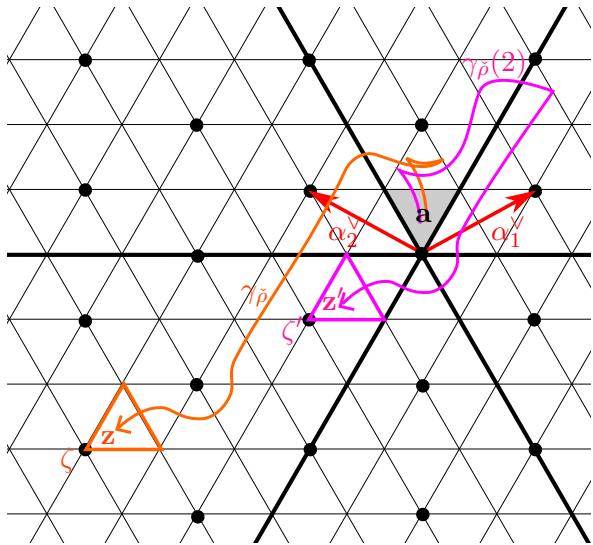

\begin{prop}\label{prop:folds}  
The only folds of $\gamma_{\check{\rho}}(c_{2},\dots,c_n)$ in an $\alpha_{i_1}$-hyperplane are its folds in $H_{\alpha_{i_1},1}$ which are folds preserved from the original gallery $\gamma_{\check{\rho}}$.  
\end{prop}

\begin{proof}  As $\gamma_{\check{\rho}}$ has $\ell(w_0) - 1$ folds, and $\gamma_{\check{\rho}}(c_2,\dots,c_n)$ has the same first $\ell(w_0) - n$ folds as $\gamma_{\check{\rho}}$, we only need to consider the last $n - 1$ folds of $\gamma_{\check{\rho}}(c_2,\dots,c_n)$.  Recall from the proof of Lemma \ref{lem:indexing1} that the last $n-1$ folds of $\gamma_{\check{\rho}}$ are in hyperplanes of types $\alpha_{i_n},\dots,\alpha_{i_2}$, and so by the construction of $\gamma_{\check{\rho}}(c_2,\dots,c_n)$, any folds it has other than the first $\ell(w_0) - n$ will be in hyperplanes  of types
\[
s_{i_n}^{\varepsilon_n} \cdots s_{i_{j+1}}^{\varepsilon_{j+1}}(\alpha_{i_j})
\] 
where $j \in \{2,\dots,n\}$ and $\varepsilon_l \in \{ 0,1\}$ for $j+1 \leq l \leq n$.  We will complete the proof by showing that this expression cannot equal $\alpha_{i_1}$. 

First recall that $\{\alpha_{i_2}, \dots, \alpha_{i_n}\} = \Delta \backslash \{\alpha_{i_1}\}$ so that $\alpha_{i_{j}} \neq \alpha_{i_1}$ when $j \in \{2,\dots,n\}$.  Applying any simple reflection $s_{i_l}$ for $j+1 \leq l \leq n$ to the simple root $\alpha_{i_j}$ results in a $\Z$-linear combination of $\alpha_{i_{j}}$ and $\alpha_{i_l}$, where the coefficient of $\alpha_{i_{j}}$ is strictly positive, and the coefficient of $\alpha_{i_l}$ is non-negative.
So applying a sequence of simple reflections other than $s_{i_{j}}$ to $\alpha_{i_{j}}$ results in a linear combination of simple roots in which the coefficient of $\alpha_{i_{j}}$ is strictly positive.  This linear combination thus cannot equal $\alpha_{i_1}$, since the simple roots are a basis.  
\end{proof}

As we did for $\gamma_{\check{\rho}}$ in Lemma \ref{lem:posFolded}, we can now identify a family of elements $y' \in \eW$ such that $\gamma_{\check{\rho}}(c_2, \dots, c_n)$ is positively folded with respect to the $(P_i, y')$-chimney. 

\begin{corollary}\label{cor:posFolded}
Let  $y' = t^{\mu'} w_0 s_i \in \eW$ where $\mu' \in Q^\vee$.  If $\langle \alpha_{i_1}, \mu' \rangle \geq 1$, then 
the gallery $\gamma_{\check{\rho}}(c_{2},\dots,c_n)$ is positively folded with respect to the $(P_i,y')$-chimney.
\end{corollary}

\begin{proof}
Apply Proposition~\ref{prop:folds} and arguments as in the proof of Lemma~\ref{lem:posFolded}.
\end{proof}

\subsection{Reaching the lower targets}\label{sec:lower}

We now construct the gallery required to prove Theorem~\ref{thm:remainingTargets} for ``many" of the Newton points $\nu'$ in its statement.  The main result in this section is Proposition~\ref{prop:lower}, which implies that for all $\nu'$ in the statement of Theorem~\ref{thm:remainingTargets}, there is a gallery $\gamma_{\check{\rho}}(c_2, \dots, c_n)$ whose final alcove is a certain $\eW$-conjugate of the standard representative $b_{\nu'}$. In Corollary~\ref{cor:remainingTargets}, we make precise those Newton points $\nu'$ for which this completes the proof of Theorem~\ref{thm:remainingTargets}; there is some subtlety involved in guaranteeing that this gallery $\gamma_{\check{\rho}}(c_2, \dots, c_n)$ is positively folded with respect to the corresponding $(P_i,y')$-chimney, as foreshadowed by Remark \ref{rem:chimney}.  Figure~\ref{fig:lower} illustrates the arguments in this section, and continues the example followed in Figures~\ref{fig:gamma}--\ref{fig:rootOps3}. Recall that in this running example $i = 1$, and it so happens that $\lambda - 2\check{\rho}=\eta$.

\begin{figure}[ht]
\centering
\begin{overpic}[width=0.65\textwidth]{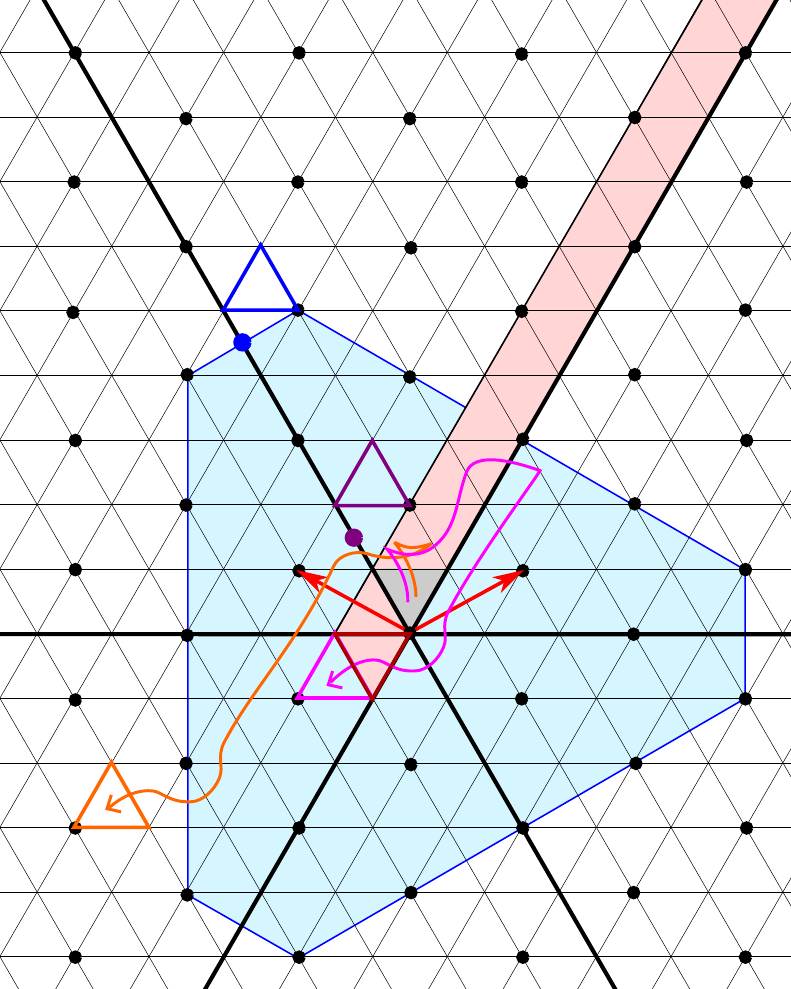}
\put(29.5,96){\color{blue}$\lambda$}
\put(24,66.5){\color{blue}$\nu$}
\put(25.5,71){\color{blue}$b_\nu$}
\put(29.5,71.5){\color{blue}$\eta$}
\put(35.5,46.5){\color{Purple}$\nu'$}
\put(36.5,51){\color{Purple}$b_{\nu'}$}
\put(41,51){\color{Purple}$\eta'$}
\put(40.5,38.5){\color{black}$\fa$}
\put(40.5,46){\color{black}$\check{\rho}$}
\put(49,37){\color{red}$\alpha_1^\vee$}
\put(31,37){\color{red}$\alpha_2^\vee$}
\put(4.5,14){\color{OrangeRed}$\zeta$}
\put(4.5,57){\color{OrangeRed}$w_0\zeta$}
\put(22.5,30){\color{OrangeRed}$\gamma_{\check{\rho}}$}
\put(27,27){\color{DeepPink}$\zeta'$}
\put(33,32){\color{DeepPink}$\z'$}
\put(28,44){\color{DeepPink}$w_0\zeta'$}
\put(46,55){\color{DeepPink}$\gamma_{\check{\rho}}(2)$}
\put(37,33){\color{Maroon}$\y'$}
\put(40.5,32.5){\color{Maroon}$\mu'$}
\end{overpic}
\caption{An example illustrating the proof of Theorem~\ref{thm:remainingTargets} for ``many" Newton points $\nu'$.}
\label{fig:lower}
\end{figure}

We first record a lemma concerning the translation part $\eta'$ of the standard representative for $\nu' \in \Conv(\sW (\lambda-2\check{\rho}))$.  This lemma will be used in the proof of Proposition~\ref{prop:lower}.  In Figure~\ref{fig:lower}, the polytope $\Conv(\sW (\lambda-2\check{\rho}))$ is shaded blue.  The Newton point $\nu$ and its standard representative $b_\nu$ are depicted in blue, and the Newton point $\nu'$ and its standard representative $b_{\nu'}$ in purple.

\begin{lemma}\label{lem:etaConv}
Suppose that $\nu'$ is any non-integral Newton point with associated spherical standard parabolic subgroup $P_i$ such that $\nu' \in \Conv(\sW (\lambda - 2\check{\rho}))$. Denote the standard representative with Kottwitz point $\omega$ by $b_{\nu'} = t^{\eta'}s_i$.  Then $\eta' \in \Conv(\sW (\lambda - 2\check{\rho})) \cap (\omega +R^\vee)$. 
\end{lemma}

\begin{proof}
For ease of notation, we put $\xi = \lambda-2\check{\rho}$ for the rest of this proof. Since both $\nu'$ and $\xi$ are elements of the closed dominant Weyl chamber, \cite[Lemma 12.14]{AtiyahBott} says that since $\nu' \in \Conv(\sW \xi)$, we have $\nu'\leq \xi$ in dominance order.  Now by Corollary~\ref{cor:stdRepDominant}, we also have that $\eta'$ is an element of the closed dominant Weyl chamber, with one exception in type $G_2$.  The result for this exception can easily be verified by inspection, so we now assume $\eta'$ is dominant.  Then by \cite[Lemma 12.14]{AtiyahBott} again, it suffices to prove that $\eta' \leq \xi$ as well.  

Since $\nu' \leq \xi$, by definition of the dominance order on $\mathfrak{a}_{\Q}^+$, we have $\xi-\nu' = \sum r_j\alpha_j^\vee$ for some $r_j \in \Q_{\geq 0}$. Formula \eqref{eq:nuProj} for $\nu'$ from Lemma \ref{lem:NPimageAsi} implies that in fact $r_j \in \Z_{\geq 0}$ for all $j \neq i$, while $r_i \in \frac{1}{2}\Z_{> 0}$ is a positive half-integer. Using the formula for $\eta'$ from Proposition \ref{prop:stdReps} in the non-integral case, we then have
\begin{equation*}
    \xi - \eta' = \xi - \left(\nu'+\frac{1}{2}\alpha_i^\vee\right) = \sum_{j=1}^n r_j \alpha_j^\vee - \frac{1}{2}\alpha_i^\vee.
\end{equation*}
Since $r_i \in \frac{1}{2}\Z_{>0}$, the difference above is also a non-negative (integral) combination of simple coroots, and so $\eta'\leq \xi$ in dominance order, as required.
\end{proof}

The next result should be viewed as a delicate analog of Lemma \ref{lem:top}.  This proposition says that there exists a gallery, constructed as the image of applying root operators to $w_0\gamma_{\check{\rho}}^{\sharp}$ as described in Section~\ref{sec:rootOps}, whose final vertex $\zeta'$ has the desired Newton point $\nu'$. We illustrate the proof of Proposition \ref{prop:lower} in Figure~\ref{fig:polytopes}.

\begin{prop}\label{prop:lower}
Let $x_0=t^\lambda w_0\in \eW$, and assume that $\x_0$ is in the shrunken dominant Weyl chamber $\Cfs$. Suppose that $\nu'$ is a non-integral Newton point with associated spherical standard parabolic subgroup $P_i$ of rank 1 such that $\nu' \in \Conv(\sW (\lambda - 2\check{\rho}))$. Denote by $b_{\nu'}=t^{\eta'}s_i$ the standard representative for $\nu'$ such that $\kappa_G(x_0)=\kappa_G(b_{\nu'})$.
	
	Then there exist integers $(c_2,\dots,c_n)$ with $0 \leq c_j \leq M_j'$ for $2 \leq j \leq n$ such that the gallery $\gamma_{\check{\rho}}(c_2,\dots,c_n):\fa \rightsquigarrow \z'$ satisfies 
	\[
	\proj_i(w_0\zeta') = \nu',
	\]
	 where $z' = t^{\zeta'}s_{i_1}$ with $\zeta' = \zeta'(c_2,\dots,c_n)$ as in \eqref{eq:zeta'}.
\end{prop}

\begin{proof} 
We are in search of integers $0 \leq c_j \leq M_j'=\langle \alpha_{i_j}', w_0\zeta\rangle$ such that $\proj_i(w_0\zeta') = \nu'$. Recalling that $\alpha_{i_j}' = -w_0\alpha_{i_j}$, Equation \eqref{eq:zeta'} says $w_0\zeta' = w_0\zeta-\sum\limits_{j=2}^nc_j(\alpha_{i_j}')^\vee$.  Further noting that $\proj_i(\alpha_i^\vee) = 0$, we define the following set of lattice points
	\[
	R^\vee(\zeta) := \left\{ \left. w_0\zeta+ c_1 \alpha_i^\vee-\sum\limits_{j=2}^nc_j(\alpha_{i_j}')^\vee\  \right|\  c_1 \in \Z, \ 0 \leq c_j \leq M_j' \mbox{ for } 2 \leq j \leq n\right\}.
	\]
For an arbitrary element of $R^\vee(\zeta)$, by definition
	\[
	\proj_i\left(w_0\zeta + c_1 \alpha_i^\vee - \sum_{j=2}^n c_j (\alpha'_{i_j})^\vee\right)
	= \proj_i(w_0\zeta').
	\]
It thus suffices to show that there is an element of $R^\vee(\zeta)$ whose image under $\proj_i$ is equal to $\nu'$. 	
By Lemma~\ref{lem:etaConv}, we know that $\eta' \in \Conv(\sW (\lambda - 2\check{\rho})) \cap (\omega + R^\vee)$. In addition, $\proj_i(\eta') = \nu'$ by Proposition \ref{prop:stdReps}. We thus instead aim to show that
\begin{equation}\label{eq:contain}
\Conv(\sW  (\lambda - 2\check{\rho})) \cap (\omega + R^\vee)  \subseteq R^\vee(\zeta),
\end{equation}
from which the result then follows.
We will prove the containment \eqref{eq:contain} by an argument involving the construction of the original gallery $\gamma_{\check{\rho}}$.

\begin{figure}[ht]
\centering
\begin{overpic}[width=0.6\textwidth]{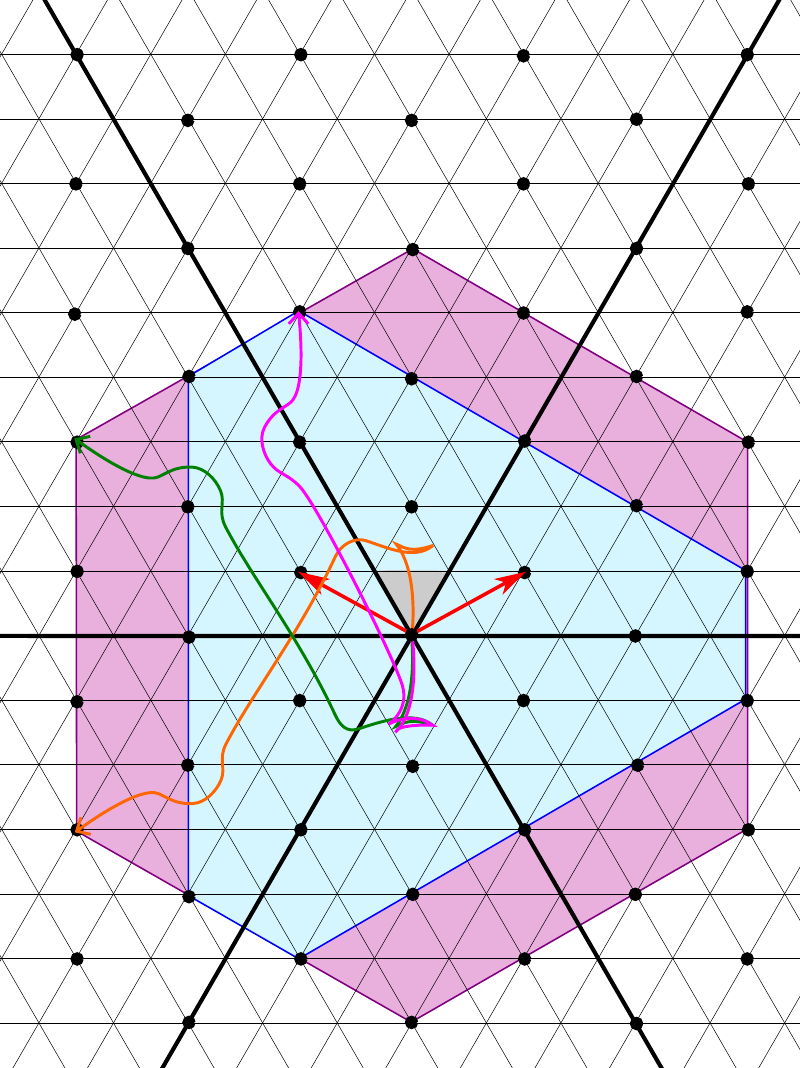}
\put(27,96){\color{blue}$\lambda$}
\put(24,73){\color{blue}$\lambda - 2\check{\rho}$}
\put(38,44){\color{black}$\fa$}
\put(37.5,50){\color{black}$\check{\rho}$}
\put(45,42){\color{red}$\alpha_1^\vee$}
\put(29,42){\color{red}$\alpha_2^\vee$}
\put(4.5,21){\color{OrangeRed}$\zeta$}
\put(3,60){\color{DarkGreen}$w_0\zeta$}
\put(22.5,30){\color{OrangeRed}$\gamma_{\check{\rho}}^\sharp$}
\put(16,45){\color{DarkGreen}$w_0\gamma_{\check{\rho}}^\sharp$}
\put(27,50){\color{Fuchsia}$\gamma'$}
\end{overpic}
\caption{An example illustrating the proof of Proposition~\ref{prop:lower}.}
\label{fig:polytopes}
\end{figure}

Recall that $\gamma_{\check{\rho}}^\sharp$ is the canonical vertex-to-vertex gallery associated to $\gamma_{\check{\rho}}$, with initial vertex the origin and final vertex $\zeta$.  The vertex-to-vertex gallery $w_0\gamma_{\check{\rho}}^\sharp$ thus has final vertex $w_0\zeta$, as depicted in green in Figure \ref{fig:rootOps}. We first claim that in fact
	\begin{equation}\label{eq:w0zetasi}
	    w_0\zeta = s_i(\lambda-2\check{\rho})-\alpha_i^\vee.
	\end{equation}
To see this equality, first recall by \eqref{eq:zeta} that $w_0\zeta = s_i(\lambda-\check{\rho})+w_0\check{\rho} = s_i(\lambda-\check{\rho})-\check{\rho}$. Now rewriting $s_i(\lambda-2\check{\rho})-\alpha_i^\vee = s_i(\lambda-\check{\rho})-s_i(\check{\rho}-\alpha_i^\vee)$, we see that \eqref{eq:w0zetasi} is equivalent to \begin{equation}\label{eq:w0zetasuffices}
    \check{\rho} = s_i(\check{\rho}-\alpha_i^\vee).
\end{equation}
Recall that by definition $\check{\rho}$ is the half-sum of the positive coroots.  Dualizing Corollary 1 in \cite[Ch.~VI $\S$1.6]{Bourbaki4-6}, the simple reflection $s_i$ permutes all the positive coroots apart from $\alpha_i^\vee$ which it sends to $-\alpha_i^\vee$, and so 
\[
s_i(\check{\rho}) = s_i \left(\frac{1}{2}\sum\limits_{\alpha\in \Phi^+} \alpha^\vee \right) = \frac{1}{2}\left(\sum\limits_{\alpha \in \Phi^+\backslash\{\alpha_i\}}\alpha^\vee \right) -\frac{1}{2}\alpha_i^\vee = \check{\rho} - \alpha_i^\vee.
\]
Applying $s_i$ to both sides of the above, \eqref{eq:w0zetasuffices} and thus \eqref{eq:w0zetasi} follows.

We now take the gallery $w_0\gamma_{\check{\rho}}^\sharp$ and introduce the last possible fold in the $(-\check{\rho})$-residue along the hyperplane $H_{\alpha_i, -1}$.  By the same argument as in the proof of Lemma \ref{lem:posFolded}, the resulting gallery $\gamma'$ is positively folded with respect to the standard orientation.  In addition, we compute using \eqref{eq:w0zetasi} that $\gamma'$ has end vertex 
\[
s_{\alpha_i, -1}(w_0 \zeta)= w_0\zeta - \left( \langle \alpha_i ,w_0\zeta \rangle+1\right)\alpha_i^\vee= s_i(\lambda-2\check{\rho})-\langle \alpha_i, s_i(\lambda -2\check{\rho})\rangle\alpha_i^\vee= \lambda-2\check{\rho}.
\]
The gallery $\gamma'$ is shown in pink in Figure \ref{fig:polytopes}.

Recall from the proof of Lemma \ref{lem:indexingM} that the subgallery $\gamma_2$ of $\gamma_{\check{\rho}}^\sharp$ is a minimal gallery from the alcove $t^{\check{\rho}} w\fa$  to $\zeta$ which is entirely contained in the shifted Weyl chamber $t^{\check{\rho}}(\Cu)$, where $u=w_0s_i$.  The image $w_0\gamma_2$ of this subgallery  in $w_0\gamma_{\check{\rho}}^\sharp$ is thus contained in $t^{-\check{\rho}}({\Cu}_{s_i})$, running from the alcove $t^{-\check{\rho}}s_i\fa$ to $w_0\zeta$.  This subgallery $w_0\gamma_2$ is the only portion which differs from $w_0\gamma_{\check{\rho}}^\sharp$ after introducing the fold in $H_{\alpha_i,-1}$ to obtain $\gamma'$.  Since $w_0\gamma_2$ is contained in $t^{-\check{\rho}}({\Cf}_{s_i})$, its image after being reflected across $H_{\alpha_i,-1}$ lies in the shifted dominant Weyl chamber $t^{-\check{\rho}}(\Cf)$, as does the entire gallery $\gamma'$, by the proof of Lemma \ref{lem:indexingM}. (For the general case, simply relabel appropriately according to the action of $\omega$.)

Our plan now is to apply root operators to the gallery $\gamma'$.  Recall that $\Delta = \{ \alpha'_{i_1}, \dots, \alpha'_{i_n}\}$ is our chosen ordering on the basis of simple roots, where $\alpha'_{i_1}=\alpha_i$. Combining the observations from above that $\gamma' \subset t^{-\check{\rho}}(\Cf)$ with Lemma \ref{lem:indexingM}, we have $m(\gamma', \alpha'_{i_j}) =-1$, now for all $1\leq j \leq n$. If we denote by $M_j = \langle \alpha'_{i_j},\lambda-2\check{\rho} \rangle$, then $M_j \geq 1$ for all $1 \leq j \leq n$ by our hypothesis that $\lambda -2\check{\rho}$ is shrunken and dominant, and in particular $m(\gamma', \alpha'_{i_j}) \neq M_j$.  As in Lemma \ref{lem:singleRootOp}, since $\gamma'$ is positively folded with respect to the standard orientation, then for any $1 \leq j \leq n$, the root operator $f^k_{\alpha'_{i_j}}$ is defined on the gallery $\gamma'$ for all $1 \leq k \leq M_j+1$.  By the corresponding analog of Corollary \ref{cor:allRootOps}, if we fix any collection of integers $(d_1, \dots, d_n)$ such that $0 \leq d_j \leq M_j$ for all $1 \leq j \leq n$, we may define the gallery 
\[ 
\sigma' = f_{\alpha'_{i_n}}^{d_n}\left( f_{\alpha'_{i_{n-1}}}^{d_{n-1}}\left( \cdots \left(f_{\alpha'_{i_1}}^{d_1}\left(\gamma'\right)\right) \cdots \right)\right),
\]
which has final vertex
\begin{equation}\label{eq:xisum}
\xi' = \xi'(d_1, \dots, d_n) = \lambda-2\check{\rho} - \sum\limits_{j=1}^n d_j(\alpha'_{i_j})^\vee.
\end{equation}
In particular, since $0 \leq d_j \leq  \langle \alpha'_{i_j}, \lambda-2\check{\rho} \rangle$ for all $1 \leq j \leq n$, then every element of the set $\Conv(\sW(\lambda-2\check{\rho}))\cap (\omega +R^\vee)$ can be expressed as in \eqref{eq:xisum} for some choice of $(d_1, \dots, d_n)$. To complete the proof of \eqref{eq:contain}, it thus suffices to show that $\xi'(d_1, \dots, d_n) \in R^\vee(\zeta)$ for any collection of integers $0\leq d_j \leq M_j$.   The blue polytope in Figure \ref{fig:polytopes} shows $\Conv(\sW(\lambda-2\check{\rho}))$, and the set of all vertices $\xi'(d_1,d_2)$ for $0 \leq d_1 \leq 1$ and $0 \leq d_2 \leq 4$ are indicated by black dots in this blue polytope. A portion of the set $R^\vee(\zeta)$ is shaded purple in Figure \ref{fig:polytopes}; note that indeed the blue polytope corresponding to $\Conv(\sW(\lambda-2\check{\rho}))$ is contained in the purple subset of $R^\vee(\zeta)$.  

Given any $\xi'  = \lambda-2\check{\rho} - \sum\limits_{j=1}^n d_j(\alpha'_{i_j})^\vee$ as in \eqref{eq:xisum}, we first rewrite $\xi'$ as
\[
\xi' = \lambda-2\check{\rho} -(\langle \alpha_i, \lambda-2\check{\rho} \rangle+1)\alpha_i^\vee + (\langle \alpha_i, \lambda-2\check{\rho} \rangle+1)\alpha_i^\vee - \sum\limits_{j=1}^n d_j(\alpha'_{i_j})^\vee.
\]
Next, recall \eqref{eq:w0zetasi}, and compute that 
\begin{equation}\label{eq:w0zetasirewrite}
w_0\zeta = s_i(\lambda-2\check{\rho})-\alpha_i^\vee =\lambda-2\check{\rho} -(\langle \alpha_i, \lambda-2\check{\rho} \rangle+1)\alpha_i^\vee.
\end{equation}
Therefore, we have 
\[ 
\xi' = w_0\zeta + (\langle \alpha_i, \lambda-2\check{\rho} \rangle+1)\alpha_i^\vee - \sum\limits_{j=1}^n d_j(\alpha'_{i_j})^\vee =  w_0\zeta + (M_1-d_1+1)\alpha_i^\vee - \sum\limits_{j=2}^n d_j(\alpha'_{i_j})^\vee.
\]
Since $M_1-d_1+1 \in \Z$, we may define $c_1 = M_1-d_1+1$.  Our final claim is that we may also choose to set $c_j = d_j$ for $2 \leq j \leq n$.  Since all $d_j \geq 0$, we need only show that $d_j \leq M_j'$ in order that $\xi' \in R^\vee(\zeta)$.  Since $d_j \leq M_j$, it suffices to show that $M_j \leq M_j'$ for all $2\leq j \leq n$.  To this end, compute using \eqref{eq:w0zetasirewrite} that
\[ 
M_j' = \langle \alpha_{i_j}',w_0\zeta \rangle = \langle \alpha_{i_j}', \lambda-2\check{\rho} -(\langle \alpha_i, \lambda-2\check{\rho} \rangle+1)\alpha_i^\vee \rangle = M_j - (\langle \alpha_i, \lambda-2\check{\rho}\rangle +1)\langle \alpha'_{i_j}, \alpha_i^\vee \rangle.
\]
Since for $2 \leq j \leq n$, by definition $\alpha'_{i_j} \neq \alpha_i$, the value of $\langle \alpha'_{i_j}, \alpha_i^\vee\rangle$ is an off-diagonal entry of the Cartan matrix, and is thus nonpositive.  In addition, the coefficient $\langle \alpha_i, \lambda-2\check{\rho} \rangle +1 >0$ by our shrunken hypothesis on $\lambda$.  Therefore, the above calculation shows that indeed $M_j \leq M'_j$ for all $2 \leq j \leq n$, as required to verify \eqref{eq:contain} and complete the proof.
\end{proof}

We now prove the corresponding analog of Corollary \ref{cor:firstConjugation}, in which we identify $z'=t^{\zeta'}s_{i_1}$ as a particular $\eW$-conjugate of the standard representative $b_{\nu'}=t^{\eta'}s_i$.

\begin{lemma}\label{lem:conjugation'}  
Let $z' = t^{\zeta'}s_{i_1}$ be as in Proposition \ref{prop:lower}. Then:
 \begin{enumerate}
     \item there exists an integer $d_i' \in \Z$ such that $s_i w_0 \zeta' - \eta' = d_i' \alpha_i^\vee$; and 
     \item if $y' = t^{\mu'} w_0 s_i \in \eW$ with $\mu' = d_i' w_0 s_i \eta' \in Q^\vee$  for $d_i' \in \Z$ as in (1), then $z'=b_{\nu'}^{y'}$.
 \end{enumerate}
\end{lemma}

\noindent For example, in Figure~\ref{fig:lower}, we have $s_i w_0 \zeta' = \eta'$ and so $d_i' = 0$, $\mu' = 0$, and $y' = w_0 s_i = s_1 s_2$, with $\mu'$ and the alcove $\y'$ shown in maroon. 

\begin{proof}  By Proposition \ref{prop:lower} and Proposition \ref{prop:stdReps}, we have $\proj_i(w_0 \zeta') = \nu' = \proj_i(\eta')$.  Thus, by an identical argument as in the proof of Lemma~\ref{lem:firstTransverse}, we obtain $\zeta' \in (w_0 s_i \cT_{\nu'}) \cap (\omega +R^\vee)$. We now apply Proposition~\ref{prop:inverseRank1} with $u = w_0 s_i$, as in the proof of Corollary~\ref{cor:firstConjugation}. 
\end{proof}

Our next corollary combines the results of this section to prove Theorem~\ref{thm:remainingTargets} for ``many" of the Newton points $\nu' \in \Conv(\sW(\lambda-2\check{\rho}))$ with associated parabolic subgroup $P_i$.

\begin{corollary}\label{cor:remainingTargets}
Let $x_0=t^\lambda w_0\in \eW$, and assume that $\x_0$ is in the shrunken dominant Weyl chamber $\Cfs$. Suppose that 
\[\nu' \in \Conv(\sW  (\lambda - 2\check{\rho})),\]
where $\nu'$ is a non-integral Newton point with associated spherical standard parabolic subgroup $P_i$ of rank 1. Denote by $b_{\nu'}=t^{\eta'}s_i$ the standard representative for $\nu'$ such that $\kappa_G(x_0)=\kappa_G(b_{\nu'})$. 

Let $\mu' = d_i'w_0s_i\eta'$ and $y' = t^{\mu'}w_0s_i \in \eW$ be as in the statement of Lemma \ref{lem:conjugation'}. If the integer $d_i' \geq 1$, then the gallery $\gamma_{\check{\rho}}(c_2,\dots,c_n): \fa \rightsquigarrow \bb_{\nu'}^{y'}$ has type $\vec{x}_0$ and is positively folded with respect to the $(P_i,y')$-chimney.
\end{corollary}

\begin{proof} 
By construction, the gallery $\gamma_{\check{\rho}}(c_2, \dots, c_n):\fa \rightsquigarrow \z'$ has type $\vec{x}_0$. Lemma \ref{lem:conjugation'} says that $\z' = z'\fa_0$, where $z'=b_{\nu'}^{y'}$, and so $\bb_{\nu'}^{y'} = b_{\nu'}^{y'}\fa_0$ is indeed the final alcove of $\gamma_{\check{\rho}}(c_2, \dots, c_n)$.  In order for $\gamma_{\check{\rho}}(c_2, \dots, c_n)$ to be positively folded with respect to the $(P_i,y')$-chimney, by  Corollary~\ref{cor:posFolded}, it suffices to show that $\langle \alpha_{i_1}, \mu' \rangle \geq 1$.  Since $\eta' \in H_{\alpha_i,1}$ by Proposition \ref{prop:stdReps}, we have 
\begin{equation}\label{eq:dieq}
\langle \alpha_{i_1}, \mu' \rangle  = \langle w_0s_i\alpha_i, d_i'w_0s_i\eta' \rangle = d_i'\langle \alpha_i, \eta \rangle = d_i'.
\end{equation}
Therefore, the hypothesis $d_i'\geq 1$ automatically implies that $\langle \alpha_{i_1}, \mu' \rangle \geq 1$, as required.
\end{proof}

Although Corollary \ref{cor:remainingTargets} does produce ``many'' of the galleries required to prove Theorem \ref{thm:remainingTargets}, not all galleries obtained by the construction in this section satisfy the inequality required to apply Corollary \ref{cor:remainingTargets}. In Figure~\ref{fig:lower}, for example, we have $s_iw_0\zeta' = \eta'$, and so $d_i'=0$ and $\mu'=0$.  A sector representing the $(P_i,y')$-chimney is shaded pink, and it can be seen that the gallery $\gamma_{\check{\rho}}(2):\fa \rightsquigarrow \z'$ is not positively folded with respect to this chimney, since the first fold is not positive.  The modification required to complete the proof of Theorem~\ref{thm:remainingTargets} in the cases where the integer $d_i' \leq 0$ is the subject of the next section.

\subsection{Reaching the remaining targets}\label{sec:lowest}

In this final section, we complete the proof of Theorem~\ref{thm:remainingTargets} by making a small modification to the gallery $\gamma_{\check{\rho}}(c_2, \dots, c_n)$ produced by Proposition \ref{prop:lower}, in the remaining cases where the inequality $d_i'\geq 1$ required to apply Corollary \ref{cor:remainingTargets} is not met.  We continue the example from Figures~\ref{fig:gamma}--\ref{fig:lower} in Figure~\ref{fig:lowest}. 

\begin{figure}[ht]
\centering
\begin{overpic}[width=0.6\textwidth]{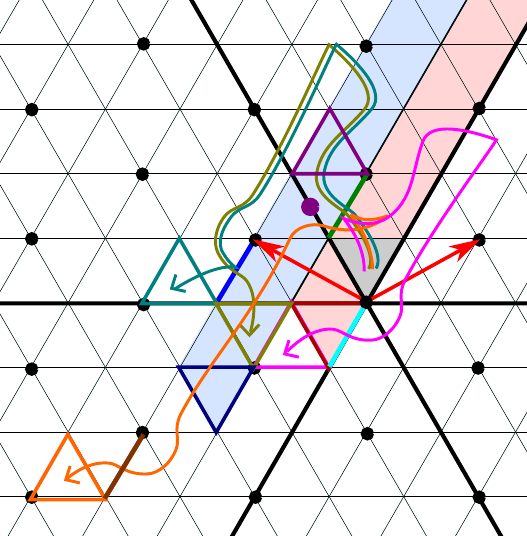}
\put(53,60){\color{Purple}$\nu'$}
\put(63,69){\color{Purple}$b_{\nu'}$}
\put(67.5,71){\color{Purple}$\eta'$}
\put(67.5,48){\color{black}$\fa$}
\put(64,64){\color{DarkGreen}$p$}
\put(65,34){\color{Aqua}$p_i(\gamma_{\check{\rho}}(2))$}
\put(43,46){\color{blue}$p'_i(\gamma_{\check{\rho}}(2))$}
\put(22,9){\color{SaddleBrown}$p_i(\gamma_{\check{\rho}})$}
\put(68,63){\color{black}$\check{\rho}$}
\put(83,48){\color{red}$\alpha_1^\vee$}
\put(54,52){\color{red}$\alpha_2^\vee$}
\put(5,3){\color{OrangeRed}$\zeta$}
\put(12,14){\color{OrangeRed}$\z$}
\put(34,16){\color{OrangeRed}$\gamma_{\check{\rho}}$}
\put(46,27){\color{DeepPink}$\zeta'=\mu''$}
\put(53,37){\color{DeepPink}$\z'$}
\put(86,62){\color{DeepPink}$\gamma_{\check{\rho}}(2)$}
\put(60,40){\color{Maroon}$\y'$}
\put(67,38){\color{Maroon}$\mu'$}
\put(25,38){\color{Teal}$\zeta''$}
\put(32,49){\color{Teal}$\z''$}
\put(39,60){\color{Teal}$\gamma''$}
\put(39,26){\color{DarkBlue}$\y''$}
\put(48,82){\color{Olive}$\gamma_{\check{\rho}}^+(2)$}
\end{overpic}
\caption{An example illustrating the proof of Theorem~\ref{thm:remainingTargets} for the remaining Newton points $\nu'$.}
\label{fig:lowest}
\end{figure}

As in Corollary \ref{cor:remainingTargets}, suppose that $\nu' \in \Conv(\sW(\lambda-2\check{\rho}))$, where $\nu'$ is a non-integral Newton point with standard representative $b_{\nu'}=t^{\eta'}s_i$.  The starting point for our modified construction will be the gallery $\gamma_{\check{\rho}}(c_2, \dots, c_n): \fa \rightsquigarrow \z'$  from Corollary~\ref{cor:remainingTargets}.  In particular, $z'=t^{\zeta'}s_{i_1}=b_{\nu'}^{y'}$, and the integer $d_i' \in \Z$ is determined by $s_iw_0\zeta'-\eta' = d_i'\alpha_i^\vee$ with all notation as in Lemma \ref{lem:conjugation'}. Throughout this subsection, we assume that $d_i' \leq 0$, since otherwise Corollary \ref{cor:remainingTargets} applies. In Figure \ref{fig:lowest}, the gallery $\gamma_{\check{\rho}}(2)$ is shown in pink.  In this example, since $\eta'=s_iw_0\zeta'$ as indicated in purple, then $d_i' = 0$, and so $\mu' = 0$ and $y' = w_0 s_i$ are shown in maroon.

We  will now modify the gallery $\gamma_{\check{\rho}}(c_2,\dots,c_n)$ to obtain a new gallery $\gamma'':\fa \rightsquigarrow \z''$ of the same type, which is positively folded with respect to the $(P_i,y'')$-chimney for an element $y'' \in \eW$ such that $z''=b_{\nu'}^{y''}$.  We first consider the final alcove $\z''$, which will just be a translate of $\z'$ by a particular multiple of $\alpha_{i_1}^\vee$.  The following is the required variation on Lemma \ref{lem:conjugation'} in case $d_i'\leq 0$.

\begin{lemma}\label{lem:z''}  Suppose that $d_i'\leq 0$, and let $z'' = t^{\zeta''} s_{i_1}$ where $\zeta'' = \zeta' + (-2d_i'+1)\alpha_{i_1}^\vee$. If $y'' = t^{\mu''}w_0 s_i \in \eW$ with $\mu'' = (-d'_i+1) w_0 s_i \eta' \in Q^\vee$, then $z'' = b_{\nu'}^{y''}$ and $\langle \alpha_{i_1},\mu''\rangle \geq 1$.
\end{lemma}

\noindent In Figure~\ref{fig:lowest}, in which we recall that $d_i'=0$, we have $\mu'' = w_0s_i\eta' = \zeta'$ so that $y'' = t^{\zeta'}w_0 s_i = t^{\zeta'}s_1 s_2$. We depict $\zeta''=\zeta'+\alpha_{i_1}^\vee$ and $\z'' = \z'+\alpha_{i_1}^\vee$ in teal, and   $\y''$ is shown in dark blue.

\begin{proof}
Since $\alpha_{i_1}^\vee = w_0 s_i \alpha_i^\vee$, the equality $\zeta'' = \zeta' + (-2d_i'+1)\alpha_{i_1}^\vee$ implies that $s_iw_0(\zeta'' - \zeta') \in \Ker A_{s_i}$.  Therefore, $s_iw_0(\zeta''-\zeta')+\nu' \in \cT_{\nu'}$  where we recall that $\cT_{\nu'} = \nu'+\Ker A_{s_i}$. From the proof of Lemma~\ref{lem:conjugation'}, we have $\zeta' \in (w_0 s_i \cT_{\nu'}) \cap (\omega +R^\vee)$, and so it follows that $\zeta'' \in (w_0 s_i \cT_{\nu'}) \cap (\omega +R^\vee)$.  Now by Proposition~\ref{prop:inverseRank1}(1) with $u=w_0s_i$, we know that $s_iw_0\zeta'' - \eta' = d_i'' \alpha_i^\vee$ for some integer $d_i'' \in \Z$.   Using the definition $\zeta'' = \zeta' + (-2d_i'+1)\alpha_{i_1}^\vee$  and the relation $s_iw_0\zeta'-\eta'=d_i'\alpha_i^\vee$ from Lemma \ref{lem:conjugation'}, a direct calculation then shows that in fact $d_i'' = -d_i'+1$. The fact that $z''=b_{\nu'}^{y''}$ then holds by Proposition~\ref{prop:inverseRank1}(2) with our choice of $\mu''= d_i''w_0s_i\eta'=(-d_i'+1)w_0s_i\eta'$.
Finally, compute using $\eta' \in H_{\alpha_i,1}$ by Proposition~\ref{prop:stdReps} that
\[
\langle \alpha_{i_1},\mu''\rangle = \langle w_0 s_i \alpha_i,(-d_i'+1) w_0 s_i \eta' \rangle = (-d'_i+1) \langle \alpha_i, \eta' \rangle = -d'_i+1 \geq 1,
\]
since $d_i'\leq 0$ by hypothesis.
\end{proof}

Fix $z''$ and $y''$ as in the statement of Lemma~\ref{lem:z''}. We shall now construct a gallery $\gamma'':\fa \rightsquigarrow \z''$ of the same type as $\gamma_{\check{\rho}}(c_2,\dots,c_n)$, which is positively folded with respect to the $(P_i,y'')$-chimney.  We will do this by first ``undoing" a fold of $\gamma_{\check{\rho}}(c_2,\dots,c_n)$, and then introducing one additional fold.  In Figure~\ref{fig:lowest}, a sector representing the $(P_i,y'')$-chimney is shaded light blue, and the final modified gallery $\gamma'':\fa \rightsquigarrow \z''$ is shown in  teal.

By Proposition~\ref{prop:folds} and the construction of the original folded gallery $\gamma_{\check{\rho}}$, all folds of $\gamma_{\check{\rho}}(c_2,\dots,c_n)$ which lie in an $\alpha_{i_1}$-hyperplane occur in $H_{\alpha_{i_1},1}$. 
More specifically, these folds are in the panel of the alcove $t^{\check{\rho}} w_0\fa$,  which is contained in $H_{\alpha_{i_1},1}$.  
Let $p$ be the last panel of  $\gamma_{\check{\rho}}(c_2,\dots,c_n)$ in which there is fold in $H_{\alpha_{i_1},1}$, and let $\gamma_{\check{\rho}}^+(c_2,\dots,c_n)$ be the gallery obtained from $\gamma_{\check{\rho}}(c_2,\dots,c_n)$ by applying the reflection $s_{\alpha_{i_1},1}$ to the subgallery of $\gamma_{\check{\rho}}(c_2,\dots,c_n)$ which starts at $p$.  In Figure~\ref{fig:lowest}, the panel $p$ is dark green, and the gallery $\gamma_{\check{\rho}}^+(2)$, in which we have ``undone'' this fold at $p$, is colored olive green.  Further, note that $\gamma_{\check{\rho}}^+(2)$ is positively folded with respect to the $(P_i,y'')$-chimney, for which a sector is shaded light blue.
Lemma \ref{lem:gamma_rho+} formalizes these observations about the gallery $\gamma_{\check{\rho}}^+(c_2, \dots, c_n)$, using arguments similar to the proof of Lemma \ref{lem:posFolded}.

\begin{lemma}\label{lem:gamma_rho+}
The gallery $\gamma_{\check{\rho}}^+(c_2,\dots,c_n): \fa \rightsquigarrow s_{\alpha_{i_1},1}\z'$ is a gallery of the same type as $\gamma_{\check{\rho}}(c_2,\dots,c_n)$, and $\gamma_{\check{\rho}}^+(c_2,\dots,c_n)$ is positively folded with respect to the $(P_i,y'')$-chimney.
\end{lemma}

\begin{proof}
The reflection $s_{\alpha_{i_1},1}$ is type-preserving, so the galleries $\gamma_{\check{\rho}}^+(c_2,\dots,c_n)$ and $\gamma_{\check{\rho}}(c_2,\dots,c_n)$ have the same type.  By construction, the three galleries $\gamma_{\check{\rho}}$, $\gamma_{\check{\rho}}(c_2,\dots,c_n)$, and $\gamma_{\check{\rho}}^+(c_2,\dots,c_n)$ all begin with a minimal gallery from $\fa$ to $t^{\check{\rho}} w_0\fa$, and so $\gamma_{\check{\rho}}^+(c_2,\dots,c_n)$ also has first alcove $\fa$.  Its final alcove is that obtained by applying the reflection $s_{\alpha_{i_1},1}$ to the final alcove of $\gamma_{\check{\rho}}(c_2,\dots,c_n)$, and so its final alcove is indeed $s_{\alpha_{i_1},1}\z'$.  

It remains to show that $\gamma_{\check{\rho}}^+(c_2, \dots, c_n)$ is positively folded with respect to the $(P_i,y'')$-chimney. By Lemma~\ref{lem:z''}, we have $\langle \alpha_{i_1},\mu''\rangle \geq 1$. By Corollary~\ref{cor:posFolded}, the gallery $\gamma_{\check{\rho}}(c_2,\dots,c_n)$ is thus positively folded with respect to the $(P_i,y'')$-chimney.  We now just need to verify that the portion of $\gamma_{\check{\rho}}^+(c_2,\dots,c_n)$ which differs from  $\gamma_{\check{\rho}}(c_2,\dots,c_n)$ remains positively folded with respect to the $(P_i,y'')$-chimney. Since by definition $p$ is the last panel of $\gamma_{\check{\rho}}(c_2,\dots,c_n)$ which has a fold in an $\alpha_{i_1}$-hyperplane,  it suffices to show that the reflection $s_{\alpha_{i_1},1}$ takes positive folds in $\beta$-hyperplanes, where $\beta \in \Phi^+ \setminus \{ \alpha_{i_1} \}$, to positive folds.  Now all such folds are positive with respect to the opposite standard orientation, and since $\beta \in \Phi^+ \setminus \{ \alpha_{i_1} \}$ we have $s_{\alpha_{i_1}}(\beta) > 0$.  The result follows.
\end{proof}

Although the gallery $\gamma_{\check{\rho}}^+(c_2, \dots, c_n)$ is positively folded with respect to the $(P_i,y'')$-chimney by Lemma \ref{lem:gamma_rho+}, its final alcove $s_{\alpha_{i_1},1}\z'$ is no longer a conjugate of the standard representative $b_{\nu'}$. To obtain the final gallery $\gamma''$, we introduce one final fold, again in an $\alpha_{i_1}$-hyperplane, though of higher index than the original fold at $p$.  However, in order to ensure that this additional fold will be both possible and positive, we require the following lemma concerning $\gamma_{\check{\rho}}^+(c_2,\dots,c_n)$.  In our example, in which $i_1=2$ and $d_i'=0$, this lemma observes that the final crossing of the hyperplane $H_{\alpha_{i_1},2}$ in the olive green gallery $\gamma_{\check{\rho}}^+(2)$ from Figure \ref{fig:lowest} moves toward the antidominant chamber, which puts the new final fold on the positive side of the chimney represented by $\y''$.

\begin{lemma}\label{lem:crossing2}
The portion of the gallery $\gamma_{\check{\rho}}^+(c_2,\dots,c_n)$ after its last fold crosses the hyperplane $H_{\alpha_{i_1},-2d_i'+2}$ from the dominant side to the antidominant side.
\end{lemma}

\begin{proof} We first claim that the portion of $\gamma_{\check{\rho}}^+(c_2,\dots,c_n)$ after its last fold is obtained by applying the reflection $s_{\alpha_{i_1},1}$ to the portion of $\gamma_{\check{\rho}}(c_2,\dots,c_n)$ after its last fold.  By construction, to verify this claim it suffices to prove that $\gamma_{\check{\rho}}(c_2,\dots,c_n)$ has a fold after panel $p \subset H_{\alpha_{i_1},1}$; that is, there exists a fold after its last fold in an $\alpha_{i_1}$-hyperplane. From the proof of Proposition~\ref{prop:folds}, the last $n-1$ folds of  $\gamma_{\check{\rho}}(c_2,\dots,c_n)$ are all in hyperplanes of types other than $\alpha_{i_1}$, so this first claim holds since $n\geq 2$ by hypothesis.

Recall that $p_i(\gamma,\lambda - \check{\rho})$ denotes the panel of the original gallery $\gamma$ at which it crosses the $\alpha_i$-hyperplane passing through $\lambda - \check{\rho}$.  Write $p_i(\gamma_{\check{\rho}})$ for the image in $\gamma_{\check{\rho}}$ of the panel $p_i(\gamma,\lambda - \check{\rho})$, and recall from Lemma~\ref{lem:tail_gammaRho} that the subgallery of $\gamma_{\check{\rho}}$ after all its folds crosses $p_i(\gamma_{\check{\rho}}) \subset H_{\alpha_{i_1},k}$ from the antidominant to the dominant side, where $k = \langle \alpha_{i_1}, \zeta \rangle - 1$.  In Figure~\ref{fig:lowest}, the panel $p_i(\gamma_{\check{\rho}})$ is shown in brown.

By construction, the image of the panel $p_i(\gamma_{\check{\rho}})$ in $\gamma_{\check{\rho}}(c_2, \dots, c_n)$ is then simply a translation of $p_i(\gamma_{\check{\rho}})$ by $\sum\limits_{j=2}^n c_j\alpha_{i_j}^\vee$.  We denote this panel of $\gamma_{\check{\rho}}(c_2, \dots, c_n)$ by $p_i(\gamma_{\check{\rho}}(c_2, \dots, c_n))$.  Since only a translation was applied to obtain $p_i(\gamma_{\check{\rho}}(c_2, \dots, c_n))$ from $p_i(\gamma_{\check{\rho}})$, then the crossing of $\gamma_{\check{\rho}}(c_2, \dots, c_n)$ at the panel $p_i(\gamma_{\check{\rho}}(c_2, \dots, c_n))$ is also from the antidominant to the dominant side.  Further, since $p_i(\gamma_{\check{\rho}}) \subset H_{\alpha_{i_1},\langle \alpha_{i_1}, \zeta \rangle-1}$ and $\zeta'=\zeta + \sum\limits_{j=2}^n c_j \alpha_{i_j}^\vee$, then this panel $p_i(\gamma_{\check{\rho}}(c_2, \dots, c_n)) \subset H_{\alpha_{i_1},\langle \alpha_{i_1}, \zeta' \rangle-1}$. For example, the panel $p_i(\gamma_{\check{\rho}}(2))$ is highlighted in aqua in Figure \ref{fig:lowest}, in which $\langle \alpha_2, \zeta'\rangle-1=0$.  Indeed, $p_i(\gamma_{\check{\rho}}(2))\subset H_{\alpha_2,0}$, and the pink gallery $\gamma_{\check{\rho}}(2)$ crosses this panel from the antidominant to the dominant side.

Now consider the image  $p'_i(\gamma_{\check{\rho}}(c_2, \dots, c_n))$ of the panel $p_i(\gamma_{\check{\rho}}(c_2, \dots, c_n))$ under the reflection $s_{\alpha_{i_1},1}$.  Since this panel lies in an $\alpha_{i_1}$-hyperplane, then the reflection $s_{\alpha_{i_1},1}$ reverses the orientation of the crossing, so that in the reflected image the crossing will be from the dominant side to the antidominant side.  The reflected panel $p'_i(\gamma_{\check{\rho}}(2))$ is shown in blue in Figure \ref{fig:lowest}, and the corresponding crossing in the olive gallery $\gamma_{\check{\rho}}^+(2)$ is toward the antidominant side.

Finally, we claim that the panel $p'_i(\gamma_{\check{\rho}}(c_2, \dots, c_n))$ in the gallery $\gamma_{\check{\rho}}^+(c_2, \dots, c_n)$ is contained in the hyperplane $H_{\alpha_{i_1},-2d_i'+2}$.  Recall from Lemma \ref{lem:conjugation'} that $s_iw_0\zeta' = \eta'+d_i'\alpha_i^\vee$, and compute that
\[ 
\langle \alpha_{i_1}, \zeta' \rangle = \langle \alpha_i, s_iw_0\zeta' \rangle = \langle \alpha_i, \eta'+d_i'\alpha_i^\vee \rangle = 1+2d_i',
\]
where we have used $\eta'\in H_{\alpha_i,1}$ by Proposition \ref{prop:stdReps}. Therefore, applying $s_{\alpha_{i_1},1}=t^{\alpha_{i_1}^\vee}s_{\alpha_{i_1}}$ to the hyperplane $H_{\alpha_{i_1},\langle \alpha_{i_1}, \zeta' \rangle-1}$ containing $p_i(\gamma_{\check{\rho}}(c_2, \dots, c_n))$ gives us another $\alpha_{i_1}$-hyperplane, now having index
\[ 
-\left(\langle \alpha_{i_1}, \zeta' \rangle-1\right) +2 = -(1+2d_i' -1)+2 = -2d_i'+2,
\]
as claimed.  In our example, $d_i'=0$ and indeed the blue panel $p_i'(\gamma_{\check{\rho}}(2))$ of the gallery $\gamma''$ is contained in $H_{\alpha_2,2}$.
\end{proof}

Finally, let $\gamma''$ be the gallery obtained from $\gamma_{\check{\rho}}^+(c_2,\dots,c_n)$ by carrying out the fold in the hyperplane $H_{\alpha_{i_1},-2d_i'+2}$.  Our final corollary says that introducing this final fold preserves the fact that it is positively folded with respect to the $(P_i,y'')$-chimney.  In addition, the new end alcove $\z''$ is once again a $\eW$-conjugate of the standard representative for $\nu'$. In Figure~\ref{fig:lowest}, this final gallery $\gamma''$ and its end alcove $\z''=\z'+\alpha_{i_1}^\vee$ are shown in teal, and $\gamma''$ is indeed positively folded with respect to the $(P_i,y'')$-chimney, a sector for which is shaded light blue.

\begin{corollary}\label{cor:gamma''}
Let $x_0=t^\lambda w_0\in \eW$, and assume that $\x_0$ is in the shrunken dominant Weyl chamber $\Cfs$.  Suppose that 
\[\nu' \in \Conv(\sW  (\lambda - 2\check{\rho})),\]
where $\nu'$ is a non-integral Newton point with associated spherical standard parabolic subgroup $P_i$ of rank 1. Denote by $b_{\nu'}=t^{\eta'}s_i$ the standard representative for $\nu'$ such that $\kappa_G(x_0)=\kappa_G(b_{\nu'})$. 

Let $d_i'\in \Z$ be as in Lemma \ref{lem:conjugation'}(1) and $y'' \in \eW$ as in Lemma \ref{lem:z''}. If the integer $d_i' \leq 0$, then the gallery $\gamma'': \fa \rightsquigarrow  \bb_{\nu'}^{y''}$ has type $\vec{x}_0$ and is positively folded with respect to the $(P_i,y'')$-chimney.
\end{corollary}

\begin{proof} Since $d_i'\leq 0$, we may construct the gallery $\gamma''$ by modifying the gallery $\gamma_{\check{\rho}}(c_2, \dots, c_n)$ as explained in this section. The galleries $\gamma''$ and $\gamma_{\check{\rho}}^+(c_2,\dots,c_n)$ differ only by the fold introduced at the hyperplane $H_{\alpha_{i_1},-2d_i'+2}$.  The reflection $s_{\alpha_{i_1},-2d_i'+2}$ preserves type, and so $\gamma''$ has the same type as $\gamma_{\check{\rho}}^+(c_2, \dots, c_n)$, and thus also as $\gamma_{\check{\rho}}(c_2, \dots, c_n)$ by Lemma \ref{lem:gamma_rho+}.  Then by Corollary \ref{cor:remainingTargets}, the type of $\gamma_{\check{\rho}}(c_2, \dots, c_n)$, and thus also of $\gamma''$, is $\vec{x}_0$.

The final alcove of $\gamma''$ equals $s_{\alpha_{i_1},-2d_i'+2} \left( s_{\alpha_{i_1},1}\z'\right)$ by construction. Given any $\alpha \in \Phi^+$, we can express the reflection across the affine hyperplane $H_{\alpha,k}$ as $s_{\alpha,k} = t^{k\alpha^\vee}s_\alpha$.  Therefore,  
\[ s_{\alpha_{i_1},-2d_i'+2} \cdot s_{\alpha_{i_1},1} = \left(t^{(-2d_i'+2)\alpha_{i_1}^\vee}s_{\alpha_{i_1}}\right) t^{\alpha_{i_1}^\vee}s_{\alpha_{i_1}} = t^{(-2d_i'+1)\alpha_{i_1}^\vee}, \]
in which case the final alcove equals $t^{(-2d_i'+1)\alpha_{i_1}^\vee}\z' = \z''$, where indeed $z''=t^{\zeta''}s_{i_1}$ with $\zeta''=\zeta'+(-2d_i'+1)\alpha_{i_1}^\vee$. Lemma \ref{lem:z''} then says that $z''=b_{\nu'}^{y''}$, and so the final alcove of $\gamma''$ is indeed $\bb_{\nu'}^{y''}$.

By Lemma \ref{lem:gamma_rho+}, the gallery $\gamma_{\check{\rho}}^+(c_2,\dots,c_n)$ is positively folded with respect to the $(P_i,y'')$-chimney.  We need only to verify that the portion of $\gamma''$ which differs from  $\gamma_{\check{\rho}}^+(c_2,\dots,c_n)$ remains positively folded with respect to the $(P_i,y'')$-chimney. Since the galleries $\gamma''$ and $\gamma_{\check{\rho}}^+(c_2,\dots,c_n)$ differ only by the new fold introduced by applying $s_{\alpha_{i_1},-2d_i'+2}$, it thus remains only to show that this new fold at $H_{\alpha_{i_1},-2d_i'+2}$ is positive with respect to the $(P_i,y'')$-chimney.

Recall from the proof of Lemma \ref{lem:posFolded} that any $(P_i, y'')$-sector lies between $H_{\alpha_{i_1},k}$ and $H_{\alpha_{i_1},k+1}$ where $k=\langle \alpha_{i_1}, \mu'' \rangle$. By Lemma \ref{lem:crossing2}, this final fold at $H_{\alpha_{i_1},-2d_i'+2}$ occurs on the dominant side of the hyperplane. 
Since $d_i'\leq 0$ by hypothesis, as in the proof of Lemma \ref{lem:z''} we have $-2d_i'+2>-d_i'+1 = \langle \alpha_{i_1},\mu''\rangle \geq 1$ so that this final fold is indeed positive. The gallery $\gamma''$ is thus positively folded with respect to the $(P_i,y'')$-chimney.
\end{proof}

We are now prepared to complete the proof of the main result of this section, after which we can immediately prove our main theorem.

\begin{proof}[Proof of Theorem~\ref{thm:remainingTargets}] Combine Corollary~\ref{cor:remainingTargets} and Corollary \ref{cor:gamma''} with Theorem~\ref{thm:ADLVChimneys}(1). 
\end{proof}

\begin{proof}[Proof of Theorem~\ref{thm:w0ShrunkenDominant}]
 Theorem \ref{thm:translations} proves case (3a), as well as case (3b) when $\nu_b$ is integral.  Combine Theorems \ref{thm:firstTarget} and \ref{thm:remainingTargets} to complete case (3b) when $\nu_b$ is non-integral. The construction for the basic case (3c), appears as Theorem A in \cite{MST1}.
\end{proof}

\renewcommand{\refname}{Bibliography}
\bibliography{bibliography}
\bibliographystyle{alpha}

\end{document}